\documentclass[10pt]{article}

\usepackage{a4wide}
\usepackage{amssymb}
\usepackage{amsfonts}
\usepackage{amsmath}
\input xy
\xyoption{arrow} \xyoption{matrix}

\usepackage{amsthm}
\usepackage{multirow}
\usepackage{marginnote}
\usepackage{a4wide}
\usepackage{amssymb}
\usepackage{amsfonts}
\usepackage{amsmath}
\usepackage{mathrsfs}
\usepackage{tikz}
\usepackage{tikz-cd}
\usetikzlibrary{arrows,matrix}
\usetikzlibrary{positioning}

\date{}

\newtheorem{proposition}{Proposition}[section]
\newtheorem{theorem}[proposition]{Theorem}
\newtheorem{lemma}[proposition]{Lemma}

\newtheorem{definition}[proposition]{Definition}
\newtheorem{corollary}[proposition]{Corollary}

\def\der{\partial }

\def\nFM0{{\nu }_{F,M_0}}
\def\nFN0{{\nu }_{F,N_0}}
\def\nGN0{{\nu }_{G,N_0}}

\def\N0{ {\bf N}_0 }

\def\t{\otimes}
\def\g{\gamma}

\def\ra{\rightarrow}

\def\Xpm{X^{\pm }}

\def\s{\sigma}
\def\Z{\mathbb{Z}}

\def\l1{{\lambda}_1}

\def\a{\alpha}
\def\a0{ {\alpha }_0}
\def\a1{ {\alpha }_1}

\def\l{\lambda}
\def\o{\omega}
\def\nFGM0{{\nu }_{F,G,M_0}}

\def\nFN0{{\nu}_{F,N_0}}

\def\sm{{\sigma}^m}

\def\sm1{{\sigma}^{-1}}

\def\smtp1{{\sigma}^{-t+1}}

\def\o{\omega }
\def\S1{S^{-1}}

\def\Xpm1{X^{\pm 1}_1}

\def\sPM1{{\sigma }^{\pm 1}}
\def\sMP1{{\sigma }^{\mp 1 }}

\def\b{\beta}
\def\d{\delta}

\def\di{{\rm d.ind}}

\def\L{\Lambda}
\def\O{\Omega}

\def\CA{{\cal A}}

\def\CD{{\cal D}}

\def\Ytm1{Y^{t-1}}
\def\Yim1{Y^{i-1}}

\def\CK{{\cal K}}

\def\CM{{\cal M}}
\def\CN{{\cal N}}

\def\CF{{\cal F}}
\def\CG{{\cal G}}
\def\CH{{\cal H}}
\def\ass{{\rm ass}}

\def\CZ{{\cal Z}}

\def\Aut{{\rm Aut}}

\def\bA{\overline{A}}

\def\dim{{\rm dim }}
\def\char{{\rm char }}

\def\ker{ {\rm ker } }

\def\gr{ {\rm gr} }
\def\gcd{ {\rm gcd } }

\def\Ev{ {\rm Ev} }

\def\SL2Z{ {\rm SL}_2({\bf Z}) }

\def\CZ{ {\cal Z}}

\def\Gp1{ G^{1 , 1 } }
\def\P11{ P^{-1 , 1 } }
\def\Pp1{ P^{1 , 1 } }

\def\lcm{{\rm lcm }}

\def\CE{{\cal E}}

\def\nCLsr{{}^\nu\kern-2pt {\cal L}^{\sigma , \rho  }}
\def\nP{{}^\nu \kern-2pt P}
\def\nL{{}^\nu\kern-2pt L}
\def\nLL{{}^\nu\kern-2pt \Lambda}
\def\nPsr{{}^\nu\kern-2pt P^{\sigma , \rho  }}
\def\nLsr{{}^\nu\kern-2pt L^{\sigma , \rho  }}
\def\nuCL{{}^\nu\kern-2pt  {\cal L}}
\def\nCLsr{{}^\nu\kern-2pt {\cal L}^{\sigma , \rho  }}
\def\nCL1m{{}^\nu\kern-2pt {\cal L}^{-1 , 1  }}
\def\x1nu{x^\frac{1}{\nu}}
\def\xm1nu{x^{-\frac{1}{\nu}}}

\def\rad{{\rm rad}}

\def\bF{\overline{F}}

\def\CN{{\cal N}}
\def\ra{\rightarrow }

\def\CB{{\cal B}}
\def\lcm{{\rm lcm}}
\def\CI{{\cal I}}

\def\CC{ {\cal C}}
\def\CE{ {\cal E} }
\def\CH{ {\cal H}}

\def\nAM0{{\nu }_{{\cal A},M_0}}
\def\nAN0{{\nu }_{{\cal A},N_0}}

\def\End{ {\rm End }}

\def\det{ {\rm det }}

\def\ga{\mathfrak{a}}
\def\gb{\mathfrak{b}}

\def\gn{\mathfrak{n}}
\def\gm{\mathfrak{m}}
\def\gp{\mathfrak{p}}
\def\gq{\mathfrak{q}}
\def\gr{\mathfrak{r}}

\def\GL{{\rm GL}}
\def\SL{{\rm SL}}

\def\Spec{{\rm Spec}}

\def\di!{\frac{\der^i}{i!}}
\def\dik!{\frac{\der^k_i}{k!}}
\def\hA{\widehat{A}}

\def\Fp{\mathbb{F}_p}

\def\gl{\mathfrak{l}}

\def\Max{{\rm Max}}

\def\N{\mathbb{N}}

\def\0{\overline{0}}
\def\1{\overline{1}}

\def\Ln1{\L_{n,\overline{1}}}

\def\a1{a_{\overline{1}}}

\def\S{\Sigma}

\def\vn1{\overrightarrow{n-1}}

\def\im{{\rm im}}

\def\gl{{\rm gl}}
\def\sl{{\rm sl}}

\def\mA{\mathbb{A}}

\def\soc{{\rm soc}}

\def\Inn{{\rm Inn}}

\def\mJ{\mathbb{J}}
\def\mI{\mathbb{I}}

\def\ann{{\rm ann}}

\def\mM{\mathbb{M}}
\def\mT{\mathbb{T}}

\def\mX{\mathbb{X}}
\def\mU{\mathbb{U}}

\def\Out{{\rm Out}}

\def\K1{{\rm K}_1}

\def\mK{\mathbb{K}}

\def\mY{\mathbb{Y}}

\def\hmI1{\widehat{\mI_1}}
\def\tmI1{\widetilde{\mI_1}}
\def\tmJ1{\widetilde{\mJ_1}}
\def\hB1{\widehat{B_1}}
\def\hCB1{\widehat{\CB_1}}

\def\Den{{\rm Den}}

\def\Ore{{\rm Ore}}

\def\Den{{\rm Den}}

\def\ga{\mathfrak{a}}

\def\tor{{\rm tor}}

\def\sl2{\mathfrak{sl}_2}
\def\gl2{\mathfrak{gl}_2}

\def\mK{\mathbb{K}}
\def\Prim{{\rm Prim}}

\def\res{{\rm res}}

\def\Irr{{\rm Irr}}
\def\b1{\overline{1}}

\def\Z{\mathbb{Z}}

\def\mR{\mathbb{R}}

\def\Deg{{\rm Deg}}
\def\Ind{{\rm Ind}}

\def\fC{{\mathfrak{C}}}
\def\fCK{{\mathfrak{C}}(K)}

\def\fCdK{{\mathfrak{C}}_d(K)}

\def\gl{{\mathfrak{l}}}

\begin{document}

\author{V. V. \  Bavula %(Spec-quantWeyl(root).tex)
}

\title{Classifications of  prime ideals and simple modules of the quantum Weyl algebra $A_1(q)$ ($q$ is a root of unity)}

\maketitle

\begin{abstract}
This paper consists of three parts: (I) To develop general theory of a (large) class of central simple finite dimensional  algebras and answering some natural questions about them (that  in general situation it is not even clear how to approach them, and the Brauer group is a step in the right directions), (II) To introduce and develop general theory of a  large class of rings,  PLM-rings,  (intuitively, they  are  the most general form of Quillen' Lemma), and (III) to apply these results for the following four classic algebras that turned out to be PLM-rings.  
 Let $K$ be an arbitrary field and $q\in K\backslash \{ 0,1\}$ be a primitive $n$'th root of unity. Classifications of 
prime, completely prime,  maximal and  primitive 
ideals,  and simple modules  
are obtained for the quantum Weyl algebra $A_1(q)=K\langle x,y \, | \,  xy-qyx =1\rangle$,  
  the skew polynomial algebra $\mA = K[h][x;\s ]$,  the  skew Laurent polynomial algebras $\CA :=  K[h][x^{\pm 1};\s ]$,  and   $\CB :=  K[h^{\pm 1}][x^{\pm 1};\s ]$   where $\s (h) = qh$. The quotient rings (of fractions) of  prime factor algebras of the algebras $A_1$, $\mA$, $\CA$,  and $\CB$  are explicitly described. Each quotient ring is a  central simple finite dimensional algebra, i.e. isomorphic to the matrix algebra $M_d(D)$ for some $d\geq 1$ and  a  central simple division algebra $D$. The division algebra $D$ is either a finite field extension of $K$ or a {\em cyclic} algebra.
  %  (i.e. central finite dimensional division algebras with strictly maximal subfield which is a Galois extension of the centre of the algebra and the Galois group is cyclic).
   These descriptions are a key fact in the classifications of  prime ideals, completely prime ideals, and simple modules  for the algebras above. For each simple module its basis, dimension, and endomorphism algebra are given.

   Explicit descriptions are obtained for the  automorphism groups of the algebras  $A_1$, $\mA$, $\CA$,  and $\CB$.\\



$${\bf Contents}$$
\begin{enumerate}
\item Introduction.
\item The algebras $\CE = (E(s), \s , a)$. 
\item PLM-rings and Quillen's Lemma.
\item   Classifications of  prime ideals and simple modules  of the  algebras  $\mA$, $\CA$,  and $\CB$.
\item Classifications of  prime ideals  and simple modules of the quantum Weyl algebra $A_1(q)$.
\item Automorphism groups of the  algebras $\mA$, $A_1$, $\CA$, and $\CB$.

\end{enumerate}
\end{abstract}

%%%%%%%%%%%%%%%%%% SECTION 1 %%%%%%%%%%%%%%%%%%%%%%%%

\section{Introduction}\label{INTR}%\marginpar{INTR}

This paper completes the study of the Weyl algebra $A_1(1)$ and its quantum analogues $A_1(q)$ as far as description of their automorphisms groups and classifications of prime, primitive, completely prime ideals, and simple modules are concerned over {\em arbitrary field}.\\

{\bf The Weyl algebra $A_1(1)$ and algebras $\mA (1)$ and  $\CA (1)$.}  If the field $K$ has characteristic zero  then the Weyl algebra $A_1(1):=K\langle x,y \, | \,  xy-yx =1\rangle$ is a central simple Noetherian domain. 
Over the field of complex numbers, simple $A_1$-modules were classified by Block, \cite{Bl} and over an arbitrary field of characteristic zero by Bavula \cite{Bav-UkrMathJ-92, Bav 2, Bav 3, Bav 5} by using approach of generalized Weyl algebras. For  an {\em algebraically closed} field $K$ of characteristic $p>0$, it is not difficult to show (by using Quillen's Lemma) that   every primitive ideal of the Weyl algebra $A_1(1)$ is of the form $(x^p-\l , y^p-\mu )$ for unique $\l , \mu \in K$ and the factor algebra $A_1(1)/(x^p-\l , y^p-\mu )$ is isomorphic to the matrix algebra $M_p(K)$,  and so every simple $A_1(1)$-module is $p$-dimensional. In prime characteristic, $A_1$-modules were studied by Tsuchimoto \cite{Tsuchi'03, Tsuchi'05},   Belov-Kanel and  Kontsevich
 \cite{Bel-Kon05JCDP},  and Bavula \cite{JC2n-DPn-Bavula}.
 
 Recently, for an arbitrary field $K$ of prime characteristic,    Bavula \cite{Bav-SpecWeylcharp} obtained classifications of prime, completely prime,   and  primitive 
ideals,  and simple modules 
 for the Weyl algebra $A_1(1)$,  
  the skew polynomial algebra $\mA (1):= K[h][x;\s ]$,  and the  skew Laurent polynomial algebra $\CA(1) :=  K[h][x^{\pm 1};\s ]$  where $\s (h) = h-1$. The quotient rings (of fractions) of  prime factor algebras of the algebras $A_1(1)$, $\mA (1)$,  and $\CA (1)$  are described. They are either fields or matrix algebras over {\em fields} or {\em cyclic} division algebras (i.e. central simple  finite dimensional division algebras with  maximal subfield which is a Galois extension of the centre of the algebra and the Galois group is a cyclic group).   These descriptions are a key fact in the classifications of completely prime ideals and simple modules  for the algebras above. Surprisingly,  the {\em general case} of central simple algebras does not appear, i.e. the matrix algebras $M_n(D)$ over division rings $D$ that are not fields  and $n\geq 2$. But for the quantum analogues of the algebras above the general case {\em does appear} and in order to treat it  a new machinery is developed.
  
  In characteristic zero  Dixmier  \cite{Dix}, and in prime characteristic   Makar-Limanov \cite{Makar-Limanov-1984}, gave an explicit set of  generators for  the automorphism group $\Aut_K (A_1(1))$ (see also \cite{Bav-AutWeylCharp} for more results on $\Aut_K (A_1(1))$ in prime characteristic).\\

{\bf The quantum Weyl algebra $A_1(q)$ and algebras $\mA$, $\CA$, and $\CB$.} In this paper, module means a left module; 
 $K$ is an arbitrary field and  $K^\times =K\backslash \{ 0\}$; $q\in K$ is a {\bf primitive $n$'th root of unity};  $\Fp=\Z /p\Z$ is the field with $p$ elements where $p$ is a prime number; algebra means a unital $K$-algebra; $$A_1=A_1(q):=K\langle x,y \, | \,  xy-qyx =1\rangle,$$   is the (first) {\bf quantum Weyl algebra}; $$\mA = K[h][x;\s ]$$ is  a skew polynomial algebra  where $K[h]$ is a polynomial algebra in a variable $h$ and  the automorphism $\s$ of $K[h]$ is given by the rule $\s (h) = qh$; and $$\CA = K[h][x, x^{-1};\s ]\;\; {\rm  and}\;\; \CB :=  K[h^{\pm 1}][x^{\pm 1};\s ]$$    are skew Laurent polynomial algebras. The four algebras are Noetherian domains of Gelfand-Kirillov dimension 2.  We identify the algebra $\mA$ with its image under the algebra monomorphism:
%\marginpar{qmAAWy}
\begin{equation}\label{qmAAWy}
\mA \ra A_1, \;\; h\mapsto y x+\frac{1}{q-1} , \;\; x\mapsto x.
\end{equation}
It follows that 
%\marginpar{qmAAWy1}
\begin{equation}\label{qmAAWy1}
\mA \subset A_1\subset \CA=\mA_x = A_{1,x}\subset \CB = \mA_{x, h} =  A_{1,x, h}
\end{equation}
where the algebras $\mA_x$ and  $ A_{1,x}$ are the localizations of the algebras $\mA$ and  $ A_1$ at the powers of the element $x$, respectively, and the algebras $\mA_{x, h}$   and  $ A_{1,x, h}$  are the localizations of the algebras $\mA$ and  $ A_1$ at the Ore set $S$ which is generated by the elements $h$ and $x$ as a monoid, respectively.  The quantum Weyl algebra $A_1$ is neither left nor right finitely generated $\mA$-module.\\

{\bf (I) The algebras $\CE =(E(s),\s, a)$.}
Let $F$ be a field that contains the primitive $n$'th root of unity $q$,  $F[h]$ be a polynomial algebra in the variable $h$, $\s$ be an $F$-automorphism of $F[h]$ which is given by the rule $\s (h)=qh$, and $s,a\in F$. The algebra $\CE =(E(s),\s, a)$ is an $F$-algebra that is generated by the elements $h$ and $x$ subject to the defining relations:
$$xh=qhx, \;\; h^n=s, \;\; {\rm and}\;\; x^n=a.$$
The algebra $\CE$ has dimension $n^2$ over the field $F$ and contains the commutative subalgebra $E=E(s):=F[h]/(h^n-s)$, and $\s$ is also an $F$-automorphism of the algebra $E(s)$. Generically, the quotient ring (of fractions) of a prime factor algebra of any of the algebras $A_1$, $\mA$, $\CA$, or $\CB$ belong to this class. So, generically,  classifications of completely prime ideals, primitive ideals, and simple modules for the four algebras are reduced to answering the following questions for the algebra $\CE$:
\begin{itemize}
\item {\bf Q1}. {\em  Give a criterion for the algebra $\CE$ to be a central simple algebra.}

\item {\bf Q2}. {\em Suppose that the algebra $\CE$ is a central simple algebra, i.e. $\CE\simeq M_m(D)$ for some natural number $m=m(\CE)\geq 1$ and a central division algebra $D=D(\CE)$. Give an explicit formula for the number $m$ and an explicit description of the division algebra $D$.}

\item {\bf Q3}. {\em For a central simple algebra $\CE$ give an explicit construction for a unique (up to isomorphism) simple $\CE$-module.}
\end{itemize}

 Theorem \ref{qA2Sep20}.(3a) and Theorem \ref{qA2Sep20}.(4) are criteria for the algebra $\CE$ be a simple and cyclic algebra, respectively.    Theorem \ref{qA2Sep20}.(5) is a criterion for the cyclic algebra $\CE$ to be isomorphic to the matrix algebra $M_n(F)$ over a field $F$.

Any central simple finite dimensional algebra  $A$ is isomorphic to a matrix algebra $M_n(D)$ over a division algebra $D$ which is also a  central simple finite dimensional algebra. The number $n$ and the  division algebra $D$ are  unique (up to isomorphism). For a given algebra $A$ it is not an easy task to find the pair $(n,D)$ or more generally to present the algebra $A$ as a matrix algebra  $M_m(B)$ where $B$ is central simple finite dimensional algebra. Theorem \ref{MnK-char} is a criterion for $A\simeq M_n(K)$ where $K$ is a field. 
Corollary \ref{MnK-char1}  presents sufficient conditions  for a central simple algebra $A$  to be isomorphic to the matrix algebra $M_n(B)$ for a central simple algebra  $B$. 

As far as the algebras $\CE$ are concerned there are two distinct cases either $s\neq 0$ and $a\neq 0$ or otherwise either $s= 0$ or $a=0$. The second case is easy and is considered in Lemma \ref{a25Jun23} where classifications of  prime ideals and simple modules for the algebra $\CE$ are obtained. So, we assume that $s\neq 0$ and $a\neq 0$. In this case the algebra $\CE$ is a cenral simple algebra.  Theorem \ref{NF-qA2Sep20} shows that   $\CE\simeq M_{m(s)}(\CE')$ where $\CE'$ is an explicit  cyclic algebra and Corollary \ref{Xa27Jun23} shows that   $\CE'\simeq M_{m(s,a)}(\CE'')$ where $\CE''$ is an explicit  cyclic algebra.

 Theorem \ref{12Jun23} presents every  algebra $\CE $ as a matrix algebra $M_m(D)$ over a division algebra $D$, i.e. the numbers $m$ and $d=\Deg (D)$ are described (where $\Deg (A)$ is the degree of a central simple algebra $A$), the division ring $D$ is realized as an explicit  subalgebra of the matrix algebra $M_d(E)$   where $d=\Deg (D)$. The division algebra $D$ is the fixed algebra of an explicit automorphism of the matrix algebra $M_d(E)$ (Theorem \ref{12Jun23}.(2)).
  An expression for the  natural number $m$  (Theorem \ref{12Jun23}.(1)) is given via the concept of the {\em matrix $(\s,i)$-norm} of a matrix $\mY\in M_d(E)$,
 $$N^{\s, i}_{M_d(E)} : M_d(E)\ra M_d(E), \;\; \mY\mapsto \mY^{\s^{i-1}}\cdots  \mY^\s \mY, \;\; i\geq 1,$$
which is  introduced in (\ref{xu=usxe4}). It is a matrix generalization of the field norm $$N_{E/F}: E\ra F\subseteq E, \;\; e\mapsto \prod_{j=0}^{n-1}\s^j (e)$$ since   $N^{\s, n}_{M_1(E)}=N_{E/F}.$
 Proposition \ref{A21Jun23} describes properties of the matrix $(\s,i)$-norm. There is a strong connection between  $N^{\s, i}_{M_d(E)}$ and $N_{E/F}$, e.g.
 $$\det \, N^{\s, n}_{M_d(E)}=N_{E/F}\, \det\;\; \;\;{\rm (Proposition \, \ref{A21Jun23}.(1))}.$$
  Lemma \ref{a14Jun23} gives an explicit description of the unique simple $\CE$-module.  
  
  The automorphism group $\Aut_F(D)$ is presented as a factor group of two explicit subgroups of $\Inn (M_d(E))$.
   Corollary \ref{b17Jun23} is a criterion for the algebra $\CE$ to be a division algebra.
  Corollary \ref{a12Jun23} gives a sufficient condition for the algebra $\CE$ to be a division algebra via the norm of the field extension $E/K$. Proposition \ref{c17Jun23} is a criterion for $\CE\simeq M_n(F)$ where $F$ is a field.
 
  If $\Deg(\CE)=\prod_{i=1}^dp_i^{n_i}$, where $p_i$ are distinct prime numbers and $n_i\geq 1$,  then the algebra  $$\CE =\bigotimes_{i=1}^d \CE_i$$ is a tensor product of its explicit  subalgebras $\CE_i$ with $\Deg(\CE_i)=p_i^{n_i}$  (Proposition \ref{A31May23}). \\

{\bf (II) PLM-rings, general picture and motivations, \cite{PLM-rings}.} In the Representation Theory of algebras,  the fact (when it is true) that {\em `the central elements of an algebra $A$ over an  algebraically closed field  act as scalars on every simple $A$-module $M$'}  is  of great importance. The  Representation Theory with this property is much more easy than without it and  is a key fact in classification of various classes of simple modules. When the ground field is not necessarily algebraically closed,  the direct generalization of the fact above is  the condition that 
\begin{eqnarray*}
 {\bf (*)} & \emph{ the image of every  central element of $A$ in the endomorphism algebra $\End_A(M)$ of each simple}\\
 & \emph{$A$-module $M$ is an algebraic element, i.e. the image of the centre of $A$ is an algebraic algebra}. 
\end{eqnarray*}
In particular, the kernel of the homomorphism $Z(A)\ra \End_A(M)$ is a {\em maximal} ideal of the centre $Z(A)$. This is the defining property for PLM-rings. Therefore, the PLM-rings is the most general form of this very useful phenomenon (that `central elements acts as scalars') and  of Quillen's Lemma \cite{Quillen'sLemma} that states that the condition (*) holds for all algebras that admit an ascending $\N$-filtration such that the associative graded algebra is a finitely generated commutative algebra.

Let $R$ be a ring, $\Spec (R)$ and $\Max (R)$ be the sets of prime and  maximal ideals of $R$, respectively, and $\Prim (A)$ be the set of primitive ideals of $R$.  The {\bf restriction map}
%\marginpar{SpecRZR}
\begin{equation}\label{SpecRZR}
\Spec (R)\ra \Spec (Z(R)), P\mapsto P\cap Z(R)
\end{equation}
is an well-defined map. 
The map 
%\marginpar{SpecRZR1}
\begin{equation}\label{SpecRZR1}
\Prim (R)\ra \Spec (Z(R)), P\mapsto P\cap Z(R)
\end{equation}
is called the {\bf primitive restriction map}. 
\begin{definition}
The images of the restriction and  primitive restriction maps are called  the {\bf central locus} and  the {\bf primitive locus} of $R$ and are denoted by ${\rm ZL} (R)$ and ${\rm PL} (R)$, respectively. 
\end{definition}
\begin{eqnarray*}
\Spec (R)&=&\coprod_{\gp\in {\rm ZL}(R)}\Spec (R,\gp)\;\; {\rm where}\;\; \Spec (R,\gp):=\{P\in \Spec (R)\, | \, P\cap Z(R)=\gp\},\\
\Prim (R)&=&\coprod_{\gp\in {\rm PL}(R)}\Prim (R,\gp)\;\; {\rm where}\;\; \Prim (R,\gp):=\{P\in \Prim (R)\, | \, P\cap Z(R)=\gp\}.
\end{eqnarray*}

\begin{definition}
The ring $R$ is called a {\bf  PLM-ring} if ${\rm PL} (R)\subseteq \Max (Z(R))$, that is for every simple left $R$-module $M$, $\ann_R(M)\cap Z(A)\in \Max (Z(A))$ (PLM-ring is the abbreviation of `prime locus maximal ring'). 
\end{definition}

All commutative rings are PLM-rings. All rings such that the centre  is a field are PLM-rings. In particular, all simple rings are PLM-rings, e.g.,  rings of differential operators on regular affine algebras over a perfect field (e.g., a field of zero characteristic). An algebra $A$ over a field $F$ is called an {\bf algebraic algebra} if each element of $A$ is algebraic over $F$, i.e. a `root' of a monic polynomial over the field $F$. 

\begin{proposition}\label{B30Jul23}%\marginpar{B30Jul23}
The following classes of algebras are PLM-algebras:
\begin{enumerate}
\item Algebras such that the endomorphism algebra of each simple module is  an algebraic algebra. 
\item  Algebras that admit an ascending $\N$-filtration such that the associative graded algebra is a finitely generated commutative algebra.
%\item Somewhat commutative algebras.
%\item Almost commutative algebras.
\item The universal enveloping algebra $U(\CG)$ of  a finite dimensional Lie algebra.

\end{enumerate}
\end{proposition}
So, the class of PLM-rings is a large class of rings. In Section \ref{PLMRINGS}, several results about PLM-rings are proved. Theorem \ref{31Jul23}.(2) is a criterion for a prime ideal $\gp\in {\rm ZL} (R)$ to belong to the set 
${\rm PL} (R)$. Corollary \ref{c1Aug23} describes the sets ${\rm PL} (R)$ and $\Prim (R)$. Theorem \ref{A1Aug23} is a criterion for a ring to be a PLM-ring. 

Theorem \ref{2Aug23}  is applied to the algebras $A_1$, $\mA$, $\CA$, and $\CB$ as they are PLM-rings, see also Corollary \ref{b2Aug23}. For a prime ideal $\gp$ of the centre $Z(R)$ of a ring $R$, $k(\gp)$ if the field of fractions of the commutative domain $Z(R)/\gp$.

\begin{theorem}\label{2Aug23}
%\marginpar{2Aug23}
Suppose that the ring $R$ is a free $Z(R)$-module and  for every $\gp \in {\rm ZL}(R)$,   the ring $k(\gp)\t_{Z(R)}R/\gp$ is a simple  artinian ring. Then  
\begin{enumerate}
\item The ring $R$ is PLM-ring with ${\rm PL}(R)=\Max (Z(R))$.
\item  For every $\gm \in \Max (Z(R))$, $\Prim (R,\gm)=\{ R\gm\}$.
\item $\Prim(R)=\Max (R)=\{ R\gm\, | \, \gm\in \Max (Z(R))\}$. 
\item $\widehat{R} =\{ U(\gm )\, | \, \gm \in \Max(Z(R))\}$  where $U(\gm)$ is a unique simple module of the simple artinian ring $k(\gm)\t_{Z(R)}R/R\gm$.
\item The map $\widehat{R} \ra \Prim (R)$, $U(\gm )\mapsto \ann_R(U(\gm))=R/R\gm$ is a bijection with inverse $R\gm\mapsto U(\gm)$.
\item For every maximal ideal $\gm \in \Max (R)$, $\End_R(U(\gm ))\simeq D(\gm)$ where $R/R\gm \simeq M_{n(\gm)}(D(\gm ))$ for some natural number $n(\gm)$ and a division ring $D(\gm )$. We write endomorphisms on the right.
\end{enumerate}
\end{theorem}

{\bf (III) Prime, completely prime, primitive ideals, and simple modules of the algebras $A_1$, $\mA$, $\CA$, and $\CB$.}  The aim of  Section \ref{qAKHSKEW} (resp.,   Section \ref{QUANTWEYL}) is to classify the sets of prime, completely prime,  maximal and primitive ideals,   and simple modules of the algebras  $\mA$, $\CA$,    and $\CB$ (resp., $A_1$), and also describe their automorphism groups.  
 Using the classification of primitive ideals we obtain a classification of simple modules for the algebras  $\mA$, $\CA$, $\CB$, and $A_1$. The case of the quantum Weyl algebra  $A_1$ is  more difficult than cases of the algebras  $\mA$, $\CA$, and $\CB$ and we use a different approach via generalized Weyl algebras exploiting the fact that the algebra $A_1$ is a generalized Weyl algebra.  These results are too technical and involved to present them in the introduction. The description of prime ideals for the algebra $A_1$ below gives a flavour of the results of Sections 
\ref{qAKHSKEW} and \ref{QUANTWEYL}. In Theorem \ref{28Jul23},  the solid (resp., dotted) lines represent containments (resp., obvious possible containments) of prime ideals.

\begin{theorem}\label{28Jul23}%\marginpar{28Jul23}

 $\Spec (A_1)=\Spec_{trs} (A_1)\, \sqcup \, \Spec_{t} (A_1/(r)) \, \sqcup \, \Spec_{r} (A_1/(t)) \, \sqcup \,\Spec (A_1/(t,r)) \, \sqcup \,
\Spec (A_1/(h))$ where
\begin{eqnarray*}
\Spec_{trs} (A_1) &=& \{   A_1\gp \, | \, \gp \in \Spec(Z(A_1) ); t,r,s \not\in \gp\} =\{ 0\} \, \sqcup \,\CN'' \, \sqcup \, \mM'', \\
\Spec_{t} (A_1/(r)) &=& \{ (r), (r,f)\, | \, f\in \Irr_m(K[t]\backslash \{ t\}\} = \{ (r)\}  \, \sqcup \, \mR, \\
 \Spec_{r} (A_1/(t)) &=& \{ (t), (t,g)\, | \, g\in \Irr_m(K[r]\backslash \{ r\}\}= \{ (t)\} \, \sqcup \, \mT,\\
\Spec (A_1/(t,r)) &=& \{ (t,r)\},  \\
\Spec (A_1/(h))&=& \{ (h), (h,f)\, | \, f\in \Irr_m(K[x]\backslash \{ x\}\}=\{ (h)\} \, \sqcup \, \mathbb{H}'', \\  
 \CN'' &:=& \{A_1\gn \,|\, \gn \in \Spec\, Z(A_1 ), {\rm ht} (\gn ) =1, \gn \neq (t), (r), (s)  \},\\ 
\mM'' &:=& 	\{ A_1 \gm\,|\, \gm \in \Max (Z(A_1 )); t,r,s \not\in \gm \}, \\  
  \mR &:=& \{ (r,f)\, | \, f\in \Irr_m(K[t]\backslash \{ t\}\}, \\
 \mT&=&  \{  (t,g)\, | \, g\in \Irr_m(K[r]\backslash \{ r\}\},\\
  \mathbb{H}'' &:=& \{ (h,f)\, | \, f \in \Irr_m(K[x])\backslash \{ x\}\},  
\end{eqnarray*}

 \begin{align}
	\begin{tikzpicture}[scale=1.3]
	\node (T) at (-4,2) {\small$\mT$};
	\node (t) at (-4, 1) {\small$(t)$};	
	\node (tr) at (-2,2) {\small$(t,r)$};
	\node (R) at (0, 2) {\small$\mR$};	
	\node (0) at (0,0)  {\small$\{ 0\}$};
	\node (r) at (0, 1) {\small$(r)$};
	\node (M) at (2, 2) {\small$\mM''$};	
	\node (N) at (2,1) {\small$\CN''$};
	\node (h) at (4, 1) {\small$(h)$};	
	\node (H) at (4, 2) {\small$\mathbb{H}''$};
	\draw [ shorten <=-2pt, shorten >=-2pt] (0)--(t)--(T);
	\draw [ shorten <=-2pt, shorten >=-2pt] (0)--(r)--(tr);
	\draw [ shorten <=-2pt, shorten >=-2pt] (0)--(N);
	\draw [ shorten <=-2pt, shorten >=-2pt] (0)--(h)--(H);
	\draw [ shorten <=-2pt, shorten >=-2pt] (t)--(tr);
	\draw [ shorten <=-2pt, shorten >=-2pt] (tr)--(r)--(R);	
\draw [semithick, dotted][ shorten <=-2pt, shorten >=-2pt]  (T)--(N)--(tr);
\draw [semithick, dotted][ shorten <=-2pt, shorten >=-2pt]  (R)--(N)--(M);
\draw [semithick, dotted][ shorten <=-2pt, shorten >=-2pt]  (N)--(H);
	\end{tikzpicture}  \label{qXYZA2} %\marginnote{qXYZA2}
	\end{align}

\end{theorem}

{\bf The automorphism groups of the algebras  $A_1$, $\mA$, $\CA$, and $\CB$.} In Section \ref{AUTGROUPS},   explicit descriptions of the automorphism groups of the algebras  $\mA$, $A_1$, $\CA$, and $\CB$ are given. The key idea is to use classifications of prime ideals and the fact that every algebra automorphism induces a {\em topological}  automorphism of the its prime spectrum equipped with  Zariski topology. `Peculiarities' of the prime spectra put strong restrictions  on the type of  automorphisms as the induced topological automorphisms must respect the `peculiarities'  and this leads to computing the automorphism groups. \\

{\bf Prime factor algebras and their quotient rings  for the algebras  $A_1$, $\mA$, $\CA$, and $\CB$.} Using the classifications of prime ideals, prime factor algebras and their quotient rings (of fractions) are described for the algebras  $A_1$, $\mA$, $\CA$, and $\CB$ (Section 
Section \ref{qAKHSKEW} and   Section \ref{QUANTWEYL}). The quotient rings are either  finite field extensions or central simple finite dimensional algebras $\CE$ from Section \ref{ALGCE}. Results of Section \ref{ALGCE} are used to determine the type of each quotient ring.

%%%%%%%%%%%%%%%%%%%%%% Section  2  %%%%%%%%%%%%%%%%%%

\section{The algebras $\CE = (E(s), \s , a)$
}\label{ALGCE}%\marginpar{ALGCE}

At the beginning of the section,  we recall some  results on central simple finite dimensional algebras that are used in the paper. 
 The aim of the section is to introduce and study a class of finite dimensional  algebras $\CE = (E(s), \s , a)$. The problem of classification of maximal ideals, completely prime ideals and simple modules for the algebras $\mA$, $A_1$, $\CA$ and $\CB$ are reduced to the same questions but for the algebras $\CE$, see Section \ref{qAKHSKEW}.

Recall that a  root of unity  $q\in K\backslash \{ 1\}$ is called a {\em primitive $n^{th}$ root of unity} if the natural number $n$ is the least one such that $q^n=1$, or equivalently  the elements $q, q^2, \ldots q^{n-1}, q^n =1$ are distinct. Suppose that  $p=\char(K)$. If $q$ is a primitive $n^{th}$ root of unity then either $p=0$ or otherwise $p \neq 0, p \not| n$ and necessarily $n \geqslant 3$.
% (if $p=2$ and $n=2$ then $0=q^2-1=(q-1)^2$ and so $q=1$, a contradiction). 
 If $q$ is a primitive $n^{th}$ root of unity then the polynomial $x^n-1$ has simple roots (since $(x^n-1)'=nx^{n-1} \neq 0$ and $\gcd(x^n-1, nx^{n-1})=1$) and

%\marginpar{xnix}
\begin{equation}\label{xnix}
x^n-1=\prod_{i=0}^{n-1}(x-q^i). 
\end{equation}
%\begin{align}
%x^n-1=\prod_{i=0}^{n-1}(x-q^i).    \label{xnix}  
%\marginnote{xnix}
%\end{align}
More generally, for all $\mu \in K$,  
\begin{align}
\prod_{i=0}^{n-1}(x-q^i \mu)=x^n-\mu^n.    \label{xnix1}  %\marginnote{xnix1}
\end{align}
Indeed, the equality is obvious if $\mu=0$. If $\mu \neq 0$ then 
\begin{align*}
\prod_{i=0}^{n-1}(x-q^i \mu)=\mu^n \prod_{i=0}^{n-1}\bigg(\frac{x}{\mu}-q^i\bigg)\stackrel{(\ref{xnix})}{=} \mu^n \bigg(\bigg(\frac{x}{\mu}\bigg)^n-1 \bigg)=x^n-\mu^n.
\end{align*}
If $q$ is a primitive $n^{th}$ root of unity then, by (\ref{xnix}), 
\begin{align}
\xi_n:=\prod_{i=1}^{n-1}q^i=q^{\frac{n(n-1)}{2}}=(-1)^{n-1}. \label{xnix2} %\marginnote{xnix2}
\end{align}

{\bf The algebras $\CE =(E, \s , a)$.}
Let $F$ be a field that contains the primitive $n$'th root of unity $q$,  $F[h]$ be a polynomial algebra in a variable $h$,  and $s\in F$. Let us consider the factor algebra 
%\marginpar{qEs}
\begin{equation}\label{qEs}
E = E(s):=F[h]/(h^n-s)=\bigoplus_{i=0}^{n-1}Fh^i.
\end{equation}
The algebra $E$ is a finite dimensional commutative algebra with $\dim_F(E)=n$. 
The algebra $E$ admits the $F$-automorphism $\s : h\mapsto qh$ of order $n$.  Let $E[x;\s ]$ be a skew polynomial algebra. Then the element $x^n$ belongs to the centre of the algebra $E[x;\s ]$.  Let $\bF$ be the algebraic closure of the field $F$. If $\l \in \bF$  is a root of the polynomial $h^n-s$  then $s=\l^n $ and 

%\marginpar{qEs1}
\begin{equation}\label{qEs1}
h^n-s=h^n-\l^n = \prod_{i=0}^{n-1}(h-q^i\l),
\end{equation}
and so the polynomial $h^n-s$ of degree $n$ has $n$ distinct roots $\{ q^i\l \, | \, i=0,1, \ldots , n-1\}$ provided $s\neq 0$. Let $s_1(x_1, \ldots , x_n)=x_1+\cdots +x_n, \ldots , s_n(x_1, \ldots , x_n)=x_1\cdots x_n$ be the elementary symmetric functions in $n$ variables. By (\ref{qEs1}), 
%\marginpar{siqi}
\begin{equation}\label{siqi}
s_i(1,q,\ldots , q^{n-1})=0\;\; {\rm for}\;\; i=1, \ldots , n-1, \;\; {\rm and}\;\; s_n(1,q,\ldots , q^{n-1})=q^{\frac{n(n-1)}{2}}=(-1)^{n-1}.
\end{equation}
For each element $a\in F$, the element $x^n-a$ belongs to the centre of the algebra $E[x;\s ]$. Hence,  the factor algebra 
%\marginpar{qEs2}
\begin{equation}\label{qEs2}
\CE =(E, \s , a):=E[x;\s ]/(x^n-a)=\bigoplus_{j=0}^{n-1}Ex^j
\end{equation}
 is a finite dimension $F$-algebra of dimension $n^2$ which is generated over the field $F$ by the elements $h$ and $x$ subject to the defining relations:
%\marginpar{qEs3}
\begin{equation}\label{qEs3}
xh=\s (h)x, \;\; h^n=s, \;\; {\rm and}\;\; x^n=a.
\end{equation}
If $s\neq 0$ the the direct sum in (\ref{qEs2}) is a direct sum of eigenspaces $Ex^i$ of the inner automorphism $\o_h $ of the algebra $(E, \s , a)$  since $\o_h(x^i)=q^{-i}x^i$ for all $i$, and the set of eigenvalues of $\o_h$ is $\Ev (\o_h)=\{ 1,q^{-1}, \ldots , q^{-n+1}\}$. 

There is an algebra isomorphism
%\marginpar{qEs6}
\begin{equation}\label{qEs6}
\xi : (F[h]/(h^n-s), \s , a) \ra (F[h]/(h^n-a), \s' , s), \;\; x\mapsto h, \;\; h\mapsto x
\end{equation}
where $\s'\in \Aut_F(F[x]/(x^n-a))$ and $\s'(x)=q^{-1}x$.

{\bf Central simple finite dimensional algebras.}  For a field $K$, we denote by $\fCK$ the class of all  central simple finite dimensional  $K$-algebras. Let  $\fCdK$ the class of all  central simple finite dimensional  division $K$-algebras. Clearly,  $\fCdK\subseteq \fCK $. A 
$K$-algebra belongs to the class $\fCK$ iff $A\simeq M_n(D)$ for some $D\in \fCdK$.

Let us recall some basic facts and definitions on central simple algebras and cyclic algebras, see the book \cite{Pierce-AssAlg} for details.  For each algebra $\CA\in \fCK$, $\dim_K(\CA )=n^2$ for some natural number $n\geq 1$ which is called the {\em degree} of $\CA $ and is denoted by $\Deg (\CA )$. The algebra $\CA$ is isomorphic to the matrix algebra $M_l(D)$ over a division algebra $D\in \fCK$ which is unique (up to a $K$-isomorphism). Clearly, $\Deg (\CA ) = l\Deg (D)$. The degree $\Deg (D)$ is called the {\em index} of $\CA$ and is denoted by $\Ind (\CA )$. If $L$ is a  subfield of $\CA$ that contains the field $K$ then $[L:K]\leq \Deg (\CA )$. The field $L$ is called a {\em strictly maximal} subfield if $[L:K]= \Deg (\CA )$. For a subset $B$ of an algebra $A$, the subalgebra of $A$, $C_A(B):=\{ a\in A\, | \, ab=ba$ for all $b\in B\}$, is called the {\em centralizer} of $B$ in $A$. 

A Galois field extension $L/K$ is called a {\em cyclic extension} if the Galois group $G(L/K)$ is a cyclic group (i.e. generated by a single element). If the field extension $L/K$ is cyclic,  $\s$ is a generator of the Galois group $G(L/K)$,  $n=|G(L/K)|$ is the order of the group $G(L/K)$, and $l\in L$ then the element of $F$,
 %\marginpar{qEs5}
\begin{equation}\label{qEs5}
N_{L/K}(l):=\prod_{i=0}^{n-1}\s^i(l),
\end{equation}
is called the {\em norm} of $l$, and the map $N_{L/K}: L\ra K$, $l\mapsto N_{L/K}(l)$ is the {\em norm} of  the field extension $L/K$. The map $N_{L/K}: L^\times \ra K^\times$ is a group homomorphism.

Two algebras $\CA, \CB\in \fCK$ are called {\bf equivalent} if $M_n\t_K \CA\simeq M_m(K)\t_K\CB$ for some natural numbers $n$ and $m$,  and we write $\CA \sim \CB$. The algebras  $\CA, \CB\in \fCK$ are equivalent iff they are Morita equivalent iff $\CA\simeq M_k(D)$ and $\CB\simeq M_l(D)$ for some natural numbers $k$ and $l$,  and some division algebra $D\in \fCK$. \\

{\bf The image $\im (N_{E/F})$ of the norm  $N_{E/F}$.}  If $s\in F^\times$ and the algebra $E(s)$ is a field then then the field extension $E(s)/F$ is a Galois field extension with Galois group $G(E(s)/F)=\langle \s \, | \, \s^n=1\rangle$ where $\s (h)=qh$ and $q$ is a primitive $n$'th root of unity. By (\ref{siqi}), 
%\marginpar{siqi1}
\begin{equation}\label{siqi1}
N_{E(s)/F}(h)=q^{\frac{n(n-1)}{2}}h^n=(-1)^{n-1}s.
\end{equation}

Let $e\in E$. If $e=\l \in F$ then $N_{E/F}(\l)=\l^n$. If $e\in E\backslash F$ then $e=\sum_{i=0}^m\mu_ih^i$ for some elements $\mu_i\in F$, $\mu_m\neq 0$,  and $m\in \{ 1, \ldots , n-1\}$. Notice that  $e=\mu_m\prod_{i=1}^m(h-\l_i)$ where $\l_1, \ldots, \l_m\in \bF$ are the roots of the polynomial $e$ (not necessarily distinct). Now,
%\marginpar{qNormEF}
\begin{equation}\label{qNormEF}
N_{E/F}(e) =(-1)^{m(n-1)}\mu_m^n\prod_{i=1}^m ( s-\l_i^n). 
\end{equation} 
In more detail,
\begin{eqnarray*}
N_{E/F}(e) &=& \mu_m^n\prod_{i=1}^mN_{E/F}(h-\l_i)=\mu_m^n\prod_{i=1}^m \prod_{j=0}^{n-1} (q^jh-\l_i)=\mu_m^n\prod_{i=1}^m (-1)^{n-1}\prod_{j=0}^{n-1} (h-q^{-j}\l_i)\\
&=&(-1)^{m(n-1)}\mu_m^n\prod_{i=1}^m  (h^n-\l_i^n)=(-1)^{m(n-1)}\mu_m^n\prod_{i=1}^m  (s-\l_i^n).\\
\end{eqnarray*}
Proposition \ref{qA21Mar23} describes the image of the norm map $N_{E(s)/F}$.

\begin{proposition}\label{qA21Mar23}%\marginpar{qA21Mar23}
Suppose that the $F$-algebra $E=E(s)$ is a field and  $s\in F^\times $, $\s \in \Aut_F(E)$ where $\s (h)=qh$ and $q$ is a primitive $n$'th root of unity,  and $N_{E/F}:E\ra F$, $e\mapsto \prod_{i=1}^{n-1}\s^i(e)$.  Then
\begin{eqnarray*}
\im (N_{E/F})&=&\bigg\{\l^n, (-1)^{m(n-1)}\mu^n\prod_{i=1}^m (s-\l_i^n) \, \bigg| \, \l\in F, \mu\in F^\times, \l_i\in \bF, s_i(\l_1, \ldots , \l_m)\in F\;\\
&&   {\rm for\; all }\;  i=1, \ldots , m\;\; {\rm where}\;\; m=1, \ldots ,n-1\bigg\}\\
%\ker (N_{E/F})&=&\bigg\{0, \mu\prod_{i=1}^m(h-\l_i) \, \bigg| \, s=\l_j^n\;\; {\rm for\; some}\;  j\in \{ 1, \ldots , m \},  \mu\in F^\times, \l_i\in \bF,\\ 
%&&  s_i(\l_1, \ldots , \l_m)\in F\;
%   {\rm for\; all }\;  i=1, \ldots , m\;\; {\rm where}\;\; m=1, \ldots ,n-1\bigg\},
\end{eqnarray*}
where   $\bF$ is the algebraic closure of the field $F$, and  $s_i(x_1, \ldots , x_m):=\sum_{1\leq j_1<\cdots <j_i\leq m}x_{j_1}\cdots x_{j_i}$ are the elementary symmetric functions in $m$ variables $x_1, \ldots , x_m$.
\end{proposition}

\begin{proof} Let $e\in E$. If $e=\l \in F$ then $N_{E/F}(\l)=\l^n$. If $e\in E\backslash F$ then $e=\sum_{i=0}^m\mu_ih^i=\mu_m\prod_{i=1}^m(h-\l_i)$ for some elements $\mu_i\in F$, $\mu_m\neq 0$,  and $m\in \{ 1, \ldots , n-1\}$  where $\l_1, \ldots, \l_m\in \bF$ are the roots of the polynomial $e$ (not necessarily distinct). Clearly, $$F\ni \mu_i=(-1)^{m-i}\mu_m s_{m-i}(\l_1, \ldots , \l_m)\;\; {\rm  for}\;\;i=0,1,  \ldots , m$$ where $s_0:=1$. Now, the result follows from (\ref{qNormEF}).
\end{proof} 

A finite dimensional central simple $K$-algebra $\CA$  is called a {\bf cyclic algebra} if there is a strictly maximal subfield $L$ of $\CA$ such that the field extension $L/K$ is cyclic. For a $K$-algebra $A$ and a subset $G\subseteq \Aut_K (A)$, the set   $A^G:= \{ a\in A\, | \, g(a)=a$ for all $g\in G\}$ is a subalgebra of $A$, the {\em fixed algebra} of $G$. Proposition \ref{Pra-15.1} is a description of  cyclic algebras. 

\begin{proposition}\label{Pra-15.1}%\marginpar{Pra-15.1}
(\cite[Proposition a, Section 15.1]{Pierce-AssAlg}.) 
If $\CA$ is a central simple finite dimensional $K$-algebra then $\CA$ is cyclic if and only if
\begin{equation}\label{Es4}
\CA = \bigoplus_{i=0}^{n-1}Lx^i, \;\;xl=\s (l) x \;\;  {\it for \; all}\;\; l\in L,\;\; {\it  and}\;\; x^n=a,
\end{equation}
 for some strictly maximal cyclic subfield $L/K$ such that the Galois group $G(L/K)=\langle \s \rangle$ is a cyclic group of order $n$, $x\in \CA^\times$, and $a\in K^\times$. 
\end{proposition}

The cyclic algebra $\CA$ is denoted by $(L, \s , a)$ (in \cite{Pierce-AssAlg}, the cyclic algebra $\CA$ is denoted by $(L, \s^{-1} , a)$). For a ring $R$ and a natural number $n\geq 1$, let $R^{[n]}$ be the image of the map $R\ra R$, $r\mapsto r^n$, i.e.  $R^{[n]}:=\{ r^n\, | \, r\in R\}$. Notice that $R=R^{[1]}\supseteq R^{[2]}\supseteq \cdots $. 

For an element $\l\in \bF$ such that $\l^\nu \in F$ for some natural number $\nu\geq 1$, the set $\{i\in \Z\, | \, \l^i\in F\}$ is a subgroup of $\Z$, and so 

%\marginpar{nFl}
\begin{equation}\label{nFl}
\{i\in \Z\, | \, \l^i\in F\}=\Z n_F(\l) \;\; {\rm for\; a \;unique\; natural \; number}\;\;n_F(\l)\geq 1.
\end{equation}
So, if $\l^i\in F$ for some nonzero integer $i$ then $ n_F(\l)\mid i$.

In general, the algebras $(E, \s , a)$ considered above (see (\ref{qEs2})) are not cyclic algebras or simple algebras.  Theorem \ref{qA2Sep20}.(3a) is a criterion for the algebra $(E, \s , a)$ to be a simple algebra. Theorem \ref{qA2Sep20}.(4) is a criterion for the algebra $(E, \s , a)$ to be a cyclic algebra. Theorem \ref{qA2Sep20}.(5) is a criterion for the cyclic algebra $(E, \s , a)$ to be isomorphic to the matrix algebra $M_n(F)$.

\begin{theorem}\label{qA2Sep20}%\marginpar{qA2Sep20}
Let $F$, $E=E(s)$,  and $\CE=(E, \s, a)$ be as above, see (\ref{qEs2}).
\begin{enumerate}
\item The polynomial $h^n-s\in F[h]$ is an irreducible polynomial over $F$ if and only if $s\not\in \bigcup_{m\mid n, m\neq 1}F^{[m]}$.

\item Suppose that the polynomial $h^n-s\in F[h]$ is an irreducible polynomial over $F$. Then the field extension $E/F$ is a Galois field extension with Galois  group $G(E/F)=\langle \s \, | \, \s^n=1\rangle$  which is a cyclic group of order $n$ where $\s (h) = qh$, and $[E:F]=n$. 
\item 
\begin{enumerate}
\item The algebra $\CE$ is a simple algebra if and only if $s\neq 0$ and $a\neq 0$. If  $s\neq 0$ and $a\neq 0$ then the  algebra $\CE$ is a central simple algebra of degree $n$ (i.e.  $Z(\CE)=F$ and $\dim_F\, (\CE)=n^2$).
\item $\rad (\CE)=\begin{cases}
(x)& \text{if $s\neq 0$ and $a=0$},\\
(h)& \text{if $s= 0$ and $a\neq 0$},\\
(x,h)& \text{if $s=0$ and $a=0$}, \\
\end{cases}$
and 

$\CE/\rad (\CE)=\begin{cases}
F[h]/(h^n-s)& \text{if $s\neq 0$ and $a=0$},\\
F[x]/(x^n-a)& \text{if $s= 0$ and $a\neq 0$},\\
F& \text{if $s=0$ and $a=0$}.\\
\end{cases}$
\end{enumerate}
\item The algebra $\CE$ is a cyclic algebra iff $E$ is a field and $a\neq 0$.
% In particular, the centre of the algebra $(E, \s ,a)$ is $F$  and the field $E$  is  a strongly maximal subfield of the cyclic algebra $(E,\s , a)$ (i.e. $[E:F]=n$). 

%\item *** Suppose that $a\neq 0$ and the algebra  $E$ is a field then $s\neq 0$ and  the algebra $(E, \s , a)$ is  cyclic algebra  of degree $n$.  In particular, the centre of the algebra $(E, \s ,a)$ is $F$  and the field $E$  is  a strongly maximal subfield of the cyclic algebra $(E,\s , a)$ (i.e. $[E:F]=n$).*** 

\item Suppose that  the algebra  $E$ is a field and $a\neq 0$. Then the  algebra $\CE$ is  isomorphic to the matrix algebra $M_n(F)$ iff  $a=N_{E/F}(b$ for some  element $b\in E^\times$.

%\item ***  WRONG ***  Suppose that $a\neq 0$ and the algebra  $E$ is a field. Then the  algebra $(E,\s , a)$ is either a division algebra or otherwise is isomorphic to the matrix algebra $M_n(F)$.   ***  ***

%\item ***  WRONG ***   Suppose that $a\neq 0$  and the algebra  $E$ is a field. Then the algebra $(E,\s , a)$ is   a division algebra iff the element $a$ is not a norm $N_{E/F}(b)=\prod_{i=0}^{n-1}\s^i(b)$ of any element $b\in E^\times$, i.e. $a\not\in \{ \prod_{i=0}^{n-1}\s^i(b)\, | \, b\in E^\times\}$. ***   ***

\end{enumerate}
\end{theorem}

\begin{proof}    1. It suffices  to show
that the polynomial $h^n-s$ is {\em reducible} over $F$ if and only if $s\in \bigcup_{m\mid n, m\neq 1}F^{[m]}$.

$(\Rightarrow )$ Suppose that the polynomial  
$h^n-s$ is reducible. The implication holds for $s=0$. So, we may assume that $s\neq 0$. Then $h^n-s=fg$ for some nonscalar polynomials $f,g\in F[h]$.  By (\ref{qEs1}), there is a disjoint union 
$\{0,1, \ldots , n-1\}=S\coprod T$ of non-empty sets  $S$ and $T$ such that $f=\prod_{i\in S}(h-q^i\l)$ and $g=\prod_{j\in T}(h-q^j\l )$ where $\l =s^{\frac{1}{n}} \in \bF$ is a root of the polynomial $h^n-s$, see (\ref{qEs1}). Let $l=|S|$. Then $1\leq l\leq n-1$ and 
$$f=h^l+\sum_{k=1}^l\mu_k\l^kh^{l-k}\in F[h]$$ 
for some elements $\mu_k \in F$ not all of them are equal to zero since $s\neq 0$. Notice that $\mu_l=(-1)^lq^{\sum_{i\in S}i}\in F$ (since $q\in F$), and so $\l^l\in F$. Let $u=n_F(\l)$. By (\ref{nFl}), $\l^u\in F$ and  $u\mid n$ (since $\l^n=s\in F$).  Since $\l^l\in F$ and $l<n$,  $u\neq n$.  Then $$s=\l^n=\Big(\l^u\big)^m\in F^{[m]}\;\; {\rm  where}\;\;m=\frac{n}{u}.$$ Clearly $m\mid n$ and $m\neq 1$. 

$(\Leftarrow )$ If $s\in \bigcup_{m\mid n, m\neq 1}F^{[m]}$ then $ s\in F^{[m]}$ for some $m$ as above, i.e. $s =\mu^m$ for some $\mu\in F$,  and so  the polynomial $x^n-s=x^n-\mu^m$ is a reducible polynomial since it is divisible by the polynomial $ x^{\frac{n}{m}}-\mu\in F[x]$.

2. By (\ref{qEs1}), the field extension $E/F$ is Galois and $G(E/F)=[E:F]=n$. It follows from  $$\s (h^n-\l )=q^nh^n-\l= h^n-\l$$ that  $\s \in G(E/F)$. Furthermore, $G(E/F)=\langle \s \rangle$ since the order of the automorphism $\s$ is $n$ ($q$ is a primitive $n$'th root of unity).

3.(b) Suppose that $s=0$ and $a=0$. Then the elements $h$ and $x$ are normal nilpotent elements of the algebra $\CE$. Therefore, $(h,x)\subseteq \rad (\CE)$. Since $\CE/(h,x)\simeq \mA/ (h,x)\simeq F$, $\rad (\CE ) = (h,x)$. 

In view of the isomorphism $\xi$ (see (\ref{qEs6})), it suffices to consider the case when $s\neq 0$ and $a=0$ (as the case $s= 0$ and $a\neq 0$ can be reduced to this one via $\xi$). So, suppose that $s\neq 0$ and $a=0$.

The element $x$ is a normal and nilpotent element of the algebra $(E,\s ,0)$. Hence, the nilpotent ideal $(x)$ belongs to the radical of the algebra $(E, \s , 0)$ (which is the largest nilpotent ideal). Since the factor algebra $(E, \s ,0) / (x) \simeq E$ is a semisimple algebra (which is either a field or a direct product of fields), the ideal $(x)$ is the radical of $(E,\s ,0)$.

(a) In view of the statement (b), it suffices to show that if $s\neq 0$ and $a\neq 0$ then the algebra $\CE=(E, \s , a)$ is a central simple algebra (since $\dim_F (\CE)=n^2$, i.e. the degree of the algebra $\CE$ is $n$).  

(i) {\em The algebra $\CE$ is simple}: Given a nonzero element $z=\sum_{i,j=0}^{n-1}z_{ij}h^ix^j$ of the algebra $\CE$  where not all elements $z_{ij}\in F$ are equal to zero. We have to show that $(z)=\CE$. Each summand $z_{ij}'=z_{ij}h^ix^j$ of $z$ is a common eigenvector for the inner automorphisms $\o_h$ and $\o_x$ of the algebra $\CE$:
$$ \o_h(z_{ij}')=q^{-j}z_{ij}'\;\; {\rm and}\;\; \o_x(z_{ij}')=q^iz_{ij}'.$$
Since the elements $(q^{-j},q^i)$, $ 0\leq i,j\leq n-1$, are distinct and the elements $h$ and $x$ are units, all $z_{ij}'\in (z)$. Hence, all $z_{ij}\in (z)$, and so $(z)=\CE$. 

(ii) {\em The centre of the algebra $\CE$ is $F$}: By (\ref{qEs2}), the centralizer $C(h)$ of the element $h$ in $\CE$ is equal to $\CE^{\o_h}= E$. Now, the set of elements of $E$ that commute with the element $x$ is $E^{\o_x}=E^\s=F$, and the statement (ii) follows.

4. $(\Rightarrow )$ Suppose that the algebra $(E, \s ,a)$ is a cyclic algebra. Then necessarily  $E$ is a field and $a\neq 0$ (by statement 3).

$(\Leftarrow)$ Suppose that $E$ is a field and $a\neq 0$.
Since the algebra $E$ is a field, $s\neq 0$.  Since $s\neq 0$ and  $a\neq 0$, the algebra $\CE$ is a central simple algebra of degree $n$, by statement $3.(a)$. Since 
$$[E:F]=n= {\rm  the\; degree\; of}\;\; \CE,$$
 the field $E$ 
 is  a strongly maximal subfield of the central simple algebra $\CE$. By statement 2, the field extension $E/F$ is cyclic, and so the algebra $\CE$ is a cyclic algebra.

 5. By statement 4,  the algebra $\CE$ is a cyclic algebra. Then statement 5 follows from   \cite[Lemma 15.1]{Pierce-AssAlg}.    \end{proof}   
 
 By Theorem \ref{qA2Sep20}.(1), the algebra  $E$ is not a field and $s\neq 0$ iff $0\neq s\in \bigcup_{m\mid n, m\neq 1}F^{[m]}$. Let 
 %\marginpar{msmin}
\begin{equation}\label{msmin}
m(s):=m(s;n):=\max\{m \, |  \, m\geq 1,   s\in F^{[m]},  m\mid n \}=\max\{m \, |  \, m\geq 1,  s^\frac{1}{m}\in F,  m\mid n\}
\end{equation}
where  $s^\frac{1}{m}$ means an $m$'th root of $s$ and 
 %\marginpar{msamin}
\begin{equation}\label{msamin}
m(s,a):=m(s,a; n):=\max\bigg\{m \, |  \, m\geq 1,   
a\in F^{[m]}, m\mid \frac{n}{m(s)} \bigg\}=\max\bigg\{m \, |  \, m\geq 1,   a^\frac{1}{m}\in F,   m\mid \frac{n}{m(s)}  \bigg\}
\end{equation}
By Theorem \ref{qA2Sep20}.(1), the algebra  $E$ is  a field iff $m(s)=1$.

\begin{lemma}\label{a28May23}%\marginpar{a28May23}
%Let $F$, $E=E(s)$,   and $q$ be as above (in particular,  $q\in F$ is a primitive $n$'th root of unity). Suppose $0\neq s\in \bigcup_{m\mid n, m\neq 1}F^{[m]}$. Then 
Suppose that $s\neq 0$. 
\begin{enumerate}
\item  $m(s)=\lcm \{m \, |  \,  m\geq 1,   s\in F^{[m]},  m\mid n \}$. In particular,  $m'\mid m(s)$ for all $m'\geq 1$ such  that  $ m'\mid n$   and   $s\in  F^{[m']}$;  and $m(s;\frac{n}{m(s;n)})=1$. So, the number $m(s)$ is the largest divisor of $n$  such that $s^\frac{1}{m(s)}\in F$.

\item  The field extension $E'/F$, where $E':=F[h]/(h^\frac{n}{m(s)}-s^\frac{1}{m(s)})$,  is a cyclic  field extension with Galois  group $G(E'/F)=\langle \s':=\s^{m(s)} \, | \, {\s'}^\frac{n}{m(s)}=1\rangle$  which is a cyclic group of order $\frac{n}{m(s)}$ where $\s' (h) =\s^{m(s)}(h)= q^{m(s)}h$, and $[E':F]=\frac{n}{m(s)}$. 

\item  The factor algebras  $F[h]/\Big(h^\frac{n}{m(s)}-q^\frac{in}{m(s)}  s^\frac{1}{m(s)}\Big)$, where $i=0,1, \ldots , m(s)-1$,  are $F$-isomorphic cyclic field extensions with Galois groups equal to  $G(E'/F)=\langle \s':=\s^{m(s)} \, | \, {\s'}^\frac{n}{m(s)}=1\rangle$. 
\end{enumerate}
\end{lemma}

\begin{proof}   1. To prove the first equality it suffices to show that if $s\in  F^{[m]}$ and $s\in  F^{[m' ]}$ for some divisors  $m$ and $ m'$ of  $n$  then $s\in  F^{[l]}$ where $l:=\lcm (m,m')$ (since clearly $l\mid n$). Notice that $g:=\gcd (m,m')=bm+b'm'$ for some $b,b'\in \Z$ and $lg=mm'$. The inclusions $s\in  F^{[m]}$ and $s\in  F^{[m' ]}$ are equivalent 
to the inclusions $s^\frac{1}{m}\in F$ and $s^\frac{1}{m'}\in F$. 
Then  
$$s^\frac{1}{l}=s^\frac{g}{mm'}=s^\frac{bm+b'm'}{mm'}=s^\frac{b}{m'}s^\frac{b'}{m}\in F, $$
and so $s\in  F^{[l]}$, as required. The equality $m(s;\frac{n}{m(s;n)})=1$ follows from the first equality.

2.  By the assumption $q\in F\subseteq E'$, and so the field  $F$ contains the $\frac{n}{m(s)}$'s root of unity $q^{m(s)}$.  By statement 1,  $m(s;\frac{n}{m(s;n)})=1$, and so the algebra $E'$ is a field,  by Theorem \ref{qA2Sep20}.(1). Now, statement 2 follows from Theorem \ref{qA2Sep20}.(2). 

3. The automorphism $\s\in \Aut_F(F[h])$, where $\s (h)= qh$, induces  $F$-isomorphisms 
$$\s^{-i}:E'\ra F[h]/\Big(h^\frac{n}{m(s)}-q^\frac{in}{m(s)}  s^\frac{1}{m(s)}\Big)\;\; {\rm for}\;\; i=1, \ldots , m(s)-1.$$
By Theorem \ref{qA2Sep20}.(1) and statement 2, the $F$-algebra $E'$ is a cyclic field extension of $F$ with Galois group $G(E'/F)$ described in statement 2, and statement 3 follows.  \end{proof}

{\bf The case where either $s=0$ or $a=0$: Classification of prime ideals and simple modules of  the algebra $\CE$.}  Suppose that $A$ is a finite dimensional $F$-algebra. Then its radical $\rad (A)$ is a nilpotent ideal. Then the factor algebra $\bA :=A/\rad (A)$ is a semisimple finite dimensional algebra. Then all prime ideals of the algebra $A$ contain the radical $\rad (A)$ and all simple $A$-modules are annihilated by the radical $\rad (A)$. Hence, the map 

%\marginpar{SpAr}
\begin{equation}\label{SpAr}
\Spec (A)\ra \Spec (\bA), \;\; \gp\mapsto \gp /\rad (A)
\end{equation}
is a bijection and 
%\marginpar{SpAr1}
\begin{equation}\label{SpAr1}
\hA = \widehat{\bA},
\end{equation}
by the restriction of scalars.

\begin{lemma}\label{a25Jun23}%\marginpar{a25Jun23}
Let $\CE =(E,\s , a)$ where $E=E(s)$. 
\begin{enumerate}
\item Suppose that $s\neq 0$ and $a=0$. Then 
$$\CE /\rad (\CE)=E=
\begin{cases}
E \; {\rm is \; a \; field} & \text{if } m(s)=n,\\
\prod_{i=0}^{m(s)-1} F[h]/\Big(h^\frac{n}{m(s)}-q^\frac{in}{m(s)}s^\frac{1}{m(s)} \; {\rm is \; a \; direct \; product\; of \;fields}& \text{if }m(s)\neq n,\\
\end{cases}
$$ 
$$\Spec (\CE)=
\begin{cases}
\{ (x) \} & \text{if } m(s)=n,\\
\Big\{ \Big(x, h^\frac{n}{m(s)}-q^\frac{in}{m(s)}s^\frac{1}{m(s)} \Big) \, \Big| \, i=0, \ldots , m(s)-1\Big\} & \text{if }m(s)\neq n,\\
\end{cases}
$$ 

$$\widehat{\CE}=
\begin{cases}
\{E \} & \text{if } m(s)=n,\\
\Big\{ F[h]/\Big(h^\frac{n}{m(s)}-q^\frac{in}{m(s)}s^\frac{1}{m(s)} \Big)  \, \Big| \, i=0, \ldots , m(s)-1\Big\}& \text{if }m(s)\neq n.\\
\end{cases}
$$   

\item Suppose that $s=0$ and $a\neq 0$. Then $$\CE /\rad (\CE)=F[x]/(x^n-a)=
\begin{cases}
F[x]/(x^n-a)\; {\rm is \; a \; field} & \text{if } m(s)=n,\\
\prod_{i=0}^{m(s)-1} F[x]/\Big(x^\frac{n}{m(s)}-q^\frac{in}{m(s)}a^\frac{1}{m(s)} \; {\rm is \; a \; direct \; product\; of \;fields}& \text{if }m(s)\neq n,\\
\end{cases}
$$ 
$$\Spec (\CE)=
\begin{cases}
\{ (h) \} & \text{if } m(s)=n,\\
\Big\{ \Big(h, x^\frac{n}{m(s)}-q^\frac{in}{m(s)}a^\frac{1}{m(s)} \Big) \, \Big| \, i=0, \ldots , m(s)-1\Big\} & \text{if }m(s)\neq n,\\
\end{cases}
$$ 

$$\widehat{\CE}=
\begin{cases}
\{F[x]/(x^n-a) \} & \text{if } m(s)=n,\\
\Big\{ F[x]/\Big(x^\frac{n}{m(s)}-q^\frac{in}{m(s)}a^\frac{1}{m(s)} \Big)  \, \Big| \, i=0, \ldots , m(s)-1\Big\}& \text{if }m(s)\neq n.\\
\end{cases}
$$   
\item Suppose that $s=0$ and $a=0$. Then $\CE/\rad (\CE) =F$, $\Spec (\CE)=\{ 0\}$,  and $\widehat{\CE}=\{ F\}$.

\item Suppose that either $s=0$ or $a=0$. Then each simple $\CE$-module $M$ is a field, and so the endomorphism algebra $\End_\CE (M)\simeq M$ is a field.
\end{enumerate}
\end{lemma}

\begin{proof}   1. By Theorem \ref{qA2Sep20}.(3b), $\CE /\rad (\CE)=E$. If $m(s)=n$ then $E$ is a field, by Theorem \ref{qA2Sep20}.(1). Suppose that $m(s)\neq n$. Then $m(s)\mid n$ and 
$$ E=F[h]/(h^n-s)=F[h]/\bigg( \Big(h^\frac{n}{m(s)}\Big)^{m(s)}-\Big( s^\frac{1}{m(s)}\Big)^{m(s)}\bigg)\simeq \prod_{i=0}^{m(s)-1} F[h]/\Big(h^\frac{n}{m(s)}-q^\frac{in}{m(s)}s^\frac{1}{m(s)}\Big)$$
is a  direct  product of fields, by Theorem \ref{qA2Sep20}.(1), and so the second equality of statement 1 follows. Now, the descriptions of the sets $\Spec (\CE)$ and $\widehat{\CE}$ follow from (\ref{SpAr}) and (\ref{SpAr1}), respectively.

2. Statement 2 follows from statement 1 by interchanging $h$ and $x$, and $s$ and $a$.

3. By Theorem \ref{qA2Sep20}.(3b), $\CE /\rad (\CE)=F$, and statement 3 follows. 

4. Statement 4 follows from the descriptions of simple $\CE$-modules in statements 1--3.  \end{proof}

{\bf Sufficient conditions for $A\in \fCK$  to be isomorphic to a matrix algebra $M_n(B)$ for some subalgebra  $B\in \fCK$.} Our goal is to prove Corollary \ref{MnK-char1} that provides sufficient conditions for an algebra $A\in \fCK$  to be isomorphic to the matrix algebra $M_n(B)$ for some subalgebra  $B\in \fCK$ of $A$. This result is used in the proof Theorem \ref{NF-qA2Sep20}.(4) where the remaining case ($s\neq 0$ and $a\neq 0$) is treated.

\begin{theorem}\label{MnK-char}%\marginpar{MnK-char}
Let $A\in \fCK$ and  $n=\Deg (A)$. Then the following statements are equivalent:
\begin{enumerate}
\item  $A\simeq M_n(K) $. 
\item $1=e_1+\cdots +e_n$ is a sum of orthogonal nonzero idempotents $e_i\in A$.
%\item There is an element $a\in A$ such that $a^n=0$ and $a^{n-1}\neq 0$.
\end{enumerate}
\end{theorem}

\begin{proof} $(1\Rightarrow 2)$ Suppose that $A=M_n(K)$. Then $A=\bigoplus_{i,j=1}KE_{ij}$ where $E_{ij}$ are the matrix units, and so   $1=E_{11}+\cdots +E_{nn}$  is a sum of orthogonal nonzero idempotents $E_{ii}\in A$.

$(2\Rightarrow 1)$ The implication follows from Proposition \ref{MnK-char2}.   \end{proof}   

Proposition \ref{MnK-char2}  presents explicit matrix units for the matrix algebra $A$ in Theorem \ref{MnK-char}.

\begin{proposition}\label{MnK-char2}%\marginpar{MnK-char2}
Let $A\in \fCK$ and $n=\Deg (A)$. Suppose that $1=e_1+\cdots +e_n$ is a sum of orthogonal nonzero idempotents $e_i\in A$. Then $A\simeq M_n(K)$ and  there is a unique set of matrix units $E_{ij}\in A$ such that 
$A=\bigoplus_{i,j=1}^nKE_{ij}$,  $E_{ij}\in e_iAe_j$ for all $i,j=1, \ldots , n$, and an explicit expression for them is  given in statement 4 below.
\begin{enumerate}
\item For all $i,j=1, \ldots , n$,  $\dim_K(e_iAe_j)=1$.
\item For all $i,j,k,l=1, \ldots , n$,  $e_iAe_j\cdot e_kAe_l=\d_{jk}e_iAe_l$ where $\d_{jk}$ is the Kronecker delta.
\item For each $i=1, \ldots , n$, there are unique elements $E_{1i}\in e_1Ae_i$ and $E_{i1}\in e_iAe_1$ such that $e_i=E_{i1}e_1$ and $e_1=E_{1i}e_i$.
\item The elements $E_{ij}:=E_{i1}E_{1j}\in e_iAe_j$ are the  matrix units.
\end{enumerate}
\end{proposition}

\begin{proof} The sum $1=e_1+\cdots +e_n$ is a sum of orthogonal nonzero idempotents $e_i\in A$.
 Hence, we have a direct sum of vector spaces $A=\bigoplus_{i,j=1}^n e_iAe_j$. 
 
 (i) {\em For all} $j=1, \ldots , n$, $\ann_A(Ae_j)=0$:  Since $A$ is a simple algebra, the annihilator of the left nonzero $A$-module $Ae_j$ is equal to zero.

(ii) {\em For all} $i,j=1, \ldots , n$, $e_iAe_j\neq 0$: Since $e_i\neq 0$, the statement (ii) follows from the statement (i).

(iii) {\em For all} $i,j=1, \ldots , n$,  $\dim_K(e_iAe_j)=1$: Notice that $\dim_K(A)=n^2$ and $A=\bigoplus_{i,j=1}^n e_iAe_j$ is a sum of $n^2$ nonzero vector spaces, and the statement (iii) follows from the statement (ii).

(iv) {\em For all} $i,j,k,l=1, \ldots , n$,  $e_iAe_j\cdot e_kAe_l=\d_{jk}e_iAe_l$: The statement (iv) follows from the statements (i) and (iii). 

(v)  {\em For each $i=1, \ldots , n$, there are unique elements $E_{1i}\in e_1Ae_i$ and $E_{i1}\in e_iAe_1$ such that $e_1=E_{1i}e_i$ and $e_i=E_{i1}e_1$; and} $E_{11}=e_1$:  Existence of the elements $E_{1i}$ and $E_{i1}$ follows from the statement (iv), and the uniqueness follows from the statement (iii). Since $e_1e_1=e_1\in e_1Ae_1$, $E_{11}=e_1$ (by the uniqueness of the element $E_{11}=e_1$).

(vi) {\em For all} $i,j=1, \ldots , n$, $E_{1i}E_{i1}=E_{11}$: By the statement (v), $e_1=E_{1i}e_i=E_{1i}E_{i1}e_1$. Hence, $E_{1i}E_{i1}=e_1$, by the uniqueness of the element $E_{11}$.

(vii) {\em The elements $E_{ij}:=E_{i1}E_{1j}\in e_iAe_j$ are the  matrix units}: For all $i,j,k,l=1, \ldots , n$,
$$E_{ij}E_{kl}=\d_{jk}E_{ij}E_{jl}=\d_{jk}E_{i1}E_{1j}E_{j1}E_{1l}=\d_{jk}E_{i1} e_1 E_{1l}=\d_{jk}E_{i1}E_{1l}=\d_{jk}E_{il}.$$

(viii) {\em The  matrix units $E_{ij}\in e_iAe_j$ are unique}: The elements $E_{i1}$, $i=1, \ldots , n$,  are unique.  For all $i,j=1, \ldots , n$, $E_{ij}E_{j1}=E_{i1}$. By the statement (iii), the elements $E_{ij}$ are unique.
\end{proof}

Corollary \ref{MnK-char3} is a criterion for a central simple algebra to be isomorphic to a matrix algebra which is given in terms of nilpotent elements. 

\begin{corollary}\label{MnK-char3}%\marginpar{MnK-char3}
Let $A\in \fCK$ and  $n=\Deg (A)$. Suppose that the  field $K$ contains a primitive $n$'th root of unity, say $q\in K$. Then the following statements are equivalent:
\begin{enumerate}
\item  $A\simeq M_n(K) $. 
\item  There is an  element  $a\in$ such that $a^n=0$ but $a^{n-1}\neq 0$. 
\end{enumerate}
\end{corollary}

\begin{proof} $(1\Rightarrow 2)$ Suppose that $A=M_n(K)$. Then $A=\bigoplus_{i,j=1}KE_{ij}$ where $E_{ij}$ are the matrix units, and so   the element $a:=E_{12}+\cdots +E_{n-1,n}$ satisfies the conditions of statement 2.

$(2\Rightarrow 1)$ Let $q\in K$ be  a primitive $n$'th root of unity. Suppose that $a$ is as in statement 2. Then the algebra $A$ contains the  subalgebra $K\langle a\rangle \simeq K[x]/(x^n-1)\simeq \prod_{i=0}^{n-1}K[x]/(x-q^i)\simeq K^n$.  Let $1=e_1+\cdots +e_n$ be the corresponding   sum of primitive orthogonal (nonzero) idempotents $e_i\in K\langle a\rangle$. By Proposition \ref{MnK-char2}, $A\simeq M_n(K)$.
\end{proof}

Corollary \ref{MnK-char1}  presents sufficient conditions  for an algebra $A\in \fCK$  to be isomorphic to the matrix algebra $M_n(B)$ for some subalgebra  $B\in \fCK$. 

\begin{corollary}\label{MnK-char1}%\marginpar{MnK-char1}
Let $A, B\in \fCK$, $n=\Deg (A)$,  $B$ is a subalgebra of $A$, $C=C_A(B)$, and $n=\frac{\Deg (A)}{\Deg (B)}$. Suppose that $1=e_1+\cdots +e_n$ is a sum of orthogonal nonzero idempotents $e_i\in C$. Then $C\simeq M_n(K)$ and $A=B\t C\simeq M_n(B)$.
%\begin{enumerate}
%\item Suppose that $1=e_1+\cdots +e_n$ is a sum of orthogonal nonzero idempotents $e_i\in C$. Then $C\simeq M_n(K)$ and $A=B\t C\simeq M_n(B)$.
%\item Suppose that there is an element $a\in C$ such that $a^n=0$ and $a^{n-1}\neq 0$. Then $C\simeq M_n(K)$ and $A=B\t C\simeq M_n(B)$.
%\end{enumerate}
\end{corollary}

\begin{proof} By \cite[Theorem 12.7.(iv)]{Pierce-AssAlg}, $A=B\t C$ and $\Deg (A)=\Deg (B)\Deg (C)$, and so $\Deg (C)=n$.  Then, by Theorem \ref{MnK-char}, $C\simeq M_n(K)$, and the corollary follows.
\end{proof}

Corollary \ref{MnK-char3}  presents another sufficient conditions  for an algebra $A\in \fCK$  to be isomorphic to the matrix algebra $M_n(B)$ for some subalgebra  $B\in \fCK$  which is given in terms of nilpotent elements. 

\begin{corollary}\label{MnK-char4}%\marginpar{MnK-char4} 

Let $A, B\in \fCK$, $n=\Deg (A)$,  $B$ is a subalgebra of $A$, $C=C_A(B)$, and $n=\frac{\Deg (A)}{\Deg (B)}$. Suppose that the  field $K$ contains a primitive $n$'th root of unity and   there is an element $a\in C$ such that $a^n=0$ and $a^{n-1}\neq 0$. Then $C\simeq M_n(K)$ and $A=B\t C\simeq M_n(B)$.

\end{corollary}

\begin{proof}   Notice that  $C\supseteq K\langle a\rangle \simeq K[x]/(x^n-1)\simeq \prod_{i=0}^{n-1}K[x]/(x-q^i)\simeq K^n$.   Let $1=e_1+\cdots +e_n$ be the corresponding   sum of primitive orthogonal (nonzero) idempotents $e_i\in K\langle a\rangle\subseteq C$. Now the result follows from  Corollary  \ref{MnK-char1}.
 \end{proof}

{\bf The case: $s\neq 0$ and   $a\neq 0$.} Let $\CE =(E, \s, a)$. In view of Theorem \ref{qA2Sep20},    it remains to consider the case when  $s\neq 0$ and  $a\neq 0$ as far the structure of the algebra $\CE$ is concerned, see Theorem \ref{NF-qA2Sep20}. By Theorem \ref{qA2Sep20}.(3), the algebra $\CE$ is a central simple $F$-algebra of degree $n$.

 The algebra $E= F[h]/(h^n-s)$ contains the algebra  
$$\CF :=F[h^{\frac{n}{m(s)}}]/(h^n-s)=F[h^{\frac{n}{m(s)}}]/\Big( \Big( h^{\frac{n}{m(s)}}\Big)^{m(s)}-\Big(s^\frac{1}{m(s)}\Big)^{m(s)}\Big).$$

The field $F$ contains the primitive $m(s)$'s root of unity $q^\frac{n}{m(s)}$. Hence,  the algebra $\CF$ is a direct product of $m(s)$ copies of the field $F$,  
 $$\CF\simeq  \prod_{i=0}^{m(s)-1} F[h^{\frac{n}{m(s)}}]/\Big( h^\frac{n}{m(s)} -q^\frac{in}{m(s)}s^\frac{1}{m(s)}\Big)\simeq F^{m(s)}.$$
 Let $1=e_0+\cdots +e_{m(s)-1}$ be the corresponding sum of orthogonal nonzero idempotents,  i.e. $e_ie_i=\d_{ij}e_i$ for all $i,j=0, \ldots, m(s)-1$ where $\d_{ij}$ is the Kronecker delta. Notice  that 
 %\marginpar{eiCF}
\begin{equation}\label{eiCF}
e_i= \frac{\prod_{j\neq i} \Big(h^{\frac{n}{m(s)}}- q^\frac{jn}{m(s)}s^\frac{1}{m(s)}\Big)}{\prod_{j\neq i}\Big(q^\frac{in}{m(s)}s^\frac{1}{m(s)}- q^\frac{jn}{m(s)}s^\frac{1}{m(s)}\Big)}  =\frac{q^\frac{in}{m(s)}\prod_{j\neq i} \Big(h^{\frac{n}{m(s)}}- q^\frac{jn}{m(s)}s^\frac{1}{m(s)}\Big)}{s^{\frac{m(s)-1}{m(s)}} \prod_{\nu =1}^{m(s)-1}\Big( 1-q^\frac{\nu n}{m(s)}\Big)} \;\; {\rm for }\;\; i=0, \ldots ,m(s)-1.
\end{equation}
Clearly,
 %\marginpar{eiCF1}
\begin{equation}\label{eiCF1}
\s  (e_i)=e_{i-1}\;\; {\rm for}\;\; i=1, \ldots , m(s)-1 \;\; {\rm and}\;\; \s (e_0)=e_{m(s)-1}.
\end{equation}
Since $\CF\subseteq E$, the algebra 
%\marginpar{eiCF2}
\begin{equation}\label{eiCF2}
E= \prod_{i=0}^{m(s)-1}Ee_i, \;\; {\rm where}\;\; Ee_i=F[h]/\Big( h^\frac{n}{m(s)} -q^\frac{in}{m(s)}s^\frac{1}{m(s)}\Big),
\end{equation}
is a direct product of cyclic field extensions (Lemma \ref{a28May23}.(3)). If $i=0$ then $Ee_0=E'$. By (\ref{eiCF1}),  the automorphism $\s\in G(E/F)$ cyclicly permutes the fields $Ee_i$, 
%\marginpar{eiCF3}
\begin{equation}\label{eiCF3}
\s  (Ee_i)=Ee_{i-1}\;\; {\rm for}\;\; i=1, \ldots , m(s)-1 \;\; {\rm and}\;\; \s (Ee_0)=Ee_{m(s)-1}.
\end{equation}
For all $i=0,\ldots , m(s)-1$,
%\marginpar{eiCF4}
\begin{equation}\label{eiCF4}
\s^{m(s)}  (Ee_i)=Ee_i\;\; {\rm and}\;\; \s^{m(s)}(he_i)=q^{m(s)}he_i.
\end{equation}
By Lemma \ref{a28May23}.(3), the restriction of the automorphism $\s^{m(s)}\in G(E/F)$ to $Ee_i$ is a generator of the Galois group $G(Ee_i/Fe_i)$. 
 By (\ref{eiCF4}), the field extension $E/F$ contains a field extension 
 $E''/F$ where
 %\marginpar{eiCF5}
\begin{equation}\label{eiCF5}
E'':=\bigg\{ (\l, \s^{-1}(\l), \ldots , \s^{-m(s)+1}(\l))\in E=\prod_{i=0}^{m(s)-1}Ee_i\, \bigg| \, \l \in Ee_0\bigg\}.
\end{equation}
The field extension $Ee_0/Fe_0\simeq E'/F$ is isomorphic to the field extension  $E''/F$ via $F$-isomorphism
%\marginpar{eiCF6}
\begin{equation}\label{eiCF6}
\phi_0: Ee_0\ra E'', \;\; \l \mapsto (\l, \s^{-1}(\l), \ldots , \s^{-m(s)+1}(\l)).
\end{equation}
In particular, $e_0\mapsto e_0+\cdots +e_{m(s)-1}=1$.

\begin{theorem}\label{NF-qA2Sep20}%\marginpar{NF-qA2Sep20}  
Let   $\CE =(E, \s, a)$. Suppose that $s\neq 0$ and  $a\neq 0$.  
\begin{enumerate}
\item The algebra  
$$\CE' :=(E'=Ee_0, \s', ae_0):=E'[X; \s']/\Big(X^\frac{n}{m(s)}-ae_0\Big)=\bigoplus_{k=0}^{\frac{n}{m(s)}-1}E'X^k$$
is a cyclic  $F$-algebra  
with $\Deg (\CE')=\frac{n}{m(s)}$ and $0\neq ae_0\in F^\times e_0$.  In particular, the centre of the algebra $\CE'$ is $F$  and the field $E'$  is  a strongly maximal subfield of the cyclic algebra $\CE'$ (i.e. $[E':F]=\frac{n}{m(s)}$).

\item The  algebra $\CE'$ is  isomorphic to the matrix algebra $M_{\frac{n}{m(s)}}(F)$ iff  $a=N_{E'/F}(b)$ for some  element $b\in {E'}^\times$.

\item The monomorphism $\phi_0 : Ee_0\ra E\subseteq \CE$ can be extended to a monomorphism $\phi_0 : \CE'\ra \CE$ by the rule $X\mapsto x^{m(s)}$, 
$$\phi_0(\CE') =E''[x^{m(s)}; \s']/\Big(\Big(x^{m(s)}\Big)^\frac{n}{m(s)}-a\Big)=E''[x^{m(s)}; \s']/(x^n-a)=\bigoplus_{k=0}^{\frac{n}{m(s)}-1}E''x^{km(s)}.$$ 
We identify, the algebra $\CE'$ with its image in $\CE$. Then  
 $\CE\simeq M_{m(s)}(\CE')$ and $C_{\CE}(\CE')\simeq M_{m(s)}(F)$.

\item  $\CE=\bigoplus_{i=0}^{m(s)-1}\bigg( \prod_{j=0}^{m(s)-1}\CE'e_j\bigg)x^i$ where the algebras $\CE'e_j=Ee_j[x^{m(s)}; \s']/(x^n-ae_j)$, $j=0, \ldots , m(s)-1$, are isomorphic to the cyclic algebra $\CE'$ and the inner automorphism $\o_x$ of $\CE$ cyclicly permutes them $\o_x(\CE'e_j)=\CE'e_{j-1}$.

\end{enumerate}
\end{theorem}

\begin{proof}   1. By Lemma \ref{a28May23}.(2) and Theorem \ref{qA2Sep20}.(3,4), the algebra $\CE'$ is  a cyclic  $F$-algebra   with $\Deg (\CE')=\frac{n}{m(s)}$. 

2. Statement 2 follows from statement 1 and  Theorem  \ref{qA2Sep20}.(5).

3. The monomorphism $\phi_0 : Ee_0\ra E\subseteq \CE$ can be extended to a monomorphism $\phi_0 : \CE'\ra \CE$ by the rule $X\mapsto x^{m(s)}$ since the algebra $\CE'$ is a simple algebra,  $\im (\CE')\neq 0$, and the extension respects the defining relations of the algebra $\CE'$ ($X^\frac{n}{m(s)} -ae_0\mapsto \Big(x^{m(s)}\Big)^\frac{n}{m(s)} -a1=x^n-a=0$). Then the equalities in statement 3 follow. 

Recall that $\CE, \CE'\in \fC (F)$ (Theorem \ref{qA2Sep20}.(3) and  statement 1). Then, 
by \cite[Theorem 12.7.(iv)]{Pierce-AssAlg}, 
$$\CE=\CE'\t_F C\;\; {\rm  where}\;\;C=C_\CE (\CE')\;\; {\rm  is \; the\; centralizer\; of\; the\; algebra}\;\;\CE'\;\; {\rm  in}\;\;\CE.$$
 It follows from the equality  $\Deg (\CE)=\Deg (\CE')\Deg (C)$  that  $\Deg (C)=\frac{\Deg (\CE)}{\Deg (\CE')}=\frac{nm(s)}{n}=m(s)$ (Theorem \ref{qA2Sep20}.(3) and the statement 1).  
By (\ref{eiCF1}),  the $m(s)=\Deg (C)$ idempotents $e_0, \ldots , e_{m(s)-1}$ belong to the centralizer $C$. Now,   by Corollary \ref{MnK-char1}, $$C\simeq M_{m(s)}(F)\;\; {\rm and}\;\;  \CE= \CE'\t_F C=\CE'\t_FM_{m(s)}(F)\simeq M_{m(s)}(\CE').$$ 

4. $\CE = \bigoplus_{i=0}^{n-1}Ex^i= \bigoplus_{i=0}^{m(s)-1} \bigg( \bigoplus_{k=0}^{\frac{n}{m(s)}-1}\prod_{j=0}^{m(s)-1}Ee_j \cdot x^{km(s)}\bigg) x^i= \bigoplus_{i=0}^{m(s)-1}\bigg( \prod_{j=0}^{m(s)-1}\CE'e_j\bigg)x^i$. 
  Clearly, $\CE'e_j:=\phi_0(\CE')e_j =Ee_j[x^{m(s)}; \s']/(x^n-ae_j)$, $j=0, \ldots , m(s)-1$. By statement 1, $\CE'e_0\simeq \CE'$. By (\ref{eiCF3}), 
  $$\o_x(\CE'e_j)=\s (Ee_j)[x^{m(s)}; \s']/(x^n-a\s(e_j))=Ee_{j-1}[x^{m(s)}; \s']/(x^n-ae_{j-1})=\CE_j'e_{j-1}.$$ Hence, 
the algebras $\CE'e_j$ are isomorphic to the cyclic algebra $\CE'e_0\simeq \CE'$.
  \end{proof}  

By interchanging the roles of $h$ and $X$, the cyclic algebra $\CE'=(E',\s',a)$ can be written as a cyclic algebra 

%\marginpar{CEpXh}
\begin{equation}\label{CEpXh}
\CE'=(\tilde{E} ,\tau , s^\frac{1}{m(s)})=\tilde{E}[h; \tau]/\Big(h^\frac{n}{m(s)}-s^\frac{1}{m(s)}\Big)=\bigoplus_{k=0}^{\frac{n}{m(s)}-1}\tilde{E}h^k
\end{equation}
where $\tilde{E}:= F[X]/(X^\frac{n}{m(s)}-a)$ is an $F$-algebra which is not necessarily a field and $\tau (X)=q^{-m(s)}X$. Our next goal is to write down analogues of Lemma \ref{a28May23} and Theorem \ref{NF-qA2Sep20} but for the cyclic algebra $\CE'$ 
in (\ref{CEpXh}), see Lemma \ref{Xa28May23} and Corollary \ref{Xa27Jun23}, respectively. 

\begin{lemma}\label{Xa28May23}%\marginpar{Xa28May23}
%Let $F$, $E=E(s)$,   and $q$ be as above (in particular,  $q\in F$ is a primitive $n$'th root of unity). Suppose $0\neq s\in \bigcup_{m\mid n, m\neq 1}F^{[m]}$. Then 
\begin{enumerate}
%\item  $m(s)=\lcm \{m \, |  \,  m\geq 1,  m\mid n,  s\in F^{[m]} \}$. In particular,  $m'\mid m(s)$ for all $m'\geq 1$ such  that  $ m'\mid n$   and   $s\in  F^{[m']}$;  and $m(s;\frac{n}{m(s;n)})=1$. So, the number $m(s)$ is the largest divisor of $n$  such that $s^\frac{1}{m(s)}\in F$.

\item  The field extension $\tilde{E}'/F$, where $\tilde{E}':=F[X]/\Big(X^\frac{n}{m(s)m(s,a)}-a^\frac{1}{m(s,a)}\Big)$,  is a cyclic  field extension with Galois  group $G(\tilde{E}'/F)=\langle \tau':=\tau^{m(s,a)} \, | \, {\tau'}^\frac{n}{m(s)m(s,a)}=1\rangle$  which is a cyclic group of order $\frac{n}{m(s)m(s,a)}$ where $\tau' (X) =\tau^{m(s,a)}(X)= q^{-m(s)m(s,a)}X$, and $[\tilde{E}':F]=\frac{n}{m(s)m(s,a)}$. 

\item  The factor algebras  $F[X]/\Big(X^\frac{n}{m(s)m(s,a)}-q^\frac{in}{m(s)m(s,a)}  a^\frac{1}{m(s,a)}\Big)$, where $i=0,1, \ldots , m(s)m(s,a)-1$,  are $F$-isomorphic cyclic field extensions with Galois groups equal to  $G(\tilde{E}'/F)=\langle \tau':=\tau^{m(s,a)} \, | \, {\tau'}^\frac{n}{m(s)m(s,a)}=1\rangle$. 
\end{enumerate}
\end{lemma}

\begin{proof}   By Lemma \ref{a28May23}.(1), $m(s;\frac{n}{m(s;n)})=1$, and   the corollary follows from Lemma \ref{a28May23}(2,3).   \end{proof}

 The algebra $\tilde{E}:= F[X]/(X^\frac{n}{m(s)}-a)$ contains the algebra  
$$\tilde{\CF} :=F[X^{\frac{n}{m(s)m(s,a)}}]/(X^\frac{n}{m(s)}-a)=F[X^{\frac{n}{m(s)m(s,a)}}]/\bigg( \Big( X^\frac{n}{m(s)m(s,a)}\Big)^{m(s,a)}-\Big(a^\frac{1}{m(s,a)}\Big)^{m(s,a)}\bigg).$$

The field $F$ contains the primitive $m(s,a)$'s root of unity $q^\frac{n}{m(s,a)}$ (since $m(s)m(s,a)\mid n$). Hence,  the algebra $\tilde{\CF}$ is a direct product of $m(s,a)$ copies of the field $F$,  
 $$\tilde{\CF}\simeq  \prod_{i=0}^{m(s,a)-1} F[X^{\frac{n}{m(s)m(s,a)}}]/\Big( X^{\frac{n}{m(s)m(s,a)}} -q^\frac{in}{m(s,a)}a^\frac{1}{m(s,a)}\Big)\simeq F^{m(s,a)}.$$
 Let $1=\tilde{e}_0+\cdots +\tilde{e}_{m(s,a)-1}$ be the corresponding sum of orthogonal nonzero idempotents. Notice  that 
 %\marginpar{XeiCF}
\begin{equation}\label{XeiCF}
\tilde{e}_i= \frac{\prod_{j\neq i} \Big(X^{\frac{n}{m(s)m(s,a)}} -q^\frac{jn}{m(s,a)}a^\frac{1}{m(s,a)}\Big)}{\prod_{j\neq i}\Big(q^\frac{in}{m(s,a)}a^\frac{1}{m(s,a)}- q^\frac{jn}{m(s,a)}a^\frac{1}{m(s,a)}\Big)} \;\; {\rm for }\;\; i=0, \ldots ,m(s,a)-1.
\end{equation}
It follows from the equality $\tau (X)=-q^{m(s)}X$ that  
 %\marginpar{XeiCF1}
\begin{equation}\label{XeiCF1}
\tau (\tilde{e}_{m(s,a)-1})=\tilde{e}_0 \;\; {\rm and}\;\;  \tau  (\tilde{e}_i)=\tilde{e}_{i+1}\;\; {\rm for}\;\; i\in \{ 0, \ldots , m(s,a)-1\}\backslash \{  m(s,a)-1\} .
\end{equation}
Since $\tilde{\CF}\subseteq \tilde{E}$, 
%\marginpar{XeiCF2}
\begin{equation}\label{XeiCF2}
\tilde{E}= \prod_{i=0}^{m(s,a)-1}\tilde{E}\tilde{e}_i, \;\; {\rm where}\;\; \tilde{E}\tilde{e}_i=F[X]/\Big( X^{\frac{n}{m(s)m(s,a)}} -q^\frac{in}{m(s,a)}a^\frac{1}{m(s,a)} \Big),
\end{equation}
is a direct product of cyclic field extensions (Lemma \ref{Xa28May23}.(2)). If $i=0$ then $\tilde{E}\tilde{e}_0=\tilde{E}'$. By (\ref{XeiCF1}),  the automorphism $\tau\in G(\tilde{E}/F)$ cyclicly permutes the fields $\tilde{E}\tilde{e}_i$, 
%\marginpar{XeiCF3}
\begin{equation}\label{XeiCF3}
\tau (\tilde{E}\tilde{e}_{m(s,a)-1})=\tilde{E}\tilde{e}_0 \;\; {\rm and}\;\;  \tau  (\tilde{E}\tilde{e}_i)=\tilde{E}\tilde{e}_{i+1}\;\; {\rm for}\;\; i\in \{ 0, \ldots , m(s,a)-1\}\backslash \{  m(s,a)-1\}.
\end{equation}
For all $i=0,\ldots , m(s,a)-1$,
%\marginpar{XeiCF4}
\begin{equation}\label{XeiCF4}
\tau^{m(s,a)}  (\tilde{E}\tilde{e}_i)=\tilde{E}\tilde{e}_i\;\; {\rm and}\;\; \tau^{m(s,a)}(X\tilde{e}_i)=q^{-m(s)m(s,a)}X\tilde{e}_i.
\end{equation}
By Lemma \ref{Xa28May23}.(2), the restriction of the automorphism $\tau^{m(s,a)}\in G(\tilde{E}/F)$ to the field $\tilde{E}\tilde{e}_i$ is a generator of the Galois group $G(\tilde{E}\tilde{e}_i/F\tilde{e}_i)$. 
 By (\ref{XeiCF4}), the field extension $\tilde{E}/F$ contains a field extension 
 $\tilde{E}''/F$ where
 %\marginpar{XeiCF5}
\begin{equation}\label{XeiCF5}
\tilde{E}'':=\bigg\{ (\l, \tau (\l), \ldots , \tau^{m(s,a)-1}(\l))\in \tilde{E}=\prod_{i=0}^{m(s,a)-1}\tilde{E}\tilde{e}_i\, \bigg| \, \l \in \tilde{E}\tilde{e}_0\bigg\}.
\end{equation}
The field extension $\tilde{E}\tilde{e}_0/F\tilde{e}_0\simeq \tilde{E}'/F$ is isomorphic to the field extension  $\tilde{E}''/F$ via $F$-isomorphism
%\marginpar{XeiCF6}
\begin{equation}\label{XeiCF6}
\tilde{\phi}_0: \tilde{E}\tilde{e}_0\ra \tilde{E}'', \;\; \l \mapsto (\l, \tau (\l), \ldots , \tau^{m(s.a)-1}(\l)).
\end{equation}
In particular, $\tilde{e}_0\mapsto \tilde{e}_0+\cdots +\tilde{e}_{m(s,a)-1}=1$.

\begin{definition}
The cyclic algebra $\CE =(E(s), \s, a)$, where $s\neq 0$ and  $a\neq 0$, is called a {\bf reduced cyclic algebra} if the subalgebras $E(s)=F[h]/(h^n-s)$ and $E(a)=F[x]/(x^n-a)$ of $\CE$ are fields. 
\end{definition}

By Theorem \ref{qA2Sep20}.(1),  the cyclic algebra $\CE =(E(s), \s, a)$, where $s\neq 0$ and  $a\neq 0$, is a reduced cyclic algebra iff $m(s)=1$ and $m(a)=1$. Corollary \ref{Xa27Jun23}.(5) shows that 
 $\CE \simeq M_{m(s)m(s,a)}(\tilde{\CE}'')$ for an explicit  {\em reduced} cyclic algebra $\tilde{\CE}''$.

\begin{corollary}\label{Xa27Jun23}%\marginpar{Xa27Jun23}
Suppose that $s\neq 0$ and  $a\neq 0$.  Recall that    $\CE' =(E', \s', a)=(\tilde{E}, \tau , s^\frac{1}{m(s)})$, see (\ref{CEpXh}). 
\begin{enumerate}
\item The algebra  
$$\tilde{\CE}'' :=(\tilde{E}', \tau', s^\frac{1}{m(s)}):=\tilde{E}'[H; \tau']/\Big(H^\frac{n}{m(s)m(s,a)}-s^\frac{1}{m(s)}\Big)=\bigoplus_{k=0}^{\frac{n}{m(s)m(s,a)}-1}\tilde{E}'H^k$$
is a reduced cyclic  $F$-algebra  
with $\Deg (\tilde{\CE}'')=\frac{n}{m(s)m(s,a)}$ and $s^\frac{1}{m(s)}\in F^\times $.  In particular, the centre of the algebra $\tilde{\CE}''$ is $F$  and the field $\tilde{E}'$  is  a strongly maximal subfield of the cyclic algebra $\tilde{\CE}''$ (i.e. $[\tilde{E}':F]=\frac{n}{m(s)m(s,a)}$).

\item The  algebra $\tilde{\CE}''$ is  isomorphic to the matrix algebra $M_{\frac{n}{m(s)m(s,a)}}(F)$ iff  $a=N_{\tilde{E}'/F}(b)$ for some  element $b\in (\tilde{E}')^\times$.

\item The monomorphism $\tilde{\phi}_0 : \tilde{E}'\ra \tilde{E}\subseteq \CE'$ can be extended to a monomorphism $\tilde{\phi}_0 : \tilde{\CE}''\ra \CE'$ by the rule $H\mapsto h^{m(s,a)}$,  
$$\tilde{\phi}_0(\tilde{\CE}'') =  \tilde{\CE}''[h^{m(s,a)}; \tau']/\Big(\Big(h^{m(s,a)}\Big)^\frac{n}{m(s)m(s,a)}-s^\frac{1}{m(s)}  \Big)=\tilde{\CE}''[h^{m(s,a)}; \tau']/(h^\frac{n}{m(s)}-s^\frac{1}{m(s)})=\bigoplus_{k=0}^{\frac{n}{m(s)m(s,a)}-1}\tilde{\CE}''h^{km(s,a)}.$$ 
We identify, the algebra $\tilde{\CE}''$ with its image in $\CE'$. Then  
 $\CE'\simeq M_{m(s,a)}(\tilde{\CE}'')$ and $C_{\CE'}(\tilde{\CE}'')\simeq M_{m(s,a)}(F)$.
 
\item  $\CE'=\bigoplus_{i=0}^{m(s,a)-1}\bigg( \prod_{j=0}^{m(s,a)-1}\tilde{\CE}''\tilde{e}_j\bigg)h^i$.

 \item $\CE \simeq M_{m(s)m(s,a)}(\tilde{\CE}'')$. In particular, the algebra $\CE$ is equivalent to the reduced cyclic algebra $\tilde{\CE}''$. 

\end{enumerate}
\end{corollary}

\begin{proof}   1--4. Statements 1--4 of the  corollary are precisely  Theorem \ref{NF-qA2Sep20} but for the algebra $\CE'=(\tilde{E}, \tau , s^\frac{1}{m(s)})$ rather than $\CE =(E,\s , a)$.

5. By Theorem \ref{NF-qA2Sep20}.(3) and statement 4, $\CE \simeq M_{m(s)}(\CE')$ and $\CE' \simeq M_{m(s,a)}(\tilde{\CE}'')$. Hence, $\CE \simeq M_{m(s)}( M_{m(s,a)}(\tilde{\CE}''))\simeq M_{m(s)m(s,a)}(\tilde{\CE}'')$.    \end{proof}

Applying Corollary \ref{Xa27Jun23}.(5) to the algebra $\CE$ where the roles $x$ and $h$ are reversed we see that $\CE \simeq M_{m(a)m(a,s)}(\tilde{\CE}''')$ for an explicit {\em reduced} cyclic algebra $\tilde{\CE}'''$. Hence, 
%\marginpar{CEMms}
\begin{equation}\label{CEMms}
\CE \simeq M_{m(s)m(s,a)}(\tilde{\CE}'') \simeq M_{m(a)m(a,s)}(\tilde{\CE}''')
\end{equation}
and so the reduced cyclic algebras $\tilde{\CE}''$ and $\tilde{\CE}'''$ are equivalent.

 Corollary \ref{a17Jul23}.(2) gives a sufficient condition for $\CE\simeq M_m(F)$. Proposition \ref{c17Jun23} is a criterion for $\CE\simeq M_m(F)$.

\begin{corollary}\label{a17Jul23}%\marginpar{a17Jul23}
Let   $\CE =(E, \s, a)$. Suppose that $s\neq 0$ and  $a\neq 0$.  Then

\begin{enumerate}
\item $d\mid \gcd\Big(\frac{n}{m(s)m(s,a)}, \frac{n}{m(a)m(a,s)}\Big)$.

\item If $\gcd\Big(\frac{n}{m(s)m(s,a)}, \frac{n}{m(a)m(a,s)}\Big)=1$ then $\CE\simeq M_m(F)$ for some $m\geq 1$.
\end{enumerate}

\end{corollary}

\begin{proof}   1. It follows from (\ref{CEMms}) and the inclusions $\tilde{\CE}'',\tilde{\CE}'''\in \fC (F)$ that $\tilde{\CE}''\simeq M_k(D)$ and $\tilde{\CE}'''\simeq M_l(D)$ for some natural numbers $k$ and $l$. Now, by (\ref{CEMms}),
$$\CE \simeq M_{m(s)m(s,a)k}(D) \simeq M_{m(a)m(a,s)l}(D),\;\; {\rm and \; so} $$
$$n=\Deg (\CE)=\Deg(M_{m(s)m(s,a)k}(D) )=
m(s)m(s,a)kd= \Deg(M_{m(a)m(a,s)l}(D))=m(a)m(a,s)ld$$
and statement 1 follows.

2. By Theorem \ref{NF-qA2Sep20}.(1),   $\CE \simeq M_m(D)$ for some $m\geq 1$.   By statement 1, $\Deg (D)=d=1$, i.e. $D=F$, and so $\CE\simeq M_m(F)$.    \end{proof}  

{\bf The $\CE$-module structure on a unique simple left $\CE$-module $U$ of $\CE$.} Let   $\CE =(E, \s, a)$. Suppose that $s\neq 0$, $a\neq 0$,   and the algebra  $E$ is a field. By  Theorem \ref{qA2Sep20}.(3a), $\CE\simeq M_m(D)$ for some central division algebra $D$ and $m\geq 1$. Clearly, 
$$n=\Deg (\CE)=\Deg ( M_m(D))=m\Deg (D).$$
Let $U$ be a unique simple (left) $\CE$-module (up to isomorphism) and ${}_{E}U\simeq E^d$ for some natural number $d\geq 1$. It follows from $\CE\simeq M_m(D)$ that  ${}_\CE\CE\simeq U^m$ and  ${}_DU\simeq D^m$. Since ${}_EU\simeq   E^d$, ${}_E\CE\simeq E^{dm}$, and so
$ n^2=[\CE:F]=[\CE:E][E:F]=dmn$.  Hence  
%\marginpar{n=md=dD}
\begin{equation}\label{n=md=dD}
n=md\;\; {\rm  and}\;\; d=\Deg (D)
\end{equation}
(since $n=m\Deg (D)$). Let 
$U=\bigoplus_{i=1}^dEe_i$ where $e=(e_1, \ldots , e_d)^t$ is an $E$-basis for the left $E$-module $U$ (where $t$ stands for `transpose'). Every element $\vec{u}=\sum_{i=1}^du_ie_i=ue$ of $U$ is a unique sum where the row-vector $u=(u_1, \ldots , u_d)\in L^d$ is  the coordinates of $\vec{u}$ in  the basis $e$. Using the defining relations of the algebra $\CE= E[x; \s]/(x^n-a)$, we see that the left $\CE$-module structure on the $E$-module $U=E^d$ is given by a matrix $\mX =(x_{ij})\in M_d(E)$ such that for $i=1, \ldots , d$,  $\l \in E$, and $\vec{u}\in U$, 
$$ xe_i=\sum_{j=1}^d x_{ij}e_j, \;\;
x\l \vec{u} =\s (\l) x\vec{u}, \;\; {\rm and}\;\; x^n=a.$$
The last two equalities can be written in the matrix form as follows:
%\marginpar{xu=usxe}
\begin{equation}\label{xu=usxe}
x\l\vec{u}=\s (\l) u^\s\mX e\;\; {\rm and}\;\; \mX^{\s^{n-1}}\cdots  \mX^\s \mX=a
\end{equation}
where $u^\s :=(\s (u_1), \ldots , \s (u_d))$ and $\mX^{\s^k}:=(\s^k(x_{ij}))$. Since $a\neq 0$, i.e. $a\in F^\times$,  $\mX\in \GL_{d'}(E)$.

\begin{definition}
 Let   $\CE =(E, \s, a)$. Suppose that $s\neq 0$, $a\neq 0$,   and the algebra  $E$ is a field. By  Theorem \ref{qA2Sep20}.(3a), $\CE\simeq M_m(D)$ for some central division algebra $D$ and $m\geq 1$. The natural number $m(\CE):=m$  is called the {\bf matrix size} of $\CE$,  $D(\CE):=D$ is called the {\bf division algebra} of $\CE$, the unique simple $\CE$-module $U(\CE):=U={}_EE^d$, where $d=\Deg (D(\CE))$,  is called the {\bf simple $\CE$-module}, and the matrix $\mX\in\GL_d(E)$ is called the {\bf $\CE$-module structure matrix} for $U(\CE)$.
\end{definition}

\begin{definition}
For each natural number $i\geq 1$, the map 
%\marginpar{xu=usxe4}
\begin{equation}\label{xu=usxe4}
N_{M_d(E)}^{\s , i} : M_d(E)\ra M_d(E), \;\; \mY\mapsto \mY^{\s^{i-1}}\cdots  \mY^\s \mY
\end{equation}
is called the {\bf matrix $(\s, i)$-norm} of $\mY$. 
  For each matrix $\L\in \GL_d(E)$, the map 
 %\marginpar{xu=usxe1}
\begin{equation}\label{xu=usxe1}
c^\s_\L: M_d(E)\ra M_d(E), \;\; \mY\mapsto \L^\s \mY\L^{-1}
\end{equation}
 is called a $\s$-{\bf conjugation}. If $c^\s_\L (\mY)=\mY$ then the matrix $\mY$ is called $c_\L^\s$-{\bf invariant}. 
\end{definition} 
  Notice that 
%\marginpar{xu=usxe2}
\begin{equation}\label{xu=usxe2}
\det (c^\s_\L (\mY ))=\frac{\s(\det(\L))}{\det(\L)}\det(\mY)\;\; {\rm and}\;\;\det (c^\s_{\L^{\s^{n-1}}\cdots \L^\s\L }(\mY))=\det(\mY).
\end{equation} 
The second equality follows from the first and the fact that $\s^n=1$. 

Proposition \ref{A21Jun23}.(1) shows that there is a direct connection between the norms  $N_{M_d(E)}^{\s, n}$ and $N_{E/F}$. Proposition \ref{A21Jun23}.(2) is about a `$\s$-multiplicative property' of the maps $c_\L^{\s^i}$ and  Proposition \ref{A21Jun23}.(3) gives a method of producing $c_\L^{\s^i}$-invariant matrices from $c_\L^{\s}$-invariant ones.

\begin{proposition}\label{A21Jun23}%\marginpar{A21Jun23}

\begin{enumerate}
\item $\det \, N_{M_d(E)}^{\s, n}=N_{E/F}\, \det$,  i.e. for all matrices $\mY \in M_d(E)$, 
 $$\det \Big(\mY^{\s^{n-1}}\cdots  \mY^\s \mY \Big)=\det(\mY)^{\s^{n-1}}\cdots  \det(\mY )^\s \det (\mY ).$$
\item  $c_\L^{\s^i}\Big( \mY_{i-1}^{\s^{i-1}}\cdots  \mY_1^\s \mY_0\Big)= c_\L^{\s}(\mY_{i-1})^{\s^{i-1}}\cdots  c_\L^{\s}(\mY_1)^\s c_\L^{\s}(\mY_0)$ for all $i\geq 1$, $\L \in \GL_d(E)$, and $\mY_0, \ldots \mY_{i-1}\in M_d(E)$.
\item  $c_\L^{\s^i}\Big( \mY_{i-1}^{\s^{i-1}}\cdots  \mY_1^\s \mY_0\Big)= \mY_{i-1}^{\s^{i-1}}\cdots  \mY_1^\s \mY_0$ for all $i\geq 1$, $\L \in \GL_d(E)$, and $\mY_0, \ldots \mY_{i-1}\in M_d(E)$ such that $c_\L^\s (Y_j)=Y_j$ for $j=0, \ldots , i-1$.
\end{enumerate}
\end{proposition}

\begin{proof}   1. $\det \Big(\mY^{\s^{n-1}}\cdots  \mY^\s \mY \Big)= \det (\mY^{\s^{n-1}})\cdots  \det (\mY^\s ) \det (\mY )   =\det(\mY)^{\s^{n-1}}\cdots  \det(\mY )^\s \det (\mY )$. 

2. \begin{eqnarray*}
c_\L^{\s^i}\Big( \mY_{i-1}^{\s^{i-1}}\cdots  \mY_1^\s \mY_0\Big)
&=& \L^{\s^i} \mY_{i-1}^{\s^{i-1}} (\L^{-1})^{\s^{i-1}}\cdot \L^{\s^{i-1}} \mY_{i-2}^{\s^{i-2}} \Big( \L^{-1}\Big)^{\s^{i-2}}   \cdots   \L^{\s}\mY_0\L^{-1}\\
&=& \Big(\L^\s \mY_{i-1} \L^{-1}\Big)^{\s^{i-1}}\cdot \Big(\L^\s \mY_{i-2} \L^{-1}\Big)^{\s^{i-2}}    \cdots   \L^{\s}\mY_0\L^{-1}\\
&=& c_\L^{\s}(\mY_{i-1})^{\s^{i-1}}\cdots  c_\L^{\s}(\mY_1)^\s c_\L^{\s}(\mY_0).
\end{eqnarray*}
3. Statement 3 follows from statement 2.  \end{proof}

We write endomorphisms on the {\em opposite side} of the scalars  in order not to deal with the opposite rings. So, in our situation we  write endomorphisms on the right.  But it is customary to write an {\em individual} endomorphism on the {\em left}. So, for each matrix $\L \in M_d(E)$, the map $$ \L : E^d\ra E^d, \;\; \vec{u}=ue\mapsto \vec{u}\L :=u\L e$$
is an element of $\End_E(E^d)$. If $\L \in \GL_d(E)$ and $e':=\L e$ is another choice of the $E$-basis for the left $E$-module $U=E^d$ then by 
 the first equality in (\ref{xu=usxe}), $x\vec{u}=u^\s \mX e$, the module structure on $U$ (that is determined by the matrix $\mX$ in the basis $e$)  is determined by the matrix 
 %\marginpar{xu=usxe3}
\begin{equation}\label{xu=usxe3}
\mX':=\L^\s \mX \L^{-1}. 
\end{equation} 
 in the basis $e'$ (since $x\vec{u}=u^\s \mX e=u^\s  (\L^\s)^{-1}\cdot \L^\s \mX \L^{-1}\cdot \L e=(u\L^{-1})^\s \cdot \L^\s \mX \L^{-1}\cdot \L e={u'}^\s \mX'e'$ where the row-vector $u'=u\L^{-1}$ is the coordinates of the element $\vec{u}$ in the basis $e'$). 
 
 As a result,  we have proven statement 1 of Lemma \ref{a14Jun23}.

\begin{lemma}\label{a14Jun23}%\marginpar{a14Jun23}
Let   $\CE =(E, \s, a)$. Suppose that $s\neq 0$, $a\neq 0$,   and the algebra  $E$ is a field. By  Theorem \ref{qA2Sep20}.(3a), $\CE\simeq M_m(D)$ for some central division algebra $D$ and $m\geq 1$. Let $U$ be a unique simple (left) $\CE$-module and ${}_{E}U\simeq E^d$ for some natural number $d\geq 1$. Then  
\begin{enumerate}
\item $ n=dm$, $d=\Deg (D)$,  and the left $\CE$-module structure on the left 
 $E$-module $U=E^d$ is given by a matrix $\mX\in M_d(E)$ that satisfies (\ref{xu=usxe}), and vice versa. Furthermore,  $\mX\in \GL_{d}(E)$. 
\item The matrix $\mX$ is unique up to $\s$-conjugation.
\end{enumerate}
\end{lemma}

\begin{proof}   %It remains to show that $n=dm$.  
%It follows from $\CE\simeq M_m(D)$ that  ${}_\CE\CE\simeq U^m$ and  ${}_DU\simeq D^m$. Since ${}_EU\simeq   E^d$, ${}_E\CE\simeq E^{dm}$, and so
%$ n^2=[\CE:F]=[\CE:E][E:F]=dmn$.  Hence  
%$n=dm$.
%1. Statement 1 has already been proven.
2. Statement 2 follows from the uniqueness of the simple $\CE$-module and (\ref{xu=usxe3}).    \end{proof}  

{\bf Explicit descriptions of the division algebra $D$, its degree $\Deg (D)$, and $m$ for the algebra $\CE\simeq M_m(D)$.} For each natural number $d\geq     1$, the automorphism $\s \in G(E/F)$ is extended  to an automorphism of the matrix algebra $M_d(E)$ by the rule: For all matrices $\mY =(y_{ij})\in M_d(E)$, $\mY^\s :=\s (\mY):=(\s (y_{ij}))$. Each matrix $\mX\in GL_d(E)$, determines the inner automorphism $\o_\mX$ of the algebra $M_d(E)$: For all matrices $\mY \in M_d(E)$, $\o_\mX (\mY ) :=\mX\mY \mX^{-1}$. In particular, the fixed algebra of the  automorphism  $\o_{\mX^{-1}}\s\in \Aut_F(M_d(E))$, $M_d(E)^{\o_{\mX^{-1}}\s }=\{ \mY\, | \, \mY^{\s \o_{\mX^{-1}}}=\mY\}$, is isomorphic to the division   ring $D$ in Lemma \ref{a14Jun23} (Theorem \ref{12Jun23}.(2)). By the Skolem-Noether Theorem, the map 
 $$  \GL_d(E)/E^\times \ra \Inn (M_d(E)), \;\; E^\times \mY\mapsto \o_{E^\times \mY} :=\o_{\mY}$$
  is a group isomorphism. Similarly, for the central division algebra $D$ (Lemma \ref{a14Jun23}), $$\Aut_K(D)=\Inn (D)\simeq D^\times / F^\times.$$
  Theorem \ref{12Jun23} describes the algebras $\CE \simeq M_m(D)$ and $D$, i.e. the numbers $m$ and $d=\Deg (D)$ are described, the division ring $D$ is realized as an explicit  subalgebra of the matrix algebra $M_d(E)$   where $d=\Deg (D)$, and the automorphism group $\Aut_F(D)$ is represented as a factor group of two explicit subgroups of $\Inn (M_d(E))$.

\begin{theorem}\label{12Jun23}%\marginpar{12Jun23}
Let   $\CE =(E, \s, a)$. Suppose that $s\neq 0$, $a\neq 0$,   and the algebra  $E=E(s)=F[h]/(h^n-s)$ is a field. 
%and $E(a):=F[x]/(x^n-a)$ are fields. 
By  Theorem \ref{qA2Sep20}.(3a), $\CE\simeq M_m(D)$ for some central division algebra $D$, $m\geq 1$, and   $n=m\Deg (D)$. Let $U$ be a unique simple (left) $\CE$-module and ${}_{E(s)}U\simeq E(s)^d$ where  $d=\Deg (D)$ (Lemma \ref{a14Jun23}.(1)). Then 
\begin{enumerate}
%\item $d=\Deg (D)$ and $n=md$.

\item $d=\min \{ d'\geq 1\, | \, d'\mid n$, there is a matrix $\mX\in M_{d'}(E(s))$ such that $ \mX^{\s^{n-1}}\cdots  \mX^\s \mX=a\}$ ($\mX\in \GL_{d'}(E(s))$ since $a\neq 0$) and $m=\frac{n}{d}$.  

\item The division algebra $D=\{\L \in M_d(E(s))\, | \, \L^\s = \mX \L \mX^{-1}\}$ is an  $F$-subalgebra of the matrix algebra $M_d(E(s))$ which is isomorphic to the fixed algebra $M_d(E(s))^{\o_{\mX^{-1}}\s }$ of the automorphism $\o_{\mX^{-1}}\s\in \Aut_F(M_d(E(s)))$.

\item  The automorphism group $\Aut_F(D)$ of the division algebra $D$ is isomorphic the factor group $\CN (D)/\CC (D)$ where $\CN (D)=\{ \o_{E^\times \mY}\, | \, E^\times \mY\in \GL_d(E)/E^\times, \o_{E^\times \mY}(D)=D\}$ and $\CC (D)=\{ \o_{E^\times \mY}\, | \, E^\times \mY\in \GL_d(E)/E^\times, \o_{E^\times \mY}(\L)=\L$ for all $\L \in D\}$ are subgroups of $\Inn (M_d(E))$.

\end{enumerate}
\end{theorem}

\begin{proof}   %1. Recall that  $n=dm$ and $d=\Deg (D)$  (Lemma \ref{a14Jun23}.(1)).
 1. Any simple finitely generated $\CE$-module is a finite direct sum of copies of the simple $\CE$-module $U$. Now, statement 1 follows from Lemma \ref{a14Jun23}.

2. Recall that ${}_EU\simeq   E^d$. Then $D\simeq \End_\CE (U)\subseteq \End_E (U)=\End_E(E^d)\simeq M_d(E)$.  The division algebra $D$ is a subalgebra of the matrix algebra $M_d(E)$. We write endomorphisms on the right. So, for each matrix $\L \in M_d(E)$, the map $$ \L : E^d\ra E^d, \;\; \vec{u}=ue\mapsto \vec{u}\L :=u\L e$$
is an element of $\End_E(E^d)$. A matrix $\L \in M_d(D)$ belongs to $D$ iff $x (\vec{u}\L )= (x\vec{u})\L$ for all elements $\vec{u}\in U$ iff 
$u^\s \L^\s \mX e= u^\s \mX \L e$ for all elements $\vec{u}\in U$ iff $\L^\s \mX =  \mX \L $ iff 
$\L^\s = \mX \L \mX^{-1}$ since $\mX\in \GL_d(E)$ (Lemma \ref{a14Jun23}), and statement 2 follows. 

3. Recall that $\Aut_F(D)=\Inn (D)\simeq D^\times / F^\times$ and $ \Aut_E(M_d(E))=\Inn (M_d(E))\simeq \GL_d(E)/E^\times$. Since $D$ is an  $F$-subalgebra  of $M_d(E)$, $D^\times \subseteq M_d(E)^\times$,  and so 
$$\Aut_F(D)=\Inn (D)\subseteq \Inn (M_d(K))=\Aut_E(M_d(E)).$$
Now, statement 3 follows from statement 2. 
  \end{proof}   

Theorem \ref{12Jun23}.(2) gives an efficient  way of finding an $F$-basis of the division ring $D$ by solving the system of $F$-linear equations given by the equality $ \L^\s\mX -\mX \L=0$ if the matrix $\mX$ is known. \\

{\bf Criterion for the algebra $\CE$ to be a division algebra where $s\neq 0$ and $a\neq 0$.} 
 By interchanging the roles of $h$ and $x$, the algebra $\CE =(E, \s, a)$ can be written also as
 %\marginpar{CE=Eataus}
\begin{equation}\label{CE=Eataus}
\CE =(E(a), \tau, s)\;\; {\rm where}\;\;E(a)=F[x]/(x^n-a)\;\; {\rm and}\;\; \tau (x)=q^{-1}x.
\end{equation}
\begin{corollary}\label{b17Jun23}%\marginpar{b17Jun23}
 Suppose that $s\neq 0$ and  $a\neq 0$. Recall that $E(s)=F[h]/(h^n-s)$ and  $E(a)=F[x]/(x^n-a)$.  Then the following statements are equivalent:  
 
\begin{enumerate}
\item The algebra $\CE =(E, \s, a)$ is a division algebra. 

\item The algebra    $E(s)$ is a  field and $n=\min \{ d'\geq 1\, | \, d'\mid n$, there is a matrix $\mX\in M_{d'}(E(s))$ such that $ \mX^{\s^{n-1}}\cdots  \mX^\s \mX=a\}$.

 \item The algebra    $E(a)$ is a field and $n=\min \{ d'\geq 1\, | \, d'\mid n$, there is a matrix $\mY\in M_{d'}(E(a))$ such that $ \mY^{\tau^{n-1}}\cdots  \mY^\tau \mY=s\}$.  

\end{enumerate}
\end{corollary}

\begin{proof}   It suffices to show that the first two statements are equivalent since then equivalence of statements 1 and 3 follows by (\ref{CE=Eataus}).

$(1\Leftrightarrow 2)$ We keep the notation of Theorem \ref{12Jun23}. If the algebra $\CE$ is a division algebra then its commutative subalgebra $E(s)$ is a field. So, we may assume that the algebra $E(s)$ is a field. Then, by Theorem \ref{12Jun23}.(1), the algebra $\CE\simeq M_m(D$ is a division algebra iff $m=1$ iff $d=n$ (since $n=\Deg (\CE)=\Deg (M_m(D))=m\Deg (D)=md$) iff the second condition of statement 2 holds. 
   \end{proof}  

{\bf Sufficient condition for $\CE$ to be a division algebra.} Suppose that $\mX$ be as in Theorem \ref{12Jun23}.(1), that is $\mX\in M_{d'}(E(s))$ and $ \mX^{\s^{n-1}}\cdots  \mX^\s \mX=a$. Then $\det (\mX)\in E(s)$ and 
%\marginpar{NEFmX}
\begin{equation}\label{NEFmX}
a^{d'}=N_{E(s)/F}(\det (\mX)).
\end{equation}
In more detail, 
$$
a^{d'}=\det(a)=\det(\mX^{\s^{n-1}}\cdots  \mX^\s \mX) =\det(\mX)^{\s^{n-1}}\cdots  \det(\mX)^\s \det (\mX )=N_{E(s)/F}(\det (\mX)).
$$
Corollary \ref{a12Jun23} provides  a sufficient condition for the algebra $\CE$ in Theorem \ref{12Jun23} to be a division algebra.

\begin{corollary}\label{a12Jun23}%\marginpar{a12Jun23}
Let $\CE =(E, \s, a)$. Suppose that $s\neq 0$, $a\neq 0$,   and the algebra $E=E(s)=F[h]/(h^n-s)$ is a field. If $a^{d'}\not\in N_{E(s)/F}(E(s))$ for all divisors $d'$ of $n$ such that $d'\neq n$ then the algebra $\CE$ is a division algebra. 
\end{corollary}

\begin{proof}   We keep the notation of Theorem \ref{12Jun23}. By Theorem \ref{12Jun23}.(1), the algebra $\CE$ is a division algebra iff $d=n$ iff for every  divisor $d'$ of $n$ such that $d'\neq n$ there is no matrix $\mX\in M_{d'}(E(s))$ such that $ \mX^{\s^{n-1}}\cdots  \mX^\s \mX=a$. Now,  by (\ref{NEFmX}), if $a^{d'}\not\in N_{E(s)/F}(E(s))$ for all divisors $d'$ of $n$ such that $d'\neq n$ then for every  divisor $d'$ of $n$ such that $d'\neq n$ there is no matrix $\mX\in M_{d'}(E(s))$ such that $ \mX^{\s^{n-1}}\cdots  \mX^\s \mX=a$, and so the algebra $\CE$ is a division algebra.   \end{proof}   

{\bf Criterion for $\CE\simeq M_n(F)$.} 

\begin{proposition}\label{c17Jun23}%\marginpar{c17Jun23}
Let $\CE =(E, \s, a)$. Suppose that $s\neq 0$ and  $a\neq 0$. We keep the notation of Corollary \ref{Xa27Jun23}.  Then the following statements are equivalent:  
 
\begin{enumerate}
\item  $\CE \simeq M_n(F)$. 

\item The algebra $\tilde{\CE}''\simeq M_\frac{n}{m(s)m(s,a)}(F)$.

\item The element $s^\frac{1}{m(s)}=N_{\tilde{E}'/F}(b)$ for some element $b\in \Big(\tilde{E}'\Big)^\times$.

\end{enumerate}
\end{proposition}

\begin{proof}    $(1\Leftrightarrow 2)$ By Corollary  \ref{Xa27Jun23}.(5), $\CE \simeq M_{m(s)m(s,a)}(\tilde{\CE}'')$. So, $\CE\simeq M_n(F)$ iff $\tilde{\CE}''\simeq M_\frac{n}{m(s)m(s,a)}(F)$.

 $(2\Leftrightarrow 3)$ By Corollary  \ref{Xa27Jun23}.(1), the algebra $\tilde{\CE}''=(\tilde{E}',\tau', s^\frac{1}{m(s)})$ is a (reduced) cyclic algebra. By \cite[Lemma 15.1]{Pierce-AssAlg},  $\tilde{\CE}''\simeq M_\frac{n}{m(s)m(s,a)}(F)$ iff  the element $s^\frac{1}{m(s)}=N_{\tilde{E}'/F}(b)$ for some element $b\in \Big(\tilde{E}'\Big)^\times$.
   \end{proof}  
 
 {\bf The centralizer $C_{\Inn (M_d(E))}(\o_{\mX^{-1}}\s)$ and its connection with $\Aut_F(D)$.} We keep the assumption of Theorem \ref{12Jun23}. Notice that $\Inn (M_d(E))$ is a subgroup of $\Aut_F(M_d(E))$. Let   $C_{\Inn (M_d(E))}(\o_{\mX^{-1}}\s):=\{ \tau \in \Inn(M_d(E))\: | \, \o_{\mX^{-1}}\s \tau=\tau\o_{\mX^{-1}}\s\}$. By Theorem \ref{12Jun23}.(2),  there is a  group homomorphism 
$$\res_D : C_{\Inn (M_d(E))}(\o_{\mX^{-1}}\s)\ra \Aut_F(D),\;\; \tau \mapsto \res_D (\tau ), $$  where $\res_D (\tau )$ is the restriction of the map $\tau $ to $D$,  and the image $\im (\res_D)$ is isomorphic to the factor group 
$C_{\Inn (M_d(E))}(\o_{\mX^{-1}}\s)/\Big(C_{\Inn (M_d(E))}(\o_{\mX^{-1}}\s)\cap \CC (D) \Big)$, by Theorem \ref{12Jun23}.(3).
Corollary \ref{b12Jun23} shows that, in general, the group $\im (\res_D)$ is a subgroup of $\Aut_F(D)$. 

Clearly, $\o_{\mX^{-1}}\s\in \Aut_F(M_d(E)) $ and 
 $\Aut_E(M_d(E))$ is a subgroup of  the group $\Aut_F(M_d(E))$. The set $$\CC_{\Aut_F(M_d(E))}(D):=\{ \g\in \Aut_F(M_d(E)) \, | \, \g (\L ) = \L\;\; {\rm  for \; all}\;\;\L \in D\}$$ is a proper subgroup of  $\Aut_F(M_d(E))$. Corollary \ref{b12Jun23} is a description of the group $\CN (D)$ in terms of the group commutator and  the group $\CC_{\Aut_F(M_d(E))}(D)$. 

\begin{corollary}\label{b12Jun23}%\marginpar{b12Jun23}
Let   $\CE =(E, \s, a)$. Suppose that $s\neq 0$, $a\neq 0$,   and the algebra  $E=E(s)=F[h]/(h^n-s)$ is a field.  Then 
$\CN (D)=\{ \tau \in \Inn (M_d(E))\, | \,  [ \o_{\mX}^{-1}\s, \tau]\in \CC_{\Aut_F(M_d(E))}(D)\}$ where  $[ \o_{\mX^{-1}}\s, \tau]:= \o_{\mX^{-1}}\s \tau \Big( \o_{\mX^{-1}}\s \Big)^{-1}\tau^{-1} $ is the group commutators of the elements $\o_{\mX^{-1}}\s$ and $\tau$, and  $C_{\Inn (M_d(E))}(\o_{\mX^{-1}}\s)$ $=\{ \tau \in \Inn (M_d(E))\, | \, [ \o_{\mX}^{-1}\s, \tau]=1\} \subseteq \CN (D)$.
\end{corollary}

\begin{proof}   Let $\CN':=\{ \tau \in \Inn (M_d(E))\, | \, [ \o_{\mX}^{-1}\s, \tau]\in \CC_{\Aut_F(M_d(E))}(D)\}$. By the definition of the set $\CN =\CN (D)$, $ \CN \subseteq \CN'$: If $\tau \in \CN $ then $\tau (D)=D$, and so
$ [ \o_{\mX}^{-1}\s, \tau] (\L)=\L$ for all $\L \in D$, i.e. $[ \o_{\mX}^{-1}\s, \tau]\in \CC_{\Aut_F(M_d(E))}(D)$, as required.
 Conversely, let $\tau'\in \CN'$. Then for all $\L \in D$, $\L =  [ \o_{\mX}^{-1}\s, \tau](\L)$ or, equivalently,
 $$\o_{\mX}^{-1}\s  \tau (\L) = \tau \o_{\mX}^{-1}\s (\L)=\tau (\L),$$
since $\o_{\mX}^{-1}\s (\L)=\L$ i.e. $\tau \in \CN$, by Theorem \ref{12Jun23}.(2,3). Therefore, $\CN = \CN'$.  

By the definition of the set $C_{\Inn (M_d(E))}(\o_{\mX^{-1}}\s)$, it is equal to $\{ \tau \in \Inn (M_d(E))\, | \, [ \o_{\mX}^{-1}\s, \tau]=1\}$, and so $C_{\Inn (M_d(E))}(\o_{\mX^{-1}}\s)\subseteq \CN (D)$.    \end{proof}

{\bf Tensor product decomposition $\CE =\bigotimes_{i=1}^d \CE_i$ where $\Deg(\CE)=n=\prod_{i=1}^dp_i^{n_i}$ and $\Deg (\CE_i)=p_i^{n_i}$.}
Let $F$, $E=E(s)$,  and $ \CE =(E, \s, a)$ be as above, see (\ref{qEs2}).  Suppose that $s\neq 0$ and $a\neq 0$. 
 Let  the natural number $n=\prod_{i=1}^dp_i^{n_i}$ be a unique product of primes $p_i$ with multiplicities $n_i\geq 1$.
  For each $i=1, \ldots , d$, let $n_i':=\frac{n}{p_i^{n_i}}$, i.e. $n=n_i'p_i^{n_i}$, and $\CE_i$ be the $F$-subalgebra of $\CE$ generated by the elements $h^{n_i'}$ and $x^{n_i'}$. It follows from the equality $\CE=\bigoplus_{i,j=0}^{n-1}Fh^ix^j$  that 
 %\marginpar{CEib}
\begin{equation}\label{CEib}
\CE_i:=\bigoplus_{k,j=0}^{n_i'-1}F\Big(h^{n_i'}\Big)^k \Big(x^{n_i'}\Big)^j\simeq E_i[x^{n_i'}, \s_i]/\Big(\Big(x^{n_i'}\Big)^{p_i^{n_i}}-a\Big), \; E_i:=F[h^{n_i'}]/\Big(\Big(h^{n_i'}\Big)^{p_i^{n_i}}-s\Big), \; \s_i(h^{n_i'}):=q^{(n_i')^2}h^{n_i'}.
\end{equation}
Notice that $\Big(x^{n_i'}\Big)^{p_i^{n_i}}-a=x^n-a$, $\Big(h^{n_i'}\Big)^{p_i^{n_i}}-s=h^n-s$, and $x^{n_i'}h^{n_i'}=q^{(n_i')^2}h^{n_i'}x^{n_i'}=\s_i(h^{n_i'})x^{n_i'}$.

\begin{proposition}\label{A31May23}%\marginpar{A31May23}
Let $\CE =(E, \s, a)$, $s\neq 0$,  $a\neq 0$, and $\Deg(\CE)=n=\prod_{i=1}^dp_i^{n_i}$ be as above, see (\ref{qEs2}). 

\begin{enumerate}
\item The polynomial $t^{p_i^{n_i}}-s\in F[t]$ is an irreducible polynomial over $F$ if and only if $s\not\in \bigcup_{j=1}^{n_i}F^{[p_i^j]}$.

\item Suppose that the polynomial $t^{p_i^{n_i}}-s\in F[t]$ is an irreducible polynomial over $F$. Then the field extension $E_i/F$ is a Galois field extension with Galois  group $G(E_i/F)=\langle \s_i \, | \, \s^{p_i^{n_i}}=1\rangle$  which is a cyclic group of order $p_i^{n_i}$ where $\s_i(h^{n_i'}):=q^{(n_i')^2}h^{n_i'}$, and $[E_i:F]=p_i^{n_i}$. 

\item  The  algebra $\CE_i$ is a central simple $F$-algebra with $\Deg (\CE_i)= p_i^{n_i}$.

\item The algebra $\CE =\bigotimes_{i=1}^d \CE_i$ is a tensor product of its  subalgebras $\CE_i$.

\end{enumerate}
\end{proposition}

\begin{proof}   1. Statement 1 is a particular case of Theorem Theorem \ref{qA2Sep20}.(1).

2. Statement 1 is a particular case of Theorem Theorem \ref{qA2Sep20}.(2).

3. By the assumption, $s\neq 0$ and $a\neq 0$. Then, by Theorem \ref{qA2Sep20}.(3a), the algebra $\CE_i$ is a central simple algebra with $\Deg (\CE_i)=p_i^{n_i}$.

4. For all $i\neq j$, the elements $x_i^{n_i'}$ and $h_j^{n_j'}$ commute:
$$x_i^{n_i'}h_j^{n_j'}=q^{n_i'n_j'}h_j^{n_j'}x_i^{n_i'}=h_j^{n_j'}x_i^{n_i'}$$
since $n\mid n_i'n_j'$ and $q^n=1$. So, $\alpha_i\alpha_j=\alpha_j\alpha_i$ for all elements $\alpha_i\in \CE_i$ and $\alpha_j\in \CE_j$.  To prove statement 4 we use induction on $d\geq 1$. The initial case $d=1$ is obvious. Suppose that $d>1$ and the result is true for all $d'<d$. Let $\CE (n)=\CE$. By induction on $d$, $\CE \Big(\frac{n}{p_d^{n_d}}\Big)=\bigotimes_{i=1}^{d-1} \CE_i$. By statement 3, the algebras $\CE_i$ are central simple algebras, i.e. 
the algebra $\CE \Big(\frac{n}{p_d^{n_d}}\Big)$ is a central simple algebra of degree $$\Deg\Big(\CE \Big(\frac{n}{p_d^{n_d}}\Big)\Big)=\frac{n}{p_d^{n_d}}.$$
The central simple algebra $\CE \Big(\frac{n}{p_d^{n_d}}\Big)$ is a subalgebra of the central simple algebra $\CE$. Hence,
$$\CE =\CE\Big(\frac{n}{p_d^{n_d}}\Big)\t_FC_{\CE}\Big(\CE_\frac{n}{p_d^{n_d}}\Big)\;\; {\rm and}\;\;\dim_F \Big(C_{\CE}\Big(\CE_\frac{n}{p_d^{n_d}}\Big)\Big)=\dim_F(\CE)\cdot \Big( \dim_F\Big(\CE_\frac{n}{p_d^{n_d}}\Big)\Big)^{-1}= p_d^{n_d}.$$ 
Since $\CE_d\subseteq  C_{\CE}\Big(\CE_\frac{n}{p_d^{n_d}}\Big)$ and $\dim_F(\CE_d)= p_d^{n_d}=\dim_F \Big(C_{\CE}\Big(\CE_\frac{n}{p_d^{n_d}}\Big)\Big)$, we must have  $\CE_d= C_{\CE}\Big(\CE_\frac{n}{p_d^{n_d}}\Big)$. Therefore, $\CE=\bigotimes_{i=1}^d\CE_i$, as required.   \end{proof}   

Corollary \ref{a1Jun23} describes the centralizers of the subalgebras  $\bigotimes_{i\in I}\CE_i$ of $\CE$ where $I\subseteq \{ 1, \ldots ,  d\}$.

\begin{corollary}\label{a1Jun23}%\marginpar{a1Jun23}
Let $\CE =(E, \s, a)$, $s\neq 0$,  $a\neq 0$, and $\Deg(\CE)=n=\prod_{i=1}^dp_i^{n_i}$ be as in Theorem \ref{A31May23}. In particular, $\CE=\bigotimes_{i=1}^d\CE_i$ (Theorem \ref{A31May23}.(4)). Then for all subsets $I\subseteq \{ 1, \ldots ,  d\}$, 
$C_{\CE}(\bigotimes_{i\in I}\CE_i)=\bigotimes_{j\in CI}\CE_j$ and $\CE =\Big(\bigotimes_{i\in I}\CE_i\Big)\t C_{\CE}(\bigotimes_{i\in I}\CE_i)$   where $CI=\{ 1, \ldots , d\}\backslash I$. 
\end{corollary}

\begin{proof}   Let $A$ be a central simple algebra and $B$ be its central simple subalgebra. Then $A=B\t C_A(B)$. Now, the corollary follows from the equality (Theorem \ref{A31May23}.(4))
$$\CE=\bigotimes_{i=1}^d\CE_i=\bigg(\bigotimes_{i\in I}\CE_i\bigg)\t\bigg(\bigotimes_{j\in J}\CE_j\bigg).$$
\end{proof}

%%%%%%%%%%%%%%%%%%    Section 3    %%%%%%%%%%%%%

\section{PLM-rings and Quillen's Lemma}\label{PLMRINGS}%\marginpar{PLMRINGS}

The aim of this section is to introduce a new class of rings, the PLM-rings,  and to prove several results about them that are used in classifying  primitive ideals of the algebras $\mA$, $A_1$, $\CA$, and $\CB$ (Theorem \ref{2Aug23} and Corollary \ref{b2Aug23}). \\

Below is  a proof of Proposition \ref{B30Jul23}.

 \begin{proof}    1. The first class belongs to the class of PLM-algebras since the image of the centre in  the endomorphism algebra of each simple module (which is a division algebra)  is an algebraic commutative domain, i.e. a field. 

2. Quillen's Lemma \cite{Quillen'sLemma} states  that   algebras in statement  2  belong to the class of algebras in statement 1, and so statement 2 follows from statement 1.

3. All algebras $U(\CG)$ are examples of algebras in statement 2, and so statement 3 follows from statement 2. \end{proof}

%So, the class of PLM-rings is a large class of rings. 

For each prime ideal $\gp\in \Spec (Z(R))$,  $S_\gp:=Z(R)\backslash \gp\in \Den (Z(R))$ and $S_\gp \in \Den (R)$. For an $R$-module $M$, the localization $S_\gp^{-1}M$ is denoted by $M_\gp$ and is called the {\bf localization of $M$ at the prime ideal} $\gp$. Then $k(\gp):=\Big(Z(R)/\gp\Big)_\gp
$ is the field of fractions of the commutative domain $Z(R)/\gp$. 
Clearly, $R_\gp=Z(R)_\gp\t_{Z(R)}R$ is a ring and  $M_\gp=R_\gp\t_R M\simeq Z(R)_\gp\t_{Z(R)}M$ is an $R_\gp$-module. The map 
%\marginpar{vgpRRp}
\begin{equation}\label{vgpRRp}
\phi_\gp: R/R\gp\ra \Big( R/R\gp\Big)_\gp\simeq k(\gp)\t_{Z(A)}R/\gp, \;\; u\mapsto \frac{u}{1}.
\end{equation}
is a ring homomorphism. For an $R$-module $M$,  $\ann_{Z(R)}(M):=\{z\in Z(R)\, | \, zM=0\}=\ann_R(M)\cap  Z(A)$, and let $\soc_R(M)$ be the socle of the $R$-module $M$.

\begin{lemma}\label{a3Aug23}%\marginpar{a3Aug23}

\begin{enumerate}
\item For all primes $\gp \in {\rm ZL}(R)$ and nonzero $k(\gp)\t_{Z(R)}R/R\gp$-modules $\CM$, $\ann_{Z(R)}(\CM)=\gp$.
\item ${\rm ZL}(R)=\{ \gp \in \Spec (Z(R))\, | \, 
k(\gp)\t_{Z(R)}R/R\gp\neq \{ 0\}\}$.
\end{enumerate}
\end{lemma}

\begin{proof} 1. Statement 1 follow from the definition of the ring $k(\gp)\t_{Z(R)}R/R\gp$.

2. Statement 2 follows from statement 1.
\end{proof}

\begin{lemma}\label{b1Aug23}%\marginpar{b1Aug23}
Suppose that  $M$ is a simple $R$-module  and $\gp =\ann_{Z(R)}(M)$. Then $\ann_{Z(R)}(m)=\gp$ for all nonzero elements $m\in M$.
\end{lemma}

\begin{proof}
The lemma follows at once from the equality $M=Rm$ for all nonzero elements $m\in M$:  $\ann_{Z(R)}(m)=\ann_{Z(R)}(Rm)=\ann_{Z(R)}(M)=\gp$.
\end{proof}

\begin{lemma}\label{a1Aug23}%\marginpar{a1Aug23}
Suppose that $\gp \in {\rm ZL}(R)$.
\begin{enumerate}
\item The map $\phi_\gp $ is an injection iff $\tor_{S_\gp}(R/R\gp)=0$ iff $sr\in R\gp$, for some elements $s\in S_\gp$ and $r\in R$, implies $r\in R\gp$.
\item $\gp \in {\rm PL}(R)$  iff  there is a simple $k(\gp)\t_{Z(A)}R/\gp$-module $\CM$ with 
$\soc_R(\CM)\neq 0$.
\end{enumerate}
\end{lemma}

\begin{proof} 1. Statement 1 follows from the equality
$\ker (\phi_\gp)=\tor_{S_\gp}(R/R\gp)$.

2. $(\Rightarrow)$ Suppose that  $\gp \in {\rm PL}(R)$. Then there is a simple $R$-module $M$  such that $\gp=\ann_{Z(R)}(M)$. The module $R$-module $M$ is simple.  Hence, $\ann_{Z(R)}(m)=\ann_{Z(R)}(M)=\gp$ for all nonzero elements $m\in M$ (Lemma \ref{b1Aug23}). Then the $k(\gp)\t_{Z(A)}R/\gp$-module $M_\gp$ is a nonzero module and a  simple module with $\soc_R(M_\gp)=M$  (since the simple  $R$-module $M$ is an essential $R$-submodule of $M_\gp$). 

$(\Leftarrow )$ Suppose that $\CM$ is a simple 
 $k(\gp)\t_{Z(A)}R/\gp$-module $\CM$ with 
$M:=\soc_R(\CM)\neq 0$. The $R$-module $\CM$ is $S_\gp$-torsion free and so is its submodule $M$. 
 Then the $R$-module $M$ is simple (since the $k(\gp)\t_{Z(A)}R/\gp$-module $M_\gp =\CM$ is simple and  localizations  respect the direct sums). 
 Since $\CM$ is an $k(\gp)$-module, $\ann_{Z(R)}(\CM)=\gp$. Since $\CM$ is a simple  $k(\gp)\t_{Z(A)}R/\gp$-module, $\ann_{Z(R)}(m)=\gp$  for all nonzero elements $m\in \CM$ (Lemma \ref{b1Aug23}),  $ \gp=\ann_{Z(R)}(M)\in {\rm PL}(R)$.
 \end{proof}

 Corollary \ref{A30Jul23} is a criterion for a ring to be a PLM-ring.

\begin{corollary}\label{A30Jul23}%\marginpar{A30Jul23}
The ring $R$ is a PLM-ring  iff $\soc_R(\CM)=0$ for all simple $k(\gp)\t_{Z(R)}R/R\gp$-modules $\CM$ and all $\gp \in {\rm ZL}(R)\backslash \Max (Z(R))$.
\end{corollary}

\begin{proof} The corollary follows from Lemma \ref{a1Aug23}.(2).
\end{proof}

 Theorem \ref{31Jul23}.(2) is a criterion for a  prime idea $\gp \in {\rm ZL}(R)$ to belong to the set ${\rm PL}(R)$.

\begin{theorem}\label{31Jul23}%\marginpar{31Jul23}
Suppose that  $\gp \in {\rm ZL}(R)$,  the map $\phi_\gp$ is an injection,  and the ring $k(\gp)\t_{Z(A)}R/\gp$ is a simple  ring. Then
\begin{enumerate}
\item   $\Spec (R,\gp)=\{ R\gp\}$.

%\item  $\gp \in {\rm PL}(R)$  iff  $\soc_R(R/R\gp)=0$.

\item In addition, suppose that  the ring $k(\gp)\t_{Z(A)}R/\gp$ is an artinian ring.  Then the following statements are equivalent:

\begin{enumerate}
\item $\gp \in {\rm PL}(R)$.
\item    The map $\phi_\gp$ is a bijection.
\item  $\soc_R(R/R\gp)\neq 0$.
\end{enumerate}

\end{enumerate}
\end{theorem}

\begin{proof} By Lemma \ref{a1Aug23}, the homomorphism $\phi_\gp$ is an injection  iff $\tor_{S_\gp}(R/R\gp)=0$. In particular, the ring $R/R\gp$ can be seen as a subring of $\Big(R/R\gp\Big)_\gp$.

1. Suppose that $\gp \in {\rm ZL}(R)$ and  $P\in \Spec (R,\gp)$. 
 Then  $P\cap Z(R)=\gp$, and so $R\gp \subseteq P$ and $Z(R)_\gp\neq \gp_\gp = \Big( P\cap Z(R)\Big)_\gp= P_\gp\cap Z(R)_\gp$. Therefore, $$P_\gp \neq R_\gp.$$ Suppose that $R\gp \neq P$, we seek a contradiction. Since $0\neq P/R\gp\subseteq R/R\gp\subseteq  \Big(R/R\gp\Big)_\gp$ and  $\tor_{S_\gp}(R/R\gp)=0$, 
 $$ 0\neq \Big(P/R\gp \Big)_\gp\simeq  P_\gp/R_\gp \gp\subseteq R_\gp/R_\gp \gp\simeq \Big(R/R\gp\Big)_\gp, $$
i.e. $ P_\gp/R_\gp \gp$ is a  nonzero ideal of the simple ring $\Big(R/R\gp\Big)_\gp$. Hence, 
$ P_\gp/R_\gp \gp=R_\gp /R_\gp\gp$, and so $P_\gp=R_\gp$, a contradiction. Therefore, $P=R\gp$, as required.

2. $(a)\Rightarrow (b)$  Suppose that $\gp \in {\rm PL}(R)$. The ring $k(\gp)\t_{Z(A)}R/\gp$ is a simple artinian ring.  Let $\CM$ be a unique (up to isomorphism) simple $k(\gp)\t_{Z(A)}R/\gp$-module. Then $k(\gp)\t_{Z(A)}R/\gp$-module $k(\gp)\t_{Z(A)}R/\gp$ is isomorphic to $\CM^n$ for some natural number $n\geq 1$. Let $k(\gp)\t_{Z(A)}R/\gp=\bigoplus_{i=1}^n\CM_i$ be a direct sum of simple $k(\gp)\t_{Z(A)}R/\gp$-submodules and each of them is isomorphic to $\CM$.
Since the map $\phi_\gp$ is an injection, the $R$-submodule $R/R\gp$ of $k(\gp)\t_{Z(A)}R/\gp$ is an essential submodule. Then for each $i=1, \ldots , n$, the intersection $M_i:=R/R\gp\cap \CM_i$ is a simple $R$-module and the direct sum $N:=\bigoplus_{i=1}^nM_i$ is an essential $R$-submodule of  $k(\gp)\t_{Z(A)}R/\gp$. Since $N\subseteq R/R\gp$, the $R$-module $ N$ is also an essential $R$-submodule of the $R$-modules $R/R\gp$.  Hence, $N=\soc_R(R/R\gp)\neq 0$.

For all $i,j=1, \ldots , n$, the simple $k(\gp)\t_{Z(A)}R/\gp$-modules $\Big(M_i\Big)_\gp\simeq \CM_i$ and $\Big(M_j\Big)_\gp\simeq \CM_j$ are isomorphic. Hence, their $R$-socles are  isomorphic as $R$-modules, i.e. $$M_i=\soc_R(\CM_i)\simeq \soc_R(\CM_j)=M_j.$$
By Lemma \ref{b1Aug23}, for every nonzero element of  $k(\gp)\t_{Z(A)}R/\gp\simeq \CM^n$, $\ann_{Z(R)}(m)=\ann_{Z(R)}(\CM)=\gp$. Therefore, $\ann_{Z(R)}(M_i)=\gp$ for all $i=1,\ldots , n$. The $R$-modules $M_i$ are simple with  $\ann_{Z(R)}(M_i)=\gp$. Hence, every element $s\in S_\gp$ acts as bijection on $M_i$, i.e. 
$$M_i=\Big(M_i\Big)_\gp=\CM_i\;\; (i=1, \ldots , n)\;\;{\rm and}\;\; R/R\gp = k(\gp)\t_{Z(A)}R/\gp,$$
i.e. the map $\phi_\gp$ is the identity map.

$(b)\Rightarrow (c)$ Suppose that the map $\phi_\gp$ is a bijection, i.e. $R/R\gp = k(\gp)\t_{Z(A)}R/\gp$. Then $$\soc_R(R/R\gp)=\soc_R(k(\gp)\t_{Z(A)}R/\gp)=\soc_{k(\gp)\t_{Z(A)}R/\gp}(k(\gp)\t_{Z(A)}R/\gp)=k(\gp)\t_{Z(A)}R/\gp\neq 0$$
since the ring $k(\gp)\t_{Z(A)}R/\gp$ is a simple artinian ring.

$(c)\Rightarrow (a)$ Suppose that  $\soc_R(R/R\gp)\neq 0$. Clearly, the socle $\soc_R(R/R\gp)$ is a nonzero ideal of the algebra $R/R\gp$. Since the map $\phi_\gp$ is a monomorphism, 
$R/R\gp\subseteq k(\gp)\t_{Z(A)}R/\gp$ and $\tor_{S_\gp}(R/R\gp)=0$ (Lemma \ref{a1Aug23}.(1)). Then $\soc_R(R/R\gp)_\gp$ is a nonzro ideal of the simple algebra $k(\gp)\t_{Z(A)}R/\gp$. Therefore,
$$k(\gp)\t_{Z(A)}R/\gp=\soc_R(R/R\gp)_\gp.$$
By Lemma \ref{b1Aug23}, $\gp=\ann_{Z(R)}(k(\gp)\t_{Z(A)}R/\gp)=\ann_{Z(R)}(M)$ for any simple $R$-submodule $M$ of $k(\gp)\t_{Z(A)}R/\gp$, i.e. $\gp \in {\rm PL}(R)$.
\end{proof} 

For a class of rings, Corollary \ref{c1Aug23} describe the sets ${\rm PL}(R)$ and $\Prim (R)$.

\begin{corollary}\label{c1Aug23}%\marginpar{c1Aug23}
Suppose that for each  $\gp \in {\rm ZL}(R)$,  the map $\phi_\gp$ is an injection   and the ring $k(\gp)\t_{Z(A)}R/\gp$ is a simple artinian ring. Then

\begin{enumerate}
\item ${\rm PL}(R)=\{ \gp \in {\rm ZL}(R)\, | \, $
 the map $\phi_\gp$ is a bijection$\}= \{ \gp \in {\rm ZL}(R)\, | \, \soc_R(R/R\gp)\neq 0\}$.
\item   For all $\gp\in {\rm PL}(R)$, $\Prim (R,\gp)=\{ R\gp\}$.
\item $\Prim (R)=\{ R\gp\, | \, \gp\in {\rm PL}(R)\}$.

\end{enumerate}
\end{corollary}

\begin{proof} 1. Statement 1 follows from Theorem \ref{31Jul23}.(2).

2. Statement 2 follows from Theorem \ref{31Jul23}.(1). 

3. Statement 3 follows from statement 2.
\end{proof}

 Theorem \ref{A1Aug23} is a criterion for a ring to be a PLM-ring which is given in terms of the homomorphisms $\phi_\gp$.

\begin{theorem}\label{A1Aug23}
%\marginpar{A1Aug23}
Suppose that for every $\gp \in {\rm ZL}(R)\backslash \Max (Z(R))$,  the map $\phi_\gp$ is an injection and the ring $k(\gp)\t_{Z(A)}R/\gp$ is a simple artinian ring. Then the following statements are equivalent:

\begin{enumerate}
\item The ring $R$ is PLM-ring.
\item    The map $\phi_\gp$ is not a surjection  for all $\gp \in {\rm ZL}(R)\backslash \Max (Z(R))$.
\item  $\soc_R(R/R\gp)=0$ for all $\gp \in {\rm ZL}(R)\backslash \Max (Z(R))$.
\end{enumerate}
\end{theorem}

\begin{proof} The theorem follows from Theorem \ref{31Jul23}.(2).
\end{proof}
For a ring $T$,  we denote by $T-{\rm Mod}$ the category of left $T$-modules. 
Suppose that  $S$ is a subring of $ T$. Then the ring $T$ is called a {\bf faithfully flat extension} of $S$ if the functor $T\t_S -: S-{\rm Mod}\ra T-{\rm Mod}$, $M\mapsto T\t_SM$ is an exact functor.

\begin{lemma}\label{a2Aug23}%\marginpar{a2Aug23}
\begin{enumerate}
\item If  the ring $R$ is a faithfully flat extension of the centre $Z(R)$ 
 then $ \Max (Z(R))\subseteq {\rm PL}(R)$.
\item If  the ring $R$ is a free $Z(R)$-module 
 then $ \Max (Z(R))\subseteq {\rm PL}(R)$.
\end{enumerate}

\end{lemma}

\begin{proof} 1. Since the ring $R$ is a faithfully flat extension of the centre $Z(R)$,  $R\gm\neq R$ for all $\gm \in \Max(Z(R))$. Since the factor ring  $R(\gm):=R/R\gm$ is a unital ring, there is a left maximal ideal, say $I$, of the ring $R(\gm )$, and so the $R(\gm )$-module $M=R(\gm )/I$ is simple with $\gm=\ann_{Z(R)}(M)\in {\rm PL}(R)$. Therefore, $ \Max (Z(R))\subseteq {\rm PL}(R)$.

2. If  the ring $R$ is a free $Z(R)$-module 
 then the ring $R$ is faithfully flat over $Z(R)$, and statement 2 follows from statement 1.
\end{proof}

For a class of rings that contains the quantum Weyl algebra $A_1$ (Proposition \ref{30Jul23}), Corollary \ref{b31Jul23} gives sufficient conditions for a ring to be a  PLM-ring.

\begin{corollary}\label{b31Jul23}%\marginpar{b31Jul23}
Suppose that the ring $R$ is  a free $Z(R)$-module and     the ring $k(\gp)\t_{Z(R)}R/\gp$ is a simple  artinian ring for every $\gp \in {\rm ZL}(R)\backslash \Max (Z(R))$. Then the ring $R$ is PLM-ring with ${\rm PL}(R)=\Max (Z(R))$.

\end{corollary}

\begin{proof} Since the ring $R$ is a free $Z(R)$-module, the map $\phi_\gp$ is an injection which is not a surjection for all 
 $\gp \in {\rm ZL}(R)\backslash \Max (Z(R))$. Now,  by Theorem \ref{A1Aug23}, the ring $R$ is PLM-ring.  By Lemma \ref{a2Aug23}.(2),  ${\rm PL}(R)=\Max (Z(R))$.
\end{proof}

Below is a proof of Theorem \ref{2Aug23}.

\begin{proof} 1. Statement 1 follows from Corollary \ref{b31Jul23}.

2. Since  the ring $R$ is a free $Z(R)$-module,   the map $\phi_\gp$ is an injection for every $\gp \in {\rm ZL}(R)$. Now,  statement 2 follows from Corollary \ref{c1Aug23}.(2).

3. By statements 1 and  2, $\Prim(R)=\{ R\gm\, | \, \gm\in \Max (Z(R))\}$. The inclusion  $\Max (R)\subseteq \Prim(R)$ holds (for all unital rings). 
The reverse inclusion follows from the fact for each maximal ideal $\gm \in \Max (Z(R))$, the ring
$R/R\gm=\Big( R/R\gm\Big)_\gm=k(\gm)\t_{Z(R)}R/\gm$ is a simple  artinian ring, i.e. the ideal $R\gm$ is a maximal ideal of the ring $R$.

4. For each maximal ideal $\gm \in \Max (Z(R))$,
 the ring $R/R\gm=\Big(R/R\gm \Big)_\gm=k(\gm)t_{Z(R)}R/R\gm$ is a simple  artinian ring. Now, statement 4 follows from the equality  $\Prim(R)=\Max (R)=\{ R\gm\, | \, \gm\in \Max (Z(R))\}$ (statement 3).

5. Statement 5 follows from statements 3 and 4.

6. Statement 6 is obvious (for every simple artinian ring $\L =M_n(D)$, where $n$ is a natural number and $D$ is a division ring,  and its unique simple module $V$, $\End_\L (V)\simeq D$). 
\end{proof}

\begin{corollary}\label{b2Aug23}
%\marginpar{b2Aug23}
Suppose that the ring $R$ is a free $Z(R)$-module and  for every $\gp \in \Spec (Z(R))$,   the ring $k(\gp)\t_{Z(R)}R/\gp$ is a simple  artinian ring. Then  then ${\rm ZL}(R)=\Spec (Z(R))$ and Theorem \ref{2Aug23} holds.
\end{corollary}

\begin{proof} For every $\gp \in \Spec (Z(R))$,   the ring $k(\gp)\t_{Z(R)}R/\gp$ is a simple  artinian ring. Hence, ${\rm ZL}(R)=\Spec (Z(R))$ (Lemma \ref{a3Aug23}.(2)). 
 The ring $R$ is a free $Z(R)$-module, and so  the maps $\phi_\gp$ are injections. Therefore, Theorem \ref{2Aug23} holds for the ring $R$.
\end{proof}

%%%%%%%%%%%%%%%%%%%%%% Section  4  %%%%%%%%%%%%%%%%%%

\section{Classifications of  prime ideals and simple modules  of the  algebras $\mA$, $\CA$,  and $\CB$}\label{qAKHSKEW}%\marginpar{qAKHSKEW}

 The aim of the section is to classify the sets of prime, completely prime,  maximal and primitive ideals  and simple modules of the algebras  $\mA$, $\CA$,    and $\CB$. 
 Using the classification of primitive ideals we obtain a classification of simple modules for the algebras  $\mA$, $\CA$,    and $\CB$.

For an algebra $R$, let $\Spec\,(R)$ be the set of its prime ideals. The set $(\Spec\,(R), \subseteq)$ is a partially ordered set (poset) with respect to inclusion of prime ideals. A prime ideal $\gp$ of an algebra $R$ is called a \emph{completely prime ideal} if $R/\gp$ is a domain. We denote by  $\Spec_c(R)$ the set of completely prime ideals of $R$ which  is called the \emph{completely prime spectrum} of $R$. The annihilator of a simple $R$-module is a prime ideal of $R$. That kind of prime ideals are called {\em primitive}. Let  ${\rm Prim} (R)$  be the  set of all primitive ideals of $R$. By the  definition, the poset $(\Spec\,(R), \subseteq)$ is determined by the containment information on prime ideals of $R$. In general, it is difficult to say whether one prime ideal contains or does not contain another prime ideal. Proposition \ref{aA12Mar15}.(3) is a general result about containments of primes. 
It partitions the prime spectrum of a ring into two disjoint subsets such that elements of one are not contained in elements of the other. \\

{\bf The Zariski topology on the prime spectrum of a ring.} 
Let $R$ be a ring. For an ideal $I$ of $R$, let $V(I):=\{ \gp \in \Spec (R)\, | \, I\subseteq \gp\}$ and $W(I):= \Spec (G)\backslash V(I)$. The sets $W(I)$ are open sets in the {\em Zariski topology} on the prime spectrum $\Spec (R)$ of the ring $R$. This topology is also called the {\em Stone topology} or the {\em hull-kernel topology}. Clearly, the sets $V(I)$ are precisely all the closed sets in the Zariski topology. If the ring $R$ is a Noetherian ring then the set $\min (I)$ of minimal primes of every ideal $I$ is a finite set, and so $V(I)=\bigcup_{\gp \in \min (I)}V(\gp)$. So, the Zariski topology is determined by  containments  of primes, i.e. the Zariski topology is the topology of the partially ordered set (poset) $(\Spec (R), \subseteq )$ where  the closed sets are (by definition) finite intersection of the sets $V(\gp)$ where $\gp\in \Spec (R)$. The algebras $\mA$, $A_1(q)$, $\CA$, and $\CB$  are Noetherian algebras, and so the Zariski topology on their prime spectra  is determined by containments  of primes.\\

 Each element $r \in R$ determines two maps from $R$ to $R$,  $r \cdot: s \mapsto rs$ and $\cdot r: s \mapsto sr$, where $s \in R.$  An element $r \in R$ is called a {\em regular element} of $R$ if $\ker (r\cdot )=0$ and  $\ker (\cdot r)=0$, i.e. the element $r$ is a non-zero-divisor.  An element $a \in R$ is called a {\em normal element} if $Ra=aR$. Each regular normal element $r\in R$ determines an automorpmism  $\o_r$ of $R$ which is given by the rule $\o_r(s)r=rs$ for all elements $s\in R$. 
 We denote by  $(r)$  the ideal of $R$ generated by the element $r$.
 
 A 
multiplicative subset $S$ of $R$   is called  a {\em
left Ore set} if it satisfies the {\em left Ore condition}: for
each $r\in R$ and
 $s\in S$, $$ Sr\cap Rs\neq \emptyset .$$
Let $\Ore_l(R)$ be the set of all left Ore sets of $R$.
  For  $S\in \Ore_l(R)$, $\ass_l (S) :=\{ r\in
R\, | \, sr=0 \;\; {\rm for\;  some}\;\; s\in S\}$  is an ideal of
the ring $R$. 

%$\noindent $

A left Ore set $S$ is called a {\em left denominator set} of the
ring $R$ if $rs=0$ for some elements $ r\in R$ and $s\in S$ implies
$tr=0$ for some element $t\in S$, i.e., $r\in \ass_l (S)$. Let
$\Den_l(R)$ be the set of all left denominator sets of $R$. For
$S\in \Den_l(R)$, let $$S^{-1}R=\{ s^{-1}r\, | \, s\in S, r\in R\}$$
be the {\em left localization} of the ring $R$ at $S$ (the {\em
left quotient ring} of $R$ at $S$). Let us stress that in Ore's method of localization one can localize {\em precisely} at left denominator sets.
 In a similar way, right Ore and right denominator sets are defined. 
Let $\Ore_r(R)$ and 
$\Den_r(R)$ be the set of all right  Ore and  right   denominator sets of $R$, respectively.  For $S\in \Ore_r(R)$, the set  $\ass_r(S):=\{ r\in R\, | \, rs=0$ for some $s\in S\}$ is an ideal of $R$. For
$S\in \Den_r(R)$,  
$$RS^{-1}=\{ rs^{-1}\, | \, s\in S, r\in R\}$$ is the {\em right localization} of the ring $R$ at $S$. 

Let $R$ be a ring and $\CC_R$ be the set of all regular elements  of the ring $R$. The set $\CC_R$ is a multiplicative set, i.e. a multiplicative monoid. If the set $\CC_R$ is a left (resp., right) {\em Ore set} then the localization of the ring $R$, $Q_l(R):=\CC_R^{-1}R$ (resp., $Q_r(R):=R\CC_R^{-1}$) is called the {\em left quotient ring} (resp., the {\em right quotient ring}) of $R$. If the set $\CC_R$ is an {\em Ore set}, i.e. a left and right Ore set, then the rings $Q_l(R)$ and $Q_r(R)$ are isomorphic and  the ring $Q(R):=Q_l(R)\simeq Q_r(R)$ is called the {\em quotient ring} of $R$.

Given ring homomorphisms $\nu_A: R\ra A$ and $\nu_B :R\ra B$. A ring homomorphism $f:A\ra B$ is called an $R$-{\em homomorphism} if $\nu_B= f\nu_A$.  A left and right Ore set is called an {\em Ore set}. Similarly,   a left and right denominator set is called a {\em denominator set}.   Let $\Ore (R)$ and 
$\Den (R)$ be the set of all   Ore and    denominator sets of $R$, respectively. For
$S\in \Den (R)$, $$S^{-1}R\simeq RS^{-1}$$ (an $R$-isomorphism)
 is  the {\em  localization} of the ring $R$ at $S$, and $\ass (S):=\ass_l(S) = \ass_r(S)$. 

\begin{proposition} \label{aA12Mar15} %\marginpar{aA12Mar15}
(\cite[Proposition 2.7]{Bav-Lu-BL-qAge}.) 
Let $R$ be a Noetherian ring and $s$ be an element of $R$ such that $S_s := \{s^i \,|\, i\in \N \}$ is a left denominator set  of the ring $R$ and $(s^i)= (s)^i$ for all $i\geqslant 1$ (e.g., $s$ is a normal element such that $\ker(\cdot s_{R}) \subseteq \ker(s_R \cdot)$). Then
$\Spec\,(R)= \Spec(R, s) \, \sqcup \, \Spec_s(R) $ where $\Spec(R,s):= \{ \gp \in \Spec\,(R)\,|\, s \in \gp \}$,  $\Spec_s(R)= \{\gq \in \Spec\,(R)\, |\, s \notin \gq \}$ and
\begin{enumerate}%[label=(\alph*)]
\item  the map $\Spec\,(R, s) \rightarrow \Spec\,(R/(s)),\,\, \gp \mapsto \gp/(s)$, is a bijection with the inverse $\gq \mapsto \pi^{-1}(\gq)$ where $\pi: R \rightarrow R/(s), r \mapsto r+(s),$
\item the map $\Spec_s(R) \rightarrow \Spec\,(R_s), \,\, \gp \mapsto S_s^{-1}\gp,$ is a bijection with the inverse $\gq \mapsto \sigma^{-1}(\gq)$ where $\sigma: R \rightarrow R_s:= S_s^{-1}R, \,r \mapsto \frac{r}{1}$. 
\item For all $\gp \in \Spec\,(R, s)$ and $\gq \in \Spec_s(R), \gp \not\subseteq \gq.$
\end{enumerate}
\end{proposition}

{\bf Classifications of prime ideals and simple modules of the algebra $\CB $.} 
 The centre of the algebra $\CB$  is equal to
%\marginpar{qCAZmA}
\begin{equation}\label{qCAZmA}
Z(\CB )=K[s^{\pm 1},t^{\pm 1}]=K[h^{\pm n}, x^{\pm n}] \;\; {\rm where} \;\; s:=h^n\;\; {\rm and}\;\;  t:=x^n.
\end{equation}
Clearly, $ Z(\CB)\supseteq Z:=K[s^{\pm 1},t^{\pm 1}]$ (since $xs=q^nsx=sx$ and $ht=q^{-n}th=th$).  The algebra $\CB$ is a Noetherian domain of Gelfand-Kirillov dimension $2$ and 
 %\marginpar{qCAmASum}
\begin{equation}\label{qCAmASum}
 \CB =\bigoplus_{i,j=0}^{n-1}Z h^ix^j.
\end{equation}

Since $\o_h(x^j)=q^{-j}x^j$ for all $j\geq 0$ and the elements $1, q^{-1}, \ldots , q^{-n+1}$ are distinct elements of the field $K$, the algebra 
%\marginpar{qCAmASum1}
\begin{equation}\label{qCAmASum1}
\CB =\bigoplus_{j=0}^{n-1}\Bigg(\bigoplus_{i=0}^{n-1}Z h^i\Bigg)x^j
\end{equation}
is a direct sum of eigenspaces $\Big(\bigoplus_{i=0}^{n-1}Z h^i\Big)x^j$, $0\leq j\leq n-1$, of the inner automorphism  $\o_h $ of the algebra $\CB$,  and the set of eigenvalues of $\o_h$ is $\Ev_\CB (\o_h)=\{ 1, q^{-1}, \ldots , q^{-n+1}\}$. 

Similarly, $\o_x(h^j)=q^jh^j$ for all $j\geq 0$ and the elements $1, q, \ldots , q^{n-1}$ are distinct elements of the field $K$, the algebra 
%\marginpar{qCBxh}
\begin{equation}\label{qCBxh}
\CB =\bigoplus_{i=0}^{n-1}\Bigg(\bigoplus_{j=0}^{n-1}Zx^j \Bigg)h^i
\end{equation}
is a direct sum of eigenspaces $\Big(\bigoplus_{j=0}^{n-1}Z(\CB)  x^j\Big) h^i$, $0\leq i\leq n-1$,  of the inner automorphism  $\o_x $ of the algebra $\CB$, and the set of eigenvalues of $\o_x$ is $\Ev_\CB (\o_x)=\{ 1, q, \ldots , q^{n-1}\}$. 

So, the vector spaces   $Zh^ix^j$ in  (\ref{qCAmASum}) are common eigenspaces for the pair of inner automorphisms $\o_h$ and $\o_x$  of $\CB$  that correspond to distinct pairs $(q^{-j}, q^i)$ of eigenvalues, and (\ref{qCAZmA}) follows.

 Since $\CB= \mA_{xh}=\mA_{st}$, the map
%\marginpar{qCA1Sum4}
\begin{equation}\label{qCA1Sum4}
\Spec_{st} (Z(\mA ))=\Spec (Z(\mA ))\backslash \Big(V(s)\cup V(t)\Big)\ra \Spec (Z(\CB)), \;\;\gr \mapsto \gr':= Z(\CB )\gr 
\end{equation}
is a bijection with  inverse $\gr' \mapsto \gr :=Z(\mA )\cap \gr'$. Clearly,  $k(\gr'):=Q( Z(\CB )/\gr')=Q( Z(\mA )/\gr)=k(\gr )$. 
For an algebra $A$, let $\CI (A)$ be the set of its ideals. 

Theorem \ref{qCAc28Aug20}.(1) and Theorem \ref{qCAc28Aug20}.(2) are  explicit descriptions of all the ideals and all prime ideals of the algebra $\CB$. Theorem \ref{qCAc28Aug20}.(4) shows that the quotient rings of prime factor algebras of  the algebra $\CB$ are central simple $n^2$-dimensional algebras.

\begin{theorem}\label{qCAc28Aug20}%\marginpar{qCAc28Aug20}

\begin{enumerate}
\item The map $\CI (Z(\CB))\ra \CI (\CB)$, $\ga \mapsto \CB\ga$ is a bijection with  inverse $ I\mapsto Z(\CB)\cap I$. In particular, $\CI (\CB)=\{ \CB\ga \, | \, \ga \in \CI (Z(\CB))\}$ and every nonzero ideal of $\CB$ meets the centre of $\CB$.
\item   $\Spec (\CB )=\{\CB \gr' \, | \,   \gr' \in \Spec\, Z(\CB )\}$ where $Z(\CB )=\Spec \,K[s^{\pm 1},t^{\pm 1}]$,   $s=h^n$ and $ t=x^n$. 
\item   $\Max (\CB ) = \{ \CB \gm\,|\, \gm \in \Max (Z(\CB )) \}$.

\item For each  prime ideal $\gr' \in \Spec\, Z(\CB )$, the quotient ring  of the algebra $\CB / \CB \gr'$, $$Q(\CB /\CB \gr')=\bigoplus_{i,j=0}^{n-1}k(\gr' )h^ix^j\simeq k(\gr' )\t_{Z(\CB )}\CB,$$ is a central simple $n^2$-dimensional algebra over the field $k(\gr' )$ of fractions of the commutative domain $Z(\CB )/\gr'$. The algebra $Q(\CB /\CB \gr')$ is isomorphic to the algebra $\CE (\gr'):=(E(s,\gr'), \s, a)$ where $E(s,\gr'):=k(\gr' )[h]/(h^n-s)\simeq k(\gr' )[h,h^{-1}]/(h^n-s)$ and  $\s (h)=qh$.

\item The quotient ring  of the algebra $\CB$, $$Q(\CB )=\bigoplus_{i,j=0}^{n-1}K(s,t)h^ix^j\simeq K(s,t)\t_{Z(\CB )}\CB,$$ is a central simple $n^2$-dimensional division algebra over the field $K(s,t)$ of rational functions in two variables. The algebra $Q(\CB )$  is isomorphic to the cyclic algebra $\CE (0):=(E(s,0), \s, a)$ where $E(s,0):=K(s,t)[h]/(h^n-s)\simeq K(s,t)[h,h^{-1}]/(h^n-s)$ and  $\s (h)=qh$.
\end{enumerate}
\end{theorem}

\begin{proof}   In view of (\ref{qCAmASum}),  it suffices to show that for each ideal $\gb$ of $\CB$, $\gb = \CB  \ga$ where $\ga = Z(\CB )\cap \gb$. 

(i) $\gb = \bigoplus_{i,j=0}^{n-1}\Big(\gb \cap \ Z (\CB) h^ix^j\Big)$: The sum $\CB =  \bigoplus_{i,j=0}^{n-1}Z (\CB) h^ix^j$ is a direct sum of common eigenspaces for the commuting pair $(\o_h , \o_x)$ of inner automorphisms of the algebra $\CB$ since for all elements $z\in Z (\CB)$,
$$ \o_h(zh^ix^j)=q^{-j}zh^ix^j\;\; {\rm and}\;\;  \o_x(zh^ix^j)=q^izh^ix^j. $$
Now, the statement (i) follows from the fact that the elements $\{ (q^{-j}, 	q^i)\, | \, i,j=0, \ldots , n-1\}$ are distinct and the elements $x$ and $h$ are units of the algebra $\CB$.

(ii) $\gb \cap \ Z (\CB) h^ix^j=\ga h^ix^j$: The statement (ii) follows from the inclusions  $$h^{-i}(\gb \cap  Z(\CB) h^ix^j)x^{-j}\subseteq \ga\;\; {\rm and}\;\;  \ga h^ix^j\subseteq  \gb \cap  Z(\CB)h^i x^j.$$
2--4. Statements 2--4 follow from statement 1. 

5. Statement 5 is a particular case of  statement 4.   \end{proof}

  Corollary \ref{aqAc28Aug20}.(1) is a classification of primitive ideals of the algebra $\CB$. Every primitive  ideal of the algebra $\CB$ is a maximal ideal (and vice versa) and co-finite (Corollary \ref{aqAc28Aug20}.(2)). There is a one-to-one correspondence between the set of primitive ideals of the algebra $\CB$ and the set of isomorphism classes of simple $\CB$-modules (Corollary \ref{aqAc28Aug20}.(4)). Every simple $\CB$-module is finite dimensional over the field $K$.  
  %*** Corollary \ref{aqAc28Aug20}.(2) classifies simple $\CB$-modules and  describes their dimensions and endomorphism rings. ***  
  For an algebra $A$, we denote by $\hA$ the set of isomorphism classes of simple (left) $A$-modules. 

\begin{corollary}\label{aqAc28Aug20}%\marginpar{aqAc28Aug20}

\begin{enumerate}
\item  $\Prim (\CB )=\Max (\CB ) = \{ \CB \gm\,|\, \gm \in \Max (Z(\CB )) \}$.

\item For each maximal ideal $\gm \in \Prim (\CB )=\Max (\CB )$, the factor algebra 
 $$\CB / \CB \gm =Q(\CB /\CB \gm)=\bigoplus_{i,j=0}^{n-1}k(\gm )h^ix^j\simeq k(\gm )\t_{Z(\CB )}\CB,$$ is a central simple $n^2$-dimensional algebra over the field $k(\gm )=Z(\CB )/\gm$. The algebra $\CB /\CB \gm$ is isomorphic to the algebra $\CE (\gm )=(E(s,\gm), \s, a)$ where $E(s,\gm):=k(\gm )[h]/(h^n-s)\simeq k(\gm )[h,h^{-1}]/(h^n-s)$ and  $\s (h)=qh$. Let $U(\gm):=U(\CE(\gm ))$ be a unique simple  $\CE (\gm )$-module (described in (\ref{xu=usxe}) and Theorem \ref{12Jun23}.(1,2)). 

\item $\widehat{\CB }=\{ U(\gm)\, |\, \gm \in \Max (Z(\CB)) \}$. All simple $\CB$-modules are finite dimensional.

\item The map $\widehat{\CB} \ra {\rm Prim} (\CB ) =\Max (\CB )$, $U\mapsto \ann_{\CB} (U)$ is a bijection with  inverse $\CB \gm \mapsto U(\gm)$. 

\item For $\gm \in \Max (Z(\CB))$, $\End_{\CB}(U(\gm))\simeq D(\gm)$ where $\CE (\gm)\simeq M_{n(\gm)}(D(\gm ))$ for some natural number $n(\gm)$ and a division ring $D(\gm)$, see Theorem \ref{12Jun23}. We write endomorphism on the right. 
\end{enumerate}
\end{corollary}

\begin{proof}   1.  Statement 1 follows from the description of ideals of the algebra $\CB$  (Theorem \ref{qCAc28Aug20}.(1)).

2. Statement 2 follows from Theorem \ref{qCAc28Aug20}.(4).

3 and 4. Statements 3  and 4 follow from statements 1 and  2. 

5. Statement 5 follows from statement 3 and Theorem \ref{12Jun23}.  \end{proof}

{\bf Classification of completely prime ideals of the algebra $\CB$.}  
Recall that a prime ideal $\gp$ of a ring $R$ is called a completely prime ideal if the ring $R/\gp$ is a domain. The set of all completely prime ideals of the ring $R$ is denoted by $\Spec_c(R)$. Proposition  \ref{A20Jul23} is a classification of completely prime ideals of the algebra $\CB$.

\begin{proposition}\label{A20Jul23}%\marginpar{A20Jul23}
$\Spec_c (\CB )=\{ \{ 0\}, \CB\gr' \, | \, 0\neq \gr' \in \Spec \,Z(\CB)$,  $E(s,\gr')=k(\gr')[h]/(h^n-s)$
is a field and $n=\min\{ d'\geq 1\, | \,  d'\mid n$  there is a matrix $X\in M_{d'}(E(s,\gr'))$ such that $ X^{\s^{n-1}}\cdots X^\s X=a\}\}$ where $\s (h)=qh$.
\end{proposition}
 
\begin{proof}    A prime ideal $P$ of the algebra $\CB$ is a completely prime ideal if and only if the algebra $Q(\CB /  P)$ is a domain.
 By Theorem \ref{qCAc28Aug20}.(5), $\{0\}\in \Spec_c (\CB )$.
  
If $P\neq \{ 0\}$ then $P=\CB\gr'$ for some nonzero  prime ideal $\gr'\in \Spec \, Z(\CB)$.
By  Theorem \ref{qCAc28Aug20}.(4),
$$Q(\CB/\CB\gr')\simeq \CE (\gr')=(E(s,\gr'),\s , a)\;\; {\rm where}\;\; E(s,\gr')=k(\gr')[h]/(h^n-s)\;\; {\rm   and}\;\;\s (h) = qh.$$
 Now, the proposition  follows from  Corollary \ref{b17Jun23}.   \end{proof}

{\bf The skew polynomial algebra $\mA = K[h][x;\s ]$, $\s (h)=qh$,  and its prime spectrum.} Recall that  $\mA = K[h][x;\s ]$ is a skew polynomial algebra  where $K[h]$ is a polynomial algebra in the  variable $h$ and  the automorphism $\s$ of $K[h]$ is given by the rule $\s (h) = qh$. The algebra $\mA$ is a Noetheiran domain of Gelfand-Kirillov dimension $2$. The elements $x$ and $h$ are  regular normal elements of the algebra $\mA$ and the corresponding inner  automorphisms $\o_x$ and $\o_h$ of $\mA$ are given by the rule $\o_x:h\mapsto qh$, $x\mapsto x$ and $\o_h: h\mapsto h$, $x\mapsto q^{-1}x$. The automorphisms $\o_x$ and $\o_h$ have order $n$ ($q$ is a primitive $n$'th root of unity). 
Notice that 
%\marginpar{CB=Arxi}
\begin{equation}\label{CB=Arxi}
\CB \simeq \mA_{xh}=\mA_{x^nh^n}= \mA_{st}=\mA_{(xh)^n}
\end{equation}
since $(xh)^n=q^{1+2\cdots + (n-1)}h^nx^n=q^{\frac{n(n-1)}{2}}st=(-1)^{n-1}st\in Z(\mA )$, by (\ref{xnix2}). 

The centre of the algebra $\mA$  is equal to
%\marginpar{qZmA}
\begin{equation}\label{qZmA}
Z(\mA )=A\cap Z(\CB )=K[s,t]=K[h^n, x^n]\;\; {\rm where} \;\; s=h^n,\;\;   t=x^n
\end{equation}
and $K[s,t]$ is a polynomial algebra in two variables $s$ and $t$:
\begin{eqnarray*}
K[s,t] &\subseteq &Z(\mA)\subseteq \mA\cap Z(\mA_{xh})=\mA \cap  Z(\CB)\stackrel{(\ref{qCAZmA})}{=}\mA \cap K[s^{\pm 1},t^{\pm 1}]\\
 &\stackrel{(\ref{qCAmASum})}{=}&K[s,t]\cap K[s^{\pm 1},t^{\pm 1}]=K[s,t].
\end{eqnarray*}

 The quantum Weyl algebra $A_1$ is a Noetherian domain of Gelfand-Kirillov dimension $2$. The localization $\mA_x$ of the algebra $\mA$ at the powers of the regular normal element $x$ 
 is isomorphic to the localization $A_{1,x}$ of the quantum  Weyl algebra $A_1$ at the powers of the element $x$:
 $$\mA_x=K[h][x, x^{-1}; \s ]\ra A_{1,x}, \;\; x\mapsto x,\;\; h\mapsto y x+\frac{1}{q-1}.$$
By (\ref{qZmA}), 
%\marginpar{qmASum}
\begin{equation}\label{qmASum}
 \mA =\bigoplus_{i,j=0}^{n-1}Z(\mA ) h^ix^j.
\end{equation}

Since $\o_h (x^j)=q^{-j}x^j$ for all $j\geq 0$ and the elements $1,q^{-1},  \ldots , q^{-n+1}$ are distinct elements of the field $K$, the algebra 
%\marginpar{qmASum1}
\begin{equation}\label{qmASum1}
\mA =\bigoplus_{j=0}^{n-1}\Bigg(\bigoplus_{i=0}^{n-1}Z(\mA ) h^i\Bigg)x^j
\end{equation}
is a direct sum of eigenspases $\Big(\bigoplus_{i=0}^{n-1}Z(\mA ) h^i\Big)x^j$, $0\leq j\leq n-1$, of the automorphism $\o_h $ of the algebra $\mA$ and the set of eigenvalues of $\o_h$ is $\Ev_\mA (\o_h)=\{ q^{-j}\, | \, j=0,1, \ldots , n-1\}= \{ q^j\, | \, j=0,1, \ldots , n-1\}$.

Similarly, since $\o_x (h^i)=q^ih^i$ for all $i\geq 0$ and the elements $1,q,  \ldots , q^{n-1}$ are distinct elements of the field $K$, the algebra 
%\marginpar{qmASum1-ox}
\begin{equation}\label{qmASum1-ox}
\mA =\bigoplus_{i=0}^{n-1}\Bigg(\bigoplus_{j=0}^{n-1}Z(\mA ) x^j\Bigg)h^i
\end{equation}
is a direct sum of eigenspases $\Bigg(\bigoplus_{j=0}^{n-1}Z(\mA ) x^j\Bigg)h^i$, $0\leq i\leq n-1$, of the automorphism $\o_x $ of the algebra $\mA$ and the set of eigenvalues of $\o_x$ is $\Ev_\mA (\o_x)=\{ q^i\, | \, j=0,1, \ldots , n-1\}$. \\

{\bf The prime spectrum of the algebra $\mA$.} The prime spectrum of the algebra $\mA$ is given in Theorem \ref{qAc28Aug20}. In the diagram below, the diagrams $$
  \begin{tikzpicture}
\node (q) at  (0,.5) {$\gq$};
\node (p)  at (0,0)  {$\gp$};
\draw [thick, shorten <=-2pt, shorten >=-2pt] (q) -- (p);
\end{tikzpicture} 
\;\; \;\;\;\; \;\;
\begin{tikzpicture}
\node (h) at  (0,.6) {$\CH$};
\node (g)  at (0,0)  {$\CG$};
\draw [dotted, thick, shorten <=-2pt, shorten >=-2pt] (h) -- (g);
\end{tikzpicture} 
$$
mean  $\gp \subsetneq \gq$ and  obvious inclusions between the sets of primes in sets $\CG$ and $\CH$, respectively. For a prime ideal $\gp$, we denote by ${\rm ht}(\gp )$ its height. The set of maximal ideals of an algebra $R$ is denoted by $\Max (R)$. For a non-empty subset $S$ of $R$, we denote by $V(S)$  the set of all prime ideals of $R$ that contain the set $S$.

\begin{theorem}\label{qAc28Aug20}%\marginpar{qAc28Aug20}

\begin{enumerate}
 \item  $\Spec (\mA )=\{\{ 0\}, (x), (h), (x, h)\}  )\, \sqcup \, \mX\, \sqcup \,\mM\, \sqcup \,\mathbb{H}\, \sqcup \,\CN$ where
 \begin{eqnarray*}
\mX &:=&\{ (x,\gq )\, | \,  \gq \in \Spec \,K[h]\backslash \{ 0, (h)\}\}, \\
\mM &:=& 	\{ \mA \gm\,|\, \gm \in \Max (Z(\mA )), s, t\not\in \gm \}, \\
\mathbb{H} &:=& \{ (h,\gq' )\, | \, \gq' \in \Spec \,K[x]\backslash \{ 0, (x)\}\}, \\
\CN &:=& \{\mA\gn \,|\, \gn \in \Spec\, Z(\mA ), {\rm ht} (\gn ) =1, \gn \neq (t), (s) \}, 
\end{eqnarray*}
 and $Z(\mA )=K[s,t]$ is a polynomial algebra in two variables $s=h^n$ and $ t=x^n$,   
\begin{align}
	\begin{tikzpicture}[scale=1.3]
	\node (I) at (-4,2) {\small$\mX$};
	\node (D) at (-4, 1) {\small$(x)$};	
	\node (p) at (0,1) { \small$\CN$};
	\node (m) at (1,2) {\small$\mM $};
	\node (m') at (-1,2) {\small$\{ (x,h)\}$};
	\node (0) at (0,0) {\small$0$};
	\node (I') at (4,2) {\small$\mathbb{H}$};
	\node (D') at (4, 1) {\small$(h)$};	
	\draw [ shorten <=-2pt, shorten >=-2pt] (0)--(D)--(I);
	\draw [ shorten <=-2pt, shorten >=-2pt] (0)--(D')--(I');
	\draw [ shorten <=-2pt, shorten >=-2pt] (D);
	\draw [ shorten <=-2pt, shorten >=-2pt] (0)--(p);
	\draw [semithick, dotted][ shorten <=-2pt, shorten >=-2pt]  (m)--(p)--(I);
	\draw [semithick, dotted][ shorten <=-2pt, shorten >=-2pt]  (p)--(I');
	\draw [shorten <=-2pt, shorten >=-2pt]  (D)--(m')--(D');
	\end{tikzpicture}  \label{qXYZA} %\marginnote{qXYZA}
	\end{align}
	\item  Every nonzero  ideal of the algebra $\mA$ meets the centre  $Z(\mA )$.
\item  $\Max (\mA ) =\mX \, \sqcup \, \{(x, h)\} \, \sqcup \,\mM\, \sqcup \,\mathbb{H}$.

\end{enumerate}
\end{theorem}

\begin{proof}   1. Recall that $\CB=\mA_{xh}=\mA_{st}$ where $st\in Z(\mA)$, the element $xh$ is a normal element of $\mA$,  and the element $st$ is a central  element of $\mA$. By Proposition \ref{aA12Mar15},
$$ \Spec (\mA)=\Spec(\mA, xh)\, \sqcup \,\Spec_{xh}(\mA).$$
Clearly, $\Spec(\mA, xh)=\{ (x), (h), (x,h)\}\, \sqcup \, \mX\, \sqcup \,\mathbb{H}$. By Proposition \ref{aA12Mar15}.(2), Theorem \ref{qCAc28Aug20}.(2),  and (\ref{qCA1Sum4}),
$$ \Spec_{xh}(\mA) = \{ \mA\cap \CB \gp\, | \, \gp \in \Spec \, K[s^{\pm 1},t^{\pm 1}]\} = \{ \mA \gr \, | \, \gr\in \Spec \, K[s,t],  s,t\not\in\gr\}$$
since  $\mA\cap \CB \gp=\mA(Z(\mA)\cap \gp)=\mA \gr$ where $\gr =Z(\mA)\cap \gp$. Now, the first equality of statement 1 follows. The containment information in the diagram follows from Proposition  
\ref{aA12Mar15}.(3). 

2 and 3. Statements 2 and 3 follow from statement 1.   \end{proof}

{\bf The quotient rings of prime factor algebras of $\mA$ and a classification of simple $\mA$-modules.}  Corollary \ref{QqAc28Aug20} describes the quotient rings of prime factor algebras of the algebra $\mA$.

\begin{corollary}\label{QqAc28Aug20}%\marginpar{QqAc28Aug20}

\begin{enumerate}
\item For each $(x,\gq )\in \mX$, $\mA/(x,\gq )=
Q(\mA/(x,\gq ))=K[h]/\gq$ is a field.

\item For each $(h,\gq' )\in \mathbb{H}$, $\mA/(h,\gq' )=
Q(\mA/(h,\gq' ))=K[x]/\gq'$ is a field. 

\item $\mA/(x, h)=Q(\mA/(x, h))=K$.

\item For each prime ideal $\mA\gr \in \CN \cup \mM$ (see Theorem \ref{qAc28Aug20}.(1)), the quotient ring  of the algebra $\mA / \mA \gr$, $$Q(\mA /\mA \gr)=\bigoplus_{i,j=0}^{n-1}k(\gr )h^ix^j\simeq k(\gr )\t_{Z(\mA )}\mA,$$ is a central simple $n^2$-dimensional algebra over the field $k(\gr )$ of fractions of the commutative domain $Z(\mA )/\gr$.
The algebra $Q(\mA /\mA \gr)$ is isomorphic to the algebra $\CE (\gr )=(E(s,\gr), \s, a)$ where $E(s,\gr):=k(\gr )[h]/(h^n-s)$ and  $\s (h)=qh$. 

\item The quotient ring  of the algebra $\mA$, $$Q(\mA)=Q(\CB)=\bigoplus_{i,j=0}^{n-1}K(s,t)h^ix^j\simeq K(s,t)\t_{Z(\mA )}\mA,$$ is a central simple $n^2$-dimensional algebra over the field $Z(Q(\mA))=K(s,t)$ of rational functions in two variables.  The algebra $Q(\mA )$  is isomorphic to the cyclic algebra $\CE (0):=(E(s,0), \s, a)$ where $E(s,0):=K(s,t)[h]/(h^n-s)$ and  $\s (h)=qh$.

\item $\mA/(x)=K[h]$ and $Q(\mA/(x))=K(h)$.
\item $\mA/(h)=K[x]$ and $Q(\mA/(h))=K(x)$.
\end{enumerate}
\end{corollary}

\begin{proof}   1--3, 6 and 7.  Statements 1--3, 6 and 7 follow from Theorem \ref{qAc28Aug20}.(1).

4. Statement 4 follows from Theorem \ref{qCAc28Aug20}.(4), since $Q(\mA/ \mA \gr ) = Q(\CB / \CB \gr)$.

5. Statement 5 is a particular case of  statement 4. \end{proof}

  Corollary \ref{aQqAc28Aug20}.(3) is a classification of primitive ideals of the algebra $\mA$. Every primitive  ideal of the algebra $\mA$ is a maximal ideal (and vice versa) and co-finite (Corollary \ref{aQqAc28Aug20}.(4)). There is a one-to-one correspondence between the set of primitive ideals of the algebra $\mA$ and the set of isomorphism classes of simple $\mA$-modules (Corollary \ref{aQqAc28Aug20}.(5)). Every simple $\mA$-module is finite dimensional over the field $K$.  %*** Corollary \ref{aQqAc28Aug20}.(3) classifies simple $\mA$-modules and  describes their dimensions and endomorphism rings. ***  

\begin{corollary}\label{aQqAc28Aug20}%\marginpar{aQqAc28Aug20}
The algebra $\mA$ satisfies the assumptions of Corollary \ref{b2Aug23}, ${\rm ZL}(\mA )=\Spec (Z(\mA ))$, and Theorem \ref{2Aug23} holds:
 \begin{enumerate}
\item The ring $\mA $ is PLM-ring with ${\rm PL}(\mA )=\Max (Z(\mA ))$.

\item  For every $\gm \in \Max (Z(\mA ))$, $\Prim (\mA ,\gm)=\{ \mA \gm\}$.

\item $\Prim(\mA )=\Max (\mA )=\{ \mA \gm\, | \, \gm\in \Max (Z(\mA ))\}=\mX \, \sqcup \, \{(x, h)\} \, \sqcup \,\mM\, \sqcup \,\mathbb{H}$.

 \item For each maximal ideal $\gm \in\mM$, the factor algebra 
 $$\mA / \mA \gm =Q(\mA /\mA \gm)=\bigoplus_{i,j=0}^{n-1}k(\gm )h^ix^j\simeq k(\gm )\t_{Z(\mA )}\mA,$$ is a central simple $n^2$-dimensional algebra over the field $k(\gm )=Z(\mA )/\gm$. The algebra $\mA /\mA \gm$ is isomorphic to the algebra $\CE (\gm )=(E(s,\gm), \s, a)$ where $E(s,\gm):=k(\gm )[h]/(h^n-s)$ and  $\s (h)=qh$. Let $U(\gm)$ be a unique simple  $\CE (\gm )$-module (described in (\ref{xu=usxe}) and Theorem \ref{12Jun23}.(1,2))). 

\item $\widehat{\mA } =\{ U(\gm )\, | \, \gm \in \Max(Z(\mA ))\}$  where $U(\gm)$ is a unique simple module of the simple artinian ring $k(\gm)\t_{Z(\mA )}\mA /\mA \gm$
$$
U(\gm):=\begin{cases}
K[h]/\gq & \text{if }\gm=(x,\gq)\in \mX,\\
K& \text{if }\gm = (x,h),\\
 & \text{ is described in (\ref{xu=usxe}) and Theorem \ref{12Jun23}.(1,2))  if $\gm \in \mM$},\\
K[x]/\gq' & \text{if }\gm=(h,\gq')\in \mathbb{H}.
\end{cases}
$$
All simple $\mA$-modules are finite dimensional.

\item The map $\widehat{\mA } \ra \Prim (\mA )$, $U(\gm )\mapsto \ann_{\mA }(U(\gm))=\mA /\mA \gm$ is a bijection with inverse $\mA \gm\mapsto U(\gm)$.

\item For $\gm \in \Max (Z(\mA ))$, 
$$
\End_{\mA}(U(\gm)):=\begin{cases}
K[h]/\gq & \text{if }\gm=(x,\gq)\in \mX,\\
K& \text{if }\gm = (x,h),\\
D(\gm)  & \text{ is described in  Theorem \ref{12Jun23}  if $\gm \in \mM$},\\
K[x]/\gq' & \text{if }\gm=(h,\gq')\in \mathbb{H}.
\end{cases}
$$
\end{enumerate}

\end{corollary}

\begin{proof} By Corollary \ref{QqAc28Aug20}, the  algebra $\mA $ satisfies the assumptions of Corollary \ref{b2Aug23}, and so  ${\rm ZL}(\mA )=\Spec (Z(\mA ))$, and Theorem \ref{2Aug23} holds, i.e. the corollary holds apart from statement 4, 7,  and an explicit description of the simple modules in statement 5.  Statement 4 follows from Corollary  \ref{QqAc28Aug20}.(4).
 The explicit description of the simple modules in statement 5 follows from Corollary \ref{QqAc28Aug20}.(1--3)  and  statements 1 and  2. Statement 7 follows from statement 5 and Corollary \ref{QqAc28Aug20}.(1--4). 
 \end{proof}

{\bf Classification of completely prime ideals of the algebra $\mA$.}   Proposition  \ref{B20Jul23} is a classification of completely prime ideals of the algebra $\mA$.

\begin{proposition}\label{B20Jul23}%\marginpar{B20Jul23}
$\Spec_c (\mA )=\Big\{ \{ 0\}, (x), (h), (x,h)\} \, \sqcup \, \mX\, \sqcup \, \mathbb{H}\,  \sqcup \, \{  \CB\gr\in \CN\cup \mM \, | \,   E(s,\gr)=k(\gr)[h]/(h^n-s)$
is a field and $n=\min\{ d'\geq 1\, | \,  d'\mid n$  there is a matrix $X\in M_{d'}(E(s,\gr))$ such that $ X^{\s^{n-1}}\cdots X^\s X=a\}\Big\}$ where $\s (h)=qh$.
\end{proposition}
 
\begin{proof}    A prime ideal $P$ of the algebra $\mA$ is a completely prime ideal if and only if the algebra $Q(\mA /  P)$ is a domain.
 By Corollary \ref{QqAc28Aug20}.(1--3,6,7), $$\Spec_c (\mA )\supseteq \Big\{ \{ 0\}, (x), (h), (x,h)\} \, \sqcup \, \mX\, \sqcup \, \mathbb{H}.$$
  
If $P=\mA\gr\in \CN\cup \mM$ then, by  Corollary  \ref{QqAc28Aug20}.(4),
$$Q(\mA/\mA\gr )\simeq \CE (\gr)=(E(s,\gr),\s , a)\;\; {\rm where}\;\; E(s,\gr)=k(\gr)[h]/(h^n-s)\;\; {\rm   and}\;\;\s (h) = qh.$$
 Now, the proposition  follows from  Corollary \ref{b17Jun23}.   \end{proof}

{\bf The  algebra $\CA = K[h][x^{\pm 1};\s ]$, $\s (h)=qh$,  and its prime spectrum.} Recall that  $\CA = K[h][x^{\pm 1};\s ]$ is a skew polynomial algebra  where $K[h]$ is a polynomial algebra in the  variable $h$ and  the automorphism $\s$ of $K[h]$ is given by the rule $\s (h) = qh$, and $\CA /(h)\simeq K[x^{\pm 1}]$. The algebra $\CA$ is a Noetheiran domain of Gelfand-Kirillov dimension $2$. The element $h$ is a  regular normal element of the algebra $\CA$ and the corresponding  automorphism $\o_h$ of $\CA$ is  given by the rule $\o_h: h\mapsto h$, $x\mapsto q^{-1}x$. The automorphism $\o_h$ has order $n$ ($q$ is a primitive $n$'th root of unity). Recall that $\CA = \mA_x$, $s=h^n$, and $t=x^n$. 
By (\ref{CB=Arxi}),  
%\marginpar{CACB=Arxi}
\begin{equation}\label{CACB=Arxi}
\CB \simeq \mA_{xh}=\Big(\mA_x \Big)_h=\CA_h.
\end{equation}
The centre of the algebra $\mA$  is equal to
%\marginpar{CAqZmA}
\begin{equation}\label{CAqZmA}
Z(\CA )=\CA\cap Z(\CB )=K[s,t^{\pm 1}]=K[h^n, x^{\pm n}].
\end{equation}
Indeed,
\begin{eqnarray*}
K[s,t^{\pm 1}] &\subseteq &Z(\CA)\subseteq \CA\cap Z(\CA_x)=\CA \cap  Z(\CB)\stackrel{(\ref{qCAZmA})}{=}\CA \cap K[s^{\pm 1},t^{\pm 1}]\\
 &\stackrel{(\ref{qZmA}), (\ref{qmASum})}{=}&\bigg( \bigoplus_{i,j=0}^{n-1}K[s,t^{\pm 1}]h^ix^j\bigg)\cap K[s^{\pm 1},t^{\pm 1}]= K[s,t^{\pm 1}]\cap K[s^{\pm 1},t^{\pm 1}]=K[s,t^{\pm 1}].
\end{eqnarray*}
Now, by (\ref{qZmA}) and  (\ref{qmASum}), 
%\marginpar{CAqmASum}
\begin{equation}\label{CAqmASum}
 \CA =\bigoplus_{i,j=0}^{n-1}Z(\CA ) h^ix^j.
\end{equation}
Since $\o_h (x^j)=q^{-j}x^j$ for all $j\geq 0$ and the elements $1,q^{-1},  \ldots , q^{-n+1}$ are distinct elements of the field $K$, the algebra 
%\marginpar{CAqmASum1}
\begin{equation}\label{CAqmASum1}
\CA =\bigoplus_{j=0}^{n-1}\Bigg(\bigoplus_{i=0}^{n-1}Z(\CA ) h^i\Bigg)x^j
\end{equation}
is a direct sum of eigenspases $\Big(\bigoplus_{i=0}^{n-1}Z(\CA ) h^i\Big)x^j$, $0\leq j\leq n-1$, of the automorphism $\o_h $ of the algebra $\CA$ and the set of eigenvalues of $\o_h$ is $\Ev_\CA (\o_h)=\{ q^{-j}\, | \, j=0,1, \ldots , n-1\}= \{ q^j\, | \, j=0,1, \ldots , n-1\}$.

Similarly, since $\o_x (h^i)=q^ih^i$ for all $i\geq 0$ and the elements $1,q,  \ldots , q^{n-1}$ are distinct elements of the field $K$, the algebra 
%\marginpar{CAqmASum1-ox}
\begin{equation}\label{CAqmASum1-ox}
\CA =\bigoplus_{i=0}^{n-1}\Bigg(\bigoplus_{j=0}^{n-1}Z(\CA ) x^j\Bigg)h^i
\end{equation}
is a direct sum of eigenspases $\Bigg(\bigoplus_{j=0}^{n-1}Z(\CA ) x^j\Bigg)h^i$, $0\leq i\leq n-1$, of the inner automorphism $\o_x $ of the algebra $\CA$ and the set of eigenvalues of $\o_x$ is $\Ev_\CA (\o_x)=\{ q^i\, | \, j=0,1, \ldots , n-1\}$.  Hence, the vector spaces   $Z(\CA) h^ix^j$ in  (\ref{CAqmASum}) are common eigenspaces for the pair of  automorphisms $\o_h$ and $\o_x$  of $\CA$  that correspond to distinct pairs $(q^{-j}, q^i)$ of eigenvalues.

 Since $\CB= \CA_h=\CA_s$, the map
%\marginpar{CAqCA1Sum4}
\begin{equation}\label{CAqCA1Sum4}
\Spec_{s} (Z(\CA ))=\Spec (Z(\CA ))\backslash V(s)\ra \Spec (Z(\CB)), \;\;\gr \mapsto \gr':= Z(\CB )\gr 
\end{equation}
is a bijection with  inverse $\gr' \mapsto \gr :=Z(\CA )\cap \gr'$. Clearly,  $k(\gr'):=Q( Z(\CB )/\gr')=Q( Z(\CA )/\gr)=k(\gr )$.\\

{\bf The prime spectrum of the algebra $\CA$.} The prime spectrum of the algebra $\CA$ is given in Theorem \ref{CAqAc28Aug20}.  
\begin{theorem}\label{CAqAc28Aug20}%\marginpar{CAqAc28Aug20}

\begin{enumerate}
 \item  $\Spec (\CA )=\{\{0\},  (h) \}  )\, \sqcup \,  \mM'\, \sqcup \,\mathbb{H}'\, \sqcup \,\CN'$ where
 \begin{eqnarray*}
\mM' &:=& 	\{ \CA \gm\,|\, \gm \in \Max (Z(\CA )), s\not\in \gm \}, \\
\mathbb{H}' &:=& \{ (h,\gq' )\, | \, \gq' \in \Spec \,K[x^{\pm 1}]\backslash \{ 0\}\}, \\
\CN' &:=& \{\CA\gn \,|\, \gn \in \Spec\, Z(\CA ), {\rm ht} (\gn ) =1, \gn \neq (s) \}, 
\end{eqnarray*}
 and $Z(\CA )=K[s,t^{\pm 1}]$,  $s=h^n$ and $ t=x^n$,   
\begin{align}
	\begin{tikzpicture}[scale=1.3]
	\node (I) at (-4,2) {\small$\mM'$};
	%\node (D) at (-4, 1) {\small$(x)$};	
	\node (p) at (0,1) { \small$\CN'$};
	%\node (m) at (1,2) {\small$\mX $};
	%\node (m') at (-1,2) {\small$\{ (x,h)\}$};
	\node (0) at (0,0) {\small$0$};
	\node (I') at (4,2) {\small$\mathbb{H}'$};
	\node (D') at (4, 1) {\small$(h)$};	
	%\draw [ shorten <=-2pt, shorten >=-2pt] (0)--(D)--(I);
	\draw [ shorten <=-2pt, shorten >=-2pt] (0)--(D')--(I');
	\draw [ shorten <=-2pt, shorten >=-2pt] (D);
	\draw [ shorten <=-2pt, shorten >=-2pt] (0)--(p);
	\draw [semithick, dotted][ shorten <=-2pt, shorten >=-2pt]  (p)--(I);
	\draw [semithick, dotted][ shorten <=-2pt, shorten >=-2pt]  (p)--(I');
	%\draw [shorten <=-2pt, shorten >=-2pt]  (D)--(m')--(D');
	\end{tikzpicture}  \label{CAqXYZA} %\marginnote{CAqXYZA}
	\end{align}
	\item  Every nonzero  ideal of the algebra $\CA$ meets the centre  $Z(\CA )$.
\item  $\Max (\CA ) =\mM'\, \sqcup \,\mathbb{H}'$.

\end{enumerate}
\end{theorem}

\begin{proof}   1. Recall that $\CB=\CA_h=\CA_s$ where $s\in Z(\CA)$, the element $h$ is a normal element of $\CA$,  and the element $s$ is a central  element of $\CA$. By Proposition \ref{aA12Mar15},
$$ \Spec (\CA)=\Spec(\CA, h)\, \sqcup \,\Spec_{h}(\CA).$$

Since $\CA /(h)\simeq K[x^{\pm 1}]$,   $\Spec(\mA, h)=\{ (h)\}\, \sqcup \, \mathbb{H}'$. By Proposition \ref{aA12Mar15}.(2), Theorem \ref{qCAc28Aug20}.(2),  (\ref{CAqmASum}), and (\ref{CAqCA1Sum4}),
$$ \Spec_{h}(\CA) = \{ \CA\cap \CB \gp\, | \, \gp \in \Spec \, K[s^{\pm 1},t^{\pm 1}]\} = \{ \CA \gr \, | \, \gr\in \Spec \, K[s,t^{\pm 1}],  s\not\in\gr\}$$
since  $\CA\cap \CB \gp=\CA(Z(\CA)\cap \gp)=\CA \gr$ where $\gr =Z(\CA)\cap \gp$, by (\ref{qCAmASum}) and  (\ref{CAqmASum}). Now, the first equality of statement 1 follows. The containment information in the diagram follows from Proposition  
\ref{aA12Mar15}.(3). 

2 and 3. Statements 2 and 3 follow from statement 1.   \end{proof}

{\bf The quotient rings of prime factor algebras of $\CA$ and a classification of simple $\CA$-modules.}  Corollary \ref{CAQqAc28Aug20} describes the quotient rings of prime factor algebras of the algebra $\CA$.

\begin{corollary}\label{CAQqAc28Aug20}%\marginpar{CAQqAc28Aug20}

\begin{enumerate}

\item For each $(h,\gq' )\in \mathbb{H}'$, $\mA/(h,\gq' )=
Q(\mA/(h,\gq' ))=K[x]/\gq'$ is a field. 

\item For each prime ideal $\CA\gr \in \CN' \cup \mM'$ (see Theorem \ref{CAqAc28Aug20}.(1)), the quotient ring  of the algebra $\CA / \CA \gr$, $$Q(\CA /\CA \gr)=\bigoplus_{i,j=0}^{n-1}k(\gr )h^ix^j\simeq k(\gr )\t_{Z(\CA )}\CA,$$ is a central simple $n^2$-dimensional algebra over the field $k(\gr )$ of fractions of the commutative domain $Z(\CA )/\gr$.
The algebra $Q(\CA /\CA \gr)$ is isomorphic to the algebra $\CE (\gr )=(E(s,\gr), \s, a)$ where $E(s,\gr):=k(\gr )[h]/(h^n-s)$ and  $\s (h)=qh$. 

\item The quotient ring  of the algebra $\mA$, $$Q(\CA)=Q(\CB)=\bigoplus_{i,j=0}^{n-1}K(s,t)h^ix^j\simeq K(s,t)\t_{Z(\mA )}\mA,$$ is a central simple $n^2$-dimensional algebra over the field $Z(Q(\CA))=K(s,t)$ of rational functions in two variables.  The algebra $Q(\CA )$  is isomorphic to the cyclic algebra $\CE (0):=(E(s,0), \s, a)$ where $E(s,0):=K(s,t)[h]/(h^n-s)$ and  $\s (h)=qh$.

\item $\mA/(h)=K[x]$ and $Q(\mA/(h))=K(x)$.
\end{enumerate}
\end{corollary}

\begin{proof}   1 and 4.  Statements 1 and 4 follow from Theorem \ref{CAqAc28Aug20}.(1).

2. Statement 2 follows from Theorem \ref{qCAc28Aug20}.(4), since $Q(\CA/ \CA \gr ) = Q(\CB / \CB \gr)$.

3. Statement 3 is a particular case of  statement 2.  \end{proof}

%  Corollary \ref{aQqAc28Aug20}.(1) is a classification of primitive ideals of the algebra $\mA$. Every primitive  ideal of the algebra $\mA$ is a maximal ideal (and vice versa) and co-finite (Corollary \ref{aQqAc28Aug20}.(2)). There is a one-to-one correspondence between the set of primitive ideals of the algebra $\mA$ and the set of isomorphism classes of simple $\mA$-modules (Corollary \ref{aQqAc28Aug20}.(4)). Every simple $\mA$-module is finite dimensional over the field $K$.  %*** Corollary \ref{aQqAc28Aug20}.(3) classifies simple $\mA$-modules and  describes their dimensions and endomorphism rings. ***  

\begin{proposition}\label{A4Aug23}%\marginpar{A4Aug23}
The  algebra $\CA $ satisfies the assumptions of Corollary \ref{b2Aug23}, ${\rm ZL}(\CA )=\Spec (Z(\CA ))$, and Theorem \ref{2Aug23} holds:
 \begin{enumerate}
\item The ring $\CA $ is PLM-ring with ${\rm PL}(\CA )=\Max (Z(\CA ))$.

\item  For every $\gm \in \Max (Z(\CA ))$, $\Prim (\CA ,\gm)=\{ \CA \gm\}$.

\item $\Prim(\CA )=\Max (\CA )=\{ \CA \gm\, | \, \gm\in \Max (Z(\CA ))\} =\mM'\, \sqcup \,\mathbb{H}'$. 

\item For each maximal ideal $\gm \in\mM'$, the factor algebra 
 $$\CA / \CA \gm =Q(\CA /\CA \gm)=\bigoplus_{i,j=0}^{n-1}k(\gm )h^ix^j\simeq k(\gm )\t_{Z(\CA )}\CA,$$ is a central simple $n^2$-dimensional algebra over the field $k(\gm )=Z(\CA )/\gm$. The algebra $\CA /\CA \gm$ is isomorphic to the algebra $\CE (\gm )=(E(s,\gm), \s, a)$ where $E(s,\gm):=k(\gm )[h]/(h^n-s)$ and  $\s (h)=qh$. Let $U(\gm)$ be a unique simple  $\CE (\gm )$-module (described in (\ref{xu=usxe}) and Theorem \ref{12Jun23}.(1,2))).

\item $\widehat{\CA } =\{ U(\gm )\, | \, \gm \in \Max(Z(\CA ))\}$  where $U(\gm)$ is a unique simple module of the simple artinian ring $k(\gm)\t_{Z(\CA )}\CA /\CA \gm$, for $\gm=(h,\gq')\in \mathbb{H}'$, $U(\gm):=K[x]/\gq'$.
All simple $\CA$-modules are finite dimensional.

\item The map $\widehat{\CA } \ra \Prim (\CA )$, $U(\gm )\mapsto \ann_{\CA }(U(\gm))=\CA /\CA \gm$ is a bijection with inverse $\CA \gm\mapsto U(\gm)$.

\item  For $\gm \in \Max (Z(\CA ))$, 
$$
\End_{\CA}(U(\gm)):=\begin{cases}
D(\gm)  & \text{ is described in  Theorem \ref{12Jun23}  if $\gm \in \mM'$},\\
K[x]/\gq' & \text{if }\gm=(h,\gq')\in \mathbb{H}'.
\end{cases}
$$
\end{enumerate}

\end{proposition}

\begin{proof} By Corollary \ref{CAQqAc28Aug20}, the  algebra $\CA $ satisfies the assumptions of Corollary \ref{b2Aug23}, and so  ${\rm ZL}(\CA )=\Spec (Z(\CA ))$, and Theorem \ref{2Aug23} holds, 
 i.e. the corollary holds apart from statement 4, 7 and an explicit description of the simple modules in statement 5.  Statement 4 follows from Corollary  \ref{CAQqAc28Aug20}.(2).
 The explicit description of the simple modules in statement 5 follows from Corollary \ref{CAQqAc28Aug20}.(1). Statement 7 follows from statement 5 and Corollary \ref{CAQqAc28Aug20}.(1,2). 
\end{proof}

{\bf Classification of completely prime ideals of the algebra $\CA$.}   Proposition  \ref{CAB20Jul23} is a classification of completely prime ideals of the algebra $\mA$.

\begin{proposition}\label{CAB20Jul23}%\marginpar{CAB20Jul23}
$\Spec_c (\CA )=\Big\{ \{ 0\},  (h)\} \, \sqcup \,  \mathbb{H}'\,  \sqcup \, \{  \CB\gr\in \CN'\cup \mM' \, | \,   E(s,\gr)=k(\gr)[h]/(h^n-s)$
is a field and $n=\min\{ d'\geq 1\, | \,  d'\mid n$  there is a matrix $X\in M_{d'}(E(s,\gr))$ such that $ X^{\s^{n-1}}\cdots X^\s X=a\}\Big\}$ where $\s (h)=qh$.
\end{proposition}
 
\begin{proof}    A prime ideal $P$ of the algebra $\CA$ is a completely prime ideal if and only if the algebra $Q(\CA /  P)$ is a domain.
 By Corollary \ref{CAQqAc28Aug20}.(1,3,4), $$\Spec_c (\CA )\supseteq \Big\{ \{ 0\}, (h)\} \, \sqcup \, \mathbb{H}'.$$
  
If $P=\CA\gr\in \CN'\cup \mM'$ then, by  Corollary  \ref{CAQqAc28Aug20}.(2),
$$Q(\CA/\CA\gr )\simeq \CE (\gr)=(E(s,\gr),\s , a)\;\; {\rm where}\;\; E(s,\gr)=k(\gr)[h]/(h^n-s)\;\; {\rm   and}\;\;\s (h) = qh.$$
 Now, the proposition  follows from  Corollary \ref{b17Jun23}.   \end{proof}

There are algebra isomorphisms 
\begin{eqnarray*}
\tau_{A_1(q)} &:&  A_1(q)\ra A_1(q^{-1}), \;\; x\mapsto -qy, \;\; y\mapsto x, \\
\tau_\mA  &:& \mA (q)\ra  \mA (q^{-1}), \;\; x\mapsto h, \;\; h\mapsto x, \\
\tau_\CA  &:& \CA (q)\ra  \CA (q^{-1}), \;\; x\mapsto x^{-1}, \;\; h\mapsto h, \\
\tau_\CB  &:&  \CB (q)\ra \CB (q^{-1}), \;\; x\mapsto h, \;\; h\mapsto x.
\end{eqnarray*}

In particular,  $ \tau_{A_1(-1)}\in \Aut_K(A_1(-1))$, $\tau_\mA\in \Aut_K(\mA (-1))$, $\tau_\CA \in \Aut_K(\CA (-1))$, and  $\tau_\CB \in \Aut_K(\CB (-1))$.

An algebra $B$ is called a {\bf self-opposite} algebra if it is isomorphic to its {\bf opposite algebra} $B^{\rm op}$. Recall that the opposite algebra $B^{\rm op}$ coincides with $B$  as a vector space over the field  $K$ but the product in $B^{\rm op}$ is defined as follows: $a\cdot b:=ba$, the latter is a  product in the algebra $B$. Properties of a self-opposite algebra is left-right symmetric. In particular, a classification of simple {\em right} modules is obtained straight away  from the classification of simple {\em left} modules. 
The algebras $\mA$, $A_1$, and $\CB$ are self-opposite:
There are isomorphisms 
\begin{eqnarray*}
 \iota_{A_1(q)} &:&  A_1(q)\ra A_1(q)^{\rm op}, \;\; x\mapsto y, \;\; y\mapsto x, \\
\iota_\mA &:&  \mA \ra \mA^{\rm op}, \;\; x\mapsto h, \;\; h\mapsto x, \\
\iota_\CB &:&  \CB\ra \CB^{\rm op}, \;\; x\mapsto h, \;\; h\mapsto x.
\end{eqnarray*}

%%%%%%%%%%%%%%%%%%%%%% Section 5  %%%%%%%%%%%%%%%%%%

\section{Classifications of  prime ideals  and simple modules of the quantum Weyl algebra $A_1(q)$}\label{QUANTWEYL}%\marginpar{QUANTWEYL}

  The aim of the section is to classify the sets of prime, completely prime,  maximal and primitive ideals,  and simple modules of the quantum Weyl algebra  $A_1=A_1(q)$. The case of the quantum Weyl algebra  $A_1=A_1(q)$ is  more difficult than comparing with the three algebras 
  considered in the previous section.\\

	{\bf The quantum Weyl algebra $A_1$ is a generalized Weyl algebra}. 
	\begin{definition}
	 \cite{Bav-GWA-FA-91, Bav-UkrMathJ-92, Bav 2}. Let $D$ be a ring, $\s$ be an automorphism of $D$ and $a$ is an element of the centre of $D$. {\bf The generalized Weyl algebra} (GWA)
$A:=D(\sigma, a):=D[X,Y; \s , a]$ is a ring  generated by $D$,
$X$ and $Y$ subject to the defining relations:
$$
X\alpha=\sigma(\alpha)X \;\; {\rm and}\;\;  Y\alpha=\sigma^{-1}(\alpha)Y\;\;  {\rm for \; all}\;\;
\alpha \in D, \;\;  \ YX=a \;\;  {\rm and}\;\; XY=\sigma(a).
$$
\end{definition}
The algebra $A=\bigoplus_{n\in {\bf \Z}}\, A_n$
is $\Z$-graded where $A_n=Dv_n$,
$v_n=X^n$ for $n>0$, $v_n=Y^{-n}$ for $n<0$ and $v_0=1.$
It follows from the above relations that
$v_nv_m=(n,m)v_{n+m}=v_{n+m}\langle n,m\rangle $
for some $(n,m)\in D$. If $n>0$ and $m>0$ then
\begin{eqnarray*}
	n\geq m:& &  (n,-m)=\sigma^n(a)\cdots \sigma^{n-m+1}(a),\;\;  (-n,m)=\sigma^{-n+1}(a)
	\cdots \sigma^{-n+m}(a),\\
	n\leq m: & & (n,-m)=\sigma^{n}(a)\cdots \sigma(a),\,\,\,\;\;\; \;\;\; \;\;\; (-n,m)=\sigma^{-n+1}(a)\cdots a,
\end{eqnarray*}
in other cases $(n,m)=1$. Clearly, $\langle n,m\rangle =\s^{-n-m}((n,m))$.

It follows from the relations (where $\s (h)=qh$) 
%\marginpar{xhyhs}
\begin{equation}\label{xhyhs}
xh=\s(h)x,\;\; yh=\s^{-1}(h)y, \;\; yx=h-\frac{1}{q-1}, \;\; {\rm and}\;\; xy=\s\bigg(h-\frac{1}{q-1} \bigg)=qh-\frac{1}{q-1}, 
\end{equation}
that  the quantum Weyl algebra $A_1$ is a GWA $A_1=K[h][x,y;\s , a=h-\frac{1}{q-1}]$ and  
the element $h$ of the algebra $A_1$ is a regular normal element. It determines an automorphism $\o_h: A_1\ra A_1$, $\alpha \mapsto h\alpha h^{-1}$ of the algebra $A_1$, By (\ref{xhyhs}),   $\o_h(x)=q^{-1}x$ and $\o(y)=qy$. 
 The element $t$ is a central element of the algebras $\mA$ and $A_1$. It follows from (\ref{qmAAWy1})   that 
 %\marginpar{qA1Sum5}
\begin{equation}\label{qA1Sum5}
\mA \subset A_1\subset  \CA = \mA_t=A_{1,t}.
\end{equation}
  Notice that 
 %\marginpar{qA1Sum1}
\begin{equation}\label{qA1Sum1}
rt=(-1)^{n-1}\bigg(s-\frac{1}{(q-1)^n}\bigg) \;\; {\rm where}\;\; s=h^n\;\; {\rm and}\;\; h=yx+\frac{1}{q-1}
\end{equation}
since 
\begin{eqnarray*}
rt&=&y^nx^n=y^{n-1}\bigg(h-\frac{1}{q-1}\bigg) x^{n-1}=\bigg(q^{-n+1}h-\frac{1}{q-1}\bigg)
y^{n-1}x^{n-1}
=\prod_{i=0}^{n-1}\bigg(q^{-i}h-\frac{1}{q-1} \bigg)\\
&=&q^{-\frac{n(n-1)}{2}}\prod_{i=0}^{n-1}\bigg(h-q^i\frac{1}{q-1} \bigg)\stackrel{(\ref{xnix2})}{=}(-1)^{n-1}\bigg(h^n-\frac{1}{(q-1)^n}\bigg)=(-1)^{n-1}\bigg(s-\frac{1}{(q-1)^n}\bigg).
\end{eqnarray*}
 Hence,
%\marginpar{qA1Sum8}
\begin{equation}\label{qA1Sum8}
s=(-1)^{n-1}rt+\frac{1}{(q-1)^n}\in Z(A_1)
\end{equation} 
 and  so $r=(-1)^{n-1}\bigg(s-\frac{1}{(q-1)^n}\bigg)t^{-1}\in K[s,t^{-1}]=Z (\CA$) and  
 %\marginpar{qA1Sum2}
\begin{equation}\label{qA1Sum2}
Z(\mA )=K[s,t] \subset K[r,t]=Z(A_1)\subset  K[s,t^{\pm 1}]=K[r,t^{\pm 1}]=Z(\mA )_t= Z(A_1)_t.
\end{equation}
Since $A_1=K[h][x,y;\s , a=h-\frac{1}{q-1}]$ is a GWA and $r\in Z(A_1)$, 
%\marginpar{qA1Sum7}
\begin{equation}\label{qA1Sum7}
A_{1,r}=K[h][y^{\pm 1}; \s^{-1}]=K[h][\Big( y^{-1}\Big)^{\pm 1}; \s ]\simeq \CA\;\; {\rm and}\;\; Z(A_{1,r})=K[s,r^{\pm 1}]=K[t,r^{\pm 1}].
\end{equation}  
  
{\bf Prime ideals of the Weyl algebra $A_1$.} Let us show that  
%\marginpar{qA1Sum}
\begin{equation}\label{qA1Sum}
 A_1 =\bigoplus_{i,j=0}^{n-1}Z( A_1 )x^iy^j\;\; {\rm  where}\;\; Z(A_1)=K[r,t],\;\; r:=y^n, \;\; {\rm and}\;\; t:=x^n.
\end{equation}
 The centre of quantum  Weyl algebra $A_1$, $Z(A_1)=K[r,t]$, is a polynomial algebra in the variables $r$ and $t$. 

\begin{proof}  (i) $r,t\in Z(A_1)$: The statement (i) follows from the following calculations:
\begin{eqnarray*}
ty &=& x^ny=x^{n-1}\s (a)=\s^n(a)x^{n-1}=ax^{n-1}=yxx^{n-1}=yx^n=yt,\\
th &=&x^nh=\s^n(h)x^n =
ht,\\
x r&=& xy^n=\s (a)y^{n-1}=y^{n-1}\s^n(a)=y^{n-1}a=y^{n-1}yx=rx ,\\
rh &=&y^nh=\s^{-n}(h)y^n=
hr.
\end{eqnarray*}
(ii) $A_1 =\bigoplus_{i,j=0}^{n-1}Zx^iy^j$ {\em where} $Z:=K[r,t]\subseteq Z(A_1)$:  By the statement (i), $Z=K[r,t]\subseteq Z(A_1)$. Now, 
 the statement (ii) follows from the fact that the GWA $A_1$ is a $\Z$-graded algebra.

(iii) $Z=Z(A_1)$: The statement (iii) follows from the statement (ii), (\ref{qA1Sum2}), and  (\ref{qA1Sum7}):
$$K[t,r]\subseteq Z(A_1)\subseteq Z(A_{1,t})\cap  Z(A_{1,r})=K[r,t]\cap K[r^{\pm 1},t]=K[t,r]. $$
 \end{proof}

Since $\o_h(x^i)=q^{-i}x^i$  and $\o_h(y^i)=q^iy^i$ for all $i\geq 0$ and the elements $1,q,  \ldots , q^{n-1}$ are distinct elements of the field $K$ (since the element $q$ is a primitive $n$'th root of unite), the algebra 
%\marginpar{qA1Sum3}
\begin{equation}\label{qA1Sum3}
 A_1 =\bigoplus_{j=0}^{n-1}Z( A_1) h^j\oplus\bigoplus_{i=1}^{n-1}\Bigg(\bigoplus_{j=1}^{n-1}Z( A_1) h^jx^i\oplus \bigoplus_{j=1}^{n-1}Z( A_1) h^jy^{n-i}\Bigg)
\end{equation}
is a direct sum of the eigenspaces  of the automorphism  $\o_h$ of the quantum Weyl algebra $ A_1$ and the set of eigenvalues of $\o_h$ is $\Ev_{A_1} (\o_h)=\{1,q,  \ldots , q^{n-1}\}$.

In the algebra $A_1$: For all $i=1,\ldots , n-1$,
%\marginpar{qA1Sum12}
\begin{equation}\label{qA1Sum12}
[x,y^i]=(1-q^{-i})xy^i-\frac{q^{-i+1}-1}{q-1}y^{i-1}\;\; {\rm and}\;\;
[y,x^i]=(1-q^i)yx^i-\frac{q^i-1}{q-1}x^{i-1}. 
\end{equation}
In more detail (recall that $a=h-\frac{1}{q-1}$),
\begin{eqnarray*}
&[x,y^i] =& \Big(\s (a)-\s^{-i+1}(a)\Big)y^{i-1}=(q-q^{-i+1})hy^{i-1}= (q-q^{-i+1})q^{-1}\bigg(xy+\frac{1}{q-1}\bigg)y^{i-1} \\ 
&=&(1-q^{-i})xy^i-\frac{q^{-i+1}-1}{q-1}y^{i-1},\\
&[y, x^i]=& \Big( a-\s^i(a)\Big)x^{i-1}=(1-q^i)hx^{i-1}=(1-q^i)\bigg(yx+\frac{1}{q-1}\bigg)x^{i-1}=(1-q^i)yx^i-\frac{q^i-1}{q-1}x^{i-1}. 
\end{eqnarray*}
\begin{lemma}\label{a24Jul23}%\marginpar{a24Jul23}

\begin{enumerate}
\item In  the factor algebras $A_1(t)$, $A_1/(r)$, and $A_1/(t,r)$ the element $s$ is equal to the   nonzero scalar $\frac{1}{(q-1)^n}$. In particular, the elements $s=h^n$ and $h$ are units. 
\item $(t,s)=(r,s)=(t,h)=(r,h)=A_1$.
\end{enumerate}
\end{lemma}

\begin{proof}   1. Statement 1 follows from the equality $rt=(-1)^{n-1}\bigg(s-\frac{1}{(q-1)^n}\bigg)$, see (\ref{qA1Sum1}). 

2. Statement 2 follows from statement 1.\end{proof}

By (\ref{xhyhs}),
%\marginpar{qA1Sum11}
\begin{equation}\label{qA1Sum11}
 A_1/(h) =K[x^{\pm 1}]\simeq K[y^{\pm 1}].
\end{equation}

Lemma \ref{b28Jul23} describes  simple modules, prime and primitive ideals of the algebra $A_1/(h)$.

\begin{lemma}\label{b28Jul23}%\marginpar{b28Jul23}

\begin{enumerate}
\item $\Spec (A_1/(h))=\{ \{ \{ 0\}, (f)\, | \, f \in \Irr_m(K[x])\backslash \{ x\} \}$.
\item $\Prim (A_1/(h))=\Max (A_1/(h))=\{  (f)\, | \, f \in \Irr_m(K[x])\backslash \{ x\} \}$.
\item $\widehat{A_1/(h)}=\{ A_1/(h,f)\, | \, f \in \Irr_m(K[x])\backslash \{ x\} \}$.
\end{enumerate}
\end{lemma}

\begin{proof}   Lemma  follows at once from (\ref{qA1Sum11}).
 \end{proof}

Recall that $A_{1,trs}=A_{1,trh}$ and  $h$ is a normal element of $A_1$.  
Applying Proposition \ref{aA12Mar15} several times  and using the bijections/identifications there we see that the spectrum of the algebra $A_1$ is a {\em disjoint} union of the spectra of the following 8 algebras:
\begin{eqnarray*}
 A_{1,trh}, & \Big( A_1/(t) \Big)_{rh}=\Big( A_1/(t) \Big)_{r},\;\;\Big( A_1/(r) \Big)_{th}=\Big( A_1/(r) \Big)_{t}, \;\;
\Big( A_1/(t,r) \Big)_{h}=A_1/(t,r), \\
& \Big( A_1/(h) \Big)_{rt}\stackrel{(\ref{qA1Sum11})}{=} A_1/(h),\;\;\Big( A_1/(r,h) \Big)_{t}=0,\;\; \Big( A_1/(t,h) \Big)_{r}=0\;\;A_1/(t,r,h)=0. 
\end{eqnarray*}
By Lemma \ref{a24Jul23}.(1), the subscript `$h$' is dropped in three algebras in the first row and 
the last three algebras are equal to zero by Lemma \ref{a24Jul23}.(2). Proposition \ref{aA12Mar15} provides also a containment information between primes that will be used later.  In view of (\ref{qA1Sum9}), (\ref{qA1Sum10}), and (\ref{qA1Sum11}), Proposition \ref{A25Jul23} it remains to  clarifies the structure of the three  remaining algebras: $\Big( A_1/(t) \Big)_{r}$, 
$\Big( A_1/(r) \Big)_{t}$, and  $A_1/(t,r)$. 
Let $\L$ be any of the three algebras.  By (\ref{qA1Sum8}), in the algebra $\L$, 
$$ 0=s-\frac{1}{(q-1)^n}=h^n-\frac{1}{(q-1)^n}=\prod_{i=0}^n\bigg(h-\frac{q^i}{q-1}\bigg).$$
By (\ref{qA1Sum}), the algebra $\L$ contains the subalgebra
$$\CK :=K[h]/ \bigg(h^n-\frac{1}{(q-1)^n} \bigg)\simeq \prod_{i=0}^{n-1}K[h]/ \bigg(h-\frac{q^i}{(q-1)^n} \bigg)\simeq K^n.$$
Let $1=\varepsilon_0+\cdots +\varepsilon_{n-1}$ be the corresponding sum of nonzero orthogonal idempotents. Notice that
%\marginpar{qA1Sum15}
\begin{equation}\label{qA1Sum15}
\varepsilon_i =\frac{\prod_{j\neq i} \Big(h-\frac{q^j}{(q-1)^n}\Big)}{\prod_{j\neq i} \Big(\frac{q^i}{(q-1)^n}-\frac{q^j}{(q-1)^n}\Big)}\;\; {\rm for}\;\; i=0, \ldots , n-1.
\end{equation}
Then $\CK=\prod_{i=0}^{n-1}\CK e_i$ and $\CK e_i\simeq K[h]/ \bigg(h^n-\frac{q^i}{(q-1)^n} \bigg)\simeq K$ for all $i=0, \ldots , n-1$.  
The automorphism $\s$ of the algebra $K[h]$ preserves the ideal $(s-\frac{1}{(q-1)^n})$. Hence, $\s$ is also an automorphism of the factor algebra $\CK\simeq K^n$ which  cyclicly permutes its  direct components since  
 %\marginpar{qA1Sum17}
\begin{equation}\label{qA1Sum17}
\s(\varepsilon_i)=\varepsilon_{i-1}   \;\;{\rm for}\;\; i=0, \ldots , n-1\;\; (\varepsilon_{-1}:=\varepsilon_{n-1}).
\end{equation}
We consider the subscript `$i$' in $\varepsilon_i$ as an element of the abelian group $\Z/n\Z$. \\

{\bf The algebra $A_1/(t,r)\simeq M_n(K)$ and its unique simple module $L$.} By (\ref{qA1Sum}),
the algebra
%\marginpar{qA1Sum13}
\begin{equation}\label{qA1Sum13}
A_1/(t,r) =\bigoplus_{i,j=0}^{n-1}Kx^iy^j
\end{equation}
is an $n^2$-dimensional algebra. By (\ref{qA1Sum13}), the $n$-dimensional $A_1/(t,r)$-module 
%\marginpar{qA1Sum14}
\begin{equation}\label{qA1Sum14}
L:=A_1/A_1(t,y) =\bigoplus_{i=0}^{n-1}Kx^i\b1, \;\; {\rm where}\;\; \b1 :=1+ A_1(t,y), 
\end{equation}
is a direct sum  of $1$-dimensional eigenspaces with distinct eigenvalues for the action of the element $h$:
$$ hx^i\b1=q^{-i}x^ih\b1=q^{-i}x^iq^{-1}\bigg(xy+\frac{1}{q-1}\bigg)\b1=\frac{q^{-i-1}}{q-1}x^i\b1\;\; {\rm for}\;\; i=0, \ldots , n-1.$$
Therefore, any nonzero submodule, say $M$, of $L$ contains an element $x^i\b1$ for some $i$. Then it also contains the element
$$y^ix^i\b1 =
\prod_{j=0}^{i-1}\s^{-j}(a)\b1=
\prod_{j=0}^{i-1}\bigg(q^{-j}h-\frac{1}{q-1} \bigg)\b1=\prod_{j=0}^{i-1}\frac{q^{-j-1}-1}{q-1} \b1\neq 0
$$
which is a generator for the module $L$, and so $M=L$. This proves that the module $L$ is a {\em simple} module. Clearly,
$$ x\cdot x^i\b1 =
\begin{cases}
 x^{i+1}\b1 & \text{if }i=0, \ldots , n-2,\\
0& \text{if }i=n-1,\\
\end{cases}
\;\; {\and}\;\;
y\cdot x^i\b1 =
\begin{cases}
0& \text{if }i=0,\\
\frac{q^{-i}-1}{q-1}x^{i-1}\b1& \text{if }i=1, \ldots , n-1.
\end{cases}
$$
In the $K$-basis $\b1, x\b1, \ldots , x^{n-1}\b1$ of the module $L$,  the matrices of the linear maps $h\cdot$, $x\cdot$, and $y\cdot$ are the diagonal matrix ${\rm diag}\Big(\frac{q^{-1}}{q-1}, \ldots ,  \frac{q^{-n+1}}{q-1}\Big)$, the lower triangular matrix with $1$ on the lower diagonal, and the upper triangular matrix with elements $\frac{q^{-1}-1}{q-1}, \ldots ,  \frac{q^{-n+1}-1}{q-1}$ on the upper diagonal, respectively. 

The following useful lemma is used in the proof of Proposition \ref{A26Jul23}. 

\begin{lemma}\label{a26Jul23}%\marginpar{a26Jul23}
Suppose that $A$ is an $n^2$-dimensional algebra over an arbitrary field $F$. Then $A\simeq M_n(F)$ iff there is a simple $n$-dimensional $A$-module. 
\end{lemma}

\begin{proof}   $(\Rightarrow)$ The $n$-dimensional vector space $F^n$ is a simple $ M_n(F)$-module.

$(\Leftarrow )$ Let $U$ be an $n$-dimensional simple $A$-module and $\gr$ be the radical of the algebra $A$. By the Artin-Wedernburn Theorem, the factor algebra $A/\gr\simeq \prod_{i=1}^lM_{m_i}(D_i)$ is a direct product of matrix algebras over division rings. Since $\dim_F(A)=n^2=\Big(\dim_F(U) \Big)^2$, we must have $\gr=0$, $l=1$, $n^2=\dim_F(A)=m_1^2\dim_F(D_1)$, and $n=\dim_F(U)=m_1\dim_F(D_1)$. It follows from 
$$m_1^2\dim_F(D_1)=n^2=m_1^2\Big(\dim_F(D_1)\Big)^2$$ that $\dim_F(D_1)=1$ and $m_1=n$, i.e. $A\simeq M_n(F)$.   \end{proof}

\begin{proposition}\label{A26Jul23}%\marginpar{A26Jul23}
The algebra $A_1/(t,r)$ is isomorphic to the matrix algebra $M_n(K)$, $\Spec(A_1/(t,r))=\{0\}$,  and the $A_1/(t,r)$-module $L$ is a unique (up to isomorphism) simple $A_1/(t,r)$-module. 
\end{proposition}

\begin{proof}   Since $\dim_K(A_1/(t,r))=n^2$ and $\dim_K(L)=n$, $A_1/(t,r)\simeq M_n(K)$, by Lemma \ref{a26Jul23}. Hence, $\Spec(A_1/(t,r))=\{0\}$.
 \end{proof}

 {\bf The generalized Weyl algebra $A_1/(tr) $.} 
The algebra $A_1=K[h][x,y; \s ,a=h-\frac{1}{q-1}$ is a GWA. Then,   by  (\ref{qA1Sum1}),
the factor algebra
%\marginpar{qA1Sum20}
\begin{equation}\label{qA1Sum20}
A_1/(tr)=A_1/\Big( s-\frac{1}{(q-1)^n}\Big)=\CK [x,y;\s, a=h-\frac{1}{q-1}]=\bigoplus_{i\in \Z}\CK v_i
\end{equation}
 is a GWA where $v_i=x^i$  and $v_{-i}=y^i$ for $i\in \N$. It follows from the equalities $y=yxx^{-1}=ax$ and $x=y^{-1}yx=y^{-1}a$ that the algebras 
 %\marginpar{qA1Sum18}
\begin{equation}\label{qA1Sum18}
\Big( A_1/(r) \Big)_{t}=\Big( A_1/(rt) \Big)_{t}
=\bigoplus_{i\in \Z}\CK x^i=\bigoplus_{i=0}^{n-1}\CK [t^{\pm 1}]x^i=\CK [x^{\pm 1};\s],
\end{equation}
%\marginpar{qA1Sum19}
\begin{equation}\label{qA1Sum19}
\Big( A_1/(t) \Big)_{r}=\Big( A_1/(rt) \Big)_{r}
=\bigoplus_{i\in \Z}\CK y^i=\bigoplus_{i=0}^{n-1}\CK [r^{\pm 1}]y^i=\CK [y^{\pm 1};\s^{-1}],
\end{equation}
are skew Laurent polynomial algebras. Furthermore, they are isomorphic: $$\Big( A_1/(t) \Big)_{r}
=\CK [y^{\pm 1};\s^{-1}]=\CK [\Big( y^{-1}\Big)^{\pm 1};\s]\simeq \Big( A_1/(r) \Big)_{t}.$$

{\bf The algebras $\Big( A_1/(r) \Big)_{t}$ and 
$\Big( A_1/(t) \Big)_{r}$.} By (\ref{qA1Sum}),

%\marginpar{qA1Sum16}
\begin{equation}\label{qA1Sum16}
\Big( A_1/(t) \Big)_{r}=\bigoplus_{i,j=0}^{n-1}K[t^{\pm 1}]x^iy^j\;\; {\rm and}\;\; \Big( A_1/(r) \Big)_{t}=\bigoplus_{i,j=0}^{n-1}K[r^{\pm 1}]x^iy^j.
\end{equation}

For a polynomial algebra $K[z]$,  we denote by $\Irr_m(K[z])$ the set of its monic irreducible polynomials. Clearly, $\Spec (K[z])=\{ \{ 0\}, (f)\, | \, f\in \Irr_m(K[t]) \}$. Proposition \ref{A25Jul23} shows that the algebras $\Big( A_1/(r) \Big)_{t}$ and 
$\Big( A_1/(t) \Big)_{r}$ and matrix algebras over  Laurent polynomial algebra and describes their prime spectra.

\begin{proposition}\label{A25Jul23}%\marginpar{A25Jul23}

\begin{enumerate}
\item $\Big( A_1/(r) \Big)_{t}=\bigoplus_{i,j=0}^{n-1}K[t^{\pm 1}]E_{ij}\simeq M_n(K[t^{\pm 1}])$ where $E_{ij}:=\varepsilon_ix^{j-i}\varepsilon_j$ are the matrix units.

\item $\Spec \Big(\Big( A_1/(r) \Big)_{t}\Big)=\{ \{ 0\}, (f)\, | \, f\in \Irr_m(K[t])\backslash \{ t\} \}$ and the factor algebra $$A_1/(r,f)=\Big( A_1/(r) \Big)_{t}/(f)\simeq M_n(K[t]/(f))$$ is  the matrix algebra over the field $K[t]/(f)$ for all $ f\in \Irr_m(K[t])\backslash \{ t\} $.

\item $\Big( A_1/(t) \Big)_{r}=\bigoplus_{i,j=0}^{n-1}K[r^{\pm 1}]E_{ij}\simeq M_n(K[r^{\pm 1}])$
where $E_{ij}:=\varepsilon_iy^{i-j}\varepsilon_j$ are the matrix units.

\item $\Spec \Big(\Big( A_1/(t) \Big)_{r}\Big)=\{ \{ 0\}, (g)\, | \, g\in \Irr_m(K[r])\backslash \{ r\} \}$ and the factor algebra $$ A_1/(t,g)=\Big( A_1/(t) \Big)_{r}/(g)\simeq M_n(K[r]/(g))$$ is  the matrix algebra over the field $K[r]/(g)$ for all $ g\in \Irr_m(K[r])\backslash \{ r\} $.

\end{enumerate}
\end{proposition}

\begin{proof}   1. The elements $E_{ij}$ are matrix units: For all $i,j,k,l=0,\ldots , n-1$, $$E_{ij}E_{kl}=\varepsilon_ix^{j-i}\varepsilon_j\varepsilon_kx^{l-k}\varepsilon_l=\d_{jk}\varepsilon_ix^{j-i}\varepsilon_jx^{l-k}\varepsilon_l=\d_{jk}\varepsilon_ix^{j-i}x^{l-j}\varepsilon_l=\varepsilon_ix^{l-i}\varepsilon_k=\d_{jk}E_{il}.$$
By (\ref{qA1Sum17}), (\ref{qA1Sum18}), and the equality $1=\varepsilon_0+\cdots +\varepsilon_{n-1}$, 
$$ \Big( A_1/(r) \Big)_{t}=\bigoplus_{i=0}^{n-1}\CK [t^{\pm 1}]x^i=\bigoplus_{j=0}^{n-1}
\bigg(\sum_{i=0}^{n-1}\CK [t^{\pm 1}]x^i\varepsilon_j\bigg)
=\bigoplus_{j=0}^{n-1}
\bigg(\sum_{i=0}^{n-1}K [t^{\pm 1}]\varepsilon_{j-i}x^i\varepsilon_j\bigg)
=\bigoplus_{i,j=0}^{n-1}K[t^{\pm 1}]E_{ij},
$$
and statement 1 follows.

2. The first equality of statement 2 and the isomorphism $\Big( A_1/(r) \Big)_{t}/(f)\simeq M_n(K[t]/(f))$ follow from statement 1.
 Since $ f\in \Irr_m(K[t])\backslash \{ t\} $, the elements $x$ and $t=x^n$ are invertible in the field $K[x]/(f)$. The factor algebra $A_1/(r,f)$ contains the field 
$K[x]/(f)$, and so  $A_1/(r,f)=\Big(A_1/(r,f) \Big)_t=\Big(A_1/(r) \Big)_t/(f)$.

3 and 4. Repeat the arguments of the proofs of statements 1 and 2 making obvious adjustments.
 \end{proof}

For the algebras $\Big( A_1/(r) \Big)_{t}$ and 
$\Big( A_1/(t) \Big)_{r}$, Corollary \ref{a28Jul23} describes their simple modules and primitive ideals.

\begin{corollary}\label{a28Jul23}%\marginpar{a28Jul23}

\begin{enumerate}

\item $\Prim \Big(\Big( A_1/(r) \Big)_{t}\Big)=\Max  \Big(\Big( A_1/(r) \Big)_{t}\Big)=\{  (f)\, | \, f\in \Irr_m(K[t])\backslash \{ t\} \}$.

\item $\widehat{\Big( A_1/(r) \Big)_{t}}= \{ \Big( K[t]/(f)\Big)^n \; | \; f\in \Irr_m(K[t])\backslash \{ t\} \}$ where the  $\Big( K[t]/(f)\Big)^n \simeq \Big(M_n( K[t]/(f)\Big)e_0$ is a unique simple $M_n( K[t]/(f))$-module and its annihilator   as an $\Big( A_1/(r) \Big)_{t}$-module is $(f)$.

\item $\Prim \Big(\Big( A_1/(t) \Big)_{r}\Big)=\Max  \Big(\Big( A_1/(t) \Big)_{r}\Big)=\{  (g)\, | \, g\in \Irr_m(K[r])\backslash \{ r\} \}$.

\item $\widehat{\Big( A_1/(t) \Big)_{r}}= \{ \Big( K[r]/(g)\Big)^n \; | \; g\in \Irr_m(K[r])\backslash \{ r\} \}$ where the  $\Big( K[r]/(g)\Big)^n \simeq \Big(M_n( K[r]/(g)\Big)e_0$ is a unique simple $M_n( K[r]/(g))$-module and its annihilator   as an $\Big( A_1/(t) \Big)_{r}$-module is $(g)$.

\end{enumerate}
\end{corollary}

\begin{proof}   The corollary follows from Proposition \ref{A25Jul23}.   \end{proof}

{\bf The ideals $(r)$ and $(t)$ are prime ideals of the algebra $A_1$.} Proposition \ref{A29Jul23}
gives a basis for the algebras $A_1/(r)$ and $A_1/(t)$. The ideals $(r)$ and $(t)$ of $A_1$ are homogeneous ideals with respect to the $\Z$-grading of the GWA $A_1$. Furthermore, Proposition \ref{A29Jul23}
describes the $\Z$-grading  for the algebras $A_1/(r)$ and $A_1/(t)$.

\begin{proposition}\label{A29Jul23}%\marginpar{A29Jul23}

\begin{enumerate}
\item $A_1/(r)=\bigoplus_{i=1}^{n-1}\Big(\bigoplus_{j=i}^{n-1} \CK e_j\Big)y^i\tilde{1}\oplus \CK [x;\s]\tilde{1}$ and  $t+(r)\in \CC_{A_1/(r)}$ where $\tilde{1}:=1+(r)$, and $Z(A_1/(r))=K[t]$.
\item $A_1/(t)=\bigoplus_{i=1}^{n-1}\Big(\bigoplus_{j=0}^{i} \CK e_j\Big)x^i\overline{1}\oplus \CK [y;\s^{-1}]\overline{1}$ and  $r+(t)\in \CC_{A_1/(t)}$ where $\overline{1} :=1+(t)$, and $Z(A_1/(t))=K[r]$.

\end{enumerate}
\end{proposition}

\begin{proof}   1.
\begin{eqnarray*}
\Big(A_1/(rt)\Big)r&\stackrel{(\ref{qA1Sum20})}{=}& \bigoplus_{i\geq n}\CK y^i\oplus \bigoplus_{i=1}^{n-1}\CK x^{n-i}y^n=\bigoplus_{i\geq n}\CK y^i\oplus \bigoplus_{i=1}^{n-1}\CK \prod_{\nu =1}^{n-i}\s^\nu (a)y^i\\
&=&\bigoplus_{i\geq n}\CK y^i\oplus \bigoplus_{i=1}^{n-1}\CK \prod_{\nu =1}^{n-i}(h-q^{-\nu}(q-1)^{-1})y^i=\bigoplus_{i\geq n}\CK y^i\oplus \bigoplus_{i=1}^{n-1}\CK \prod_{\mu =i}^{n-1}(h-q^{\mu}(q-1)^{-1})y^i\\
&=&\bigoplus_{i\geq n}\CK y^i\oplus \bigoplus_{i=1}^{n-1}\bigg( \bigoplus_{\g =0}^{i-1}\CK e_\nu\bigg) y^i,\\
 A_1/(r)&=&  A_1/(rt,r) = \Big(A_1/(rt)\Big)/\Big(A_1/(rt)\Big)r\stackrel{(\ref{qA1Sum20})}{=}
 \bigoplus_{i=1}^{n-1}\bigg(\bigoplus_{j=i}^{n-1} \CK e_j\bigg)y^i\tilde{1}\oplus \CK [x;\s ]\tilde{1}.
\end{eqnarray*}
Now, for each $i=1,\ldots, n-1$ the map
$\cdot t: \bigg(\bigoplus_{j=i}^{n-1} \CK e_j\bigg)y^i\tilde{1}\ra \CK x^{n-i}\tilde{1}$ is an {\em injection} since $y^it=y^ix^n=\prod_{\g =0}^{i-1}\s^{-\g}(a)\cdot x^{n-i}= \prod_{\g =0}^{i-1}\Big(q^{-\g}h-(q-1)^{-1}\Big)\cdot x^{n-i}= q^{-\frac{i(i-1)}{2}}\prod_{\g =0}^{i-1}\Big(h-q^{\g}(q-1)^{-1}\Big)\cdot x^{n-i}$. Therefore, $t+(r)\in \CC_{A_1/(r)}$. Hence, $A_1/(r)\subseteq \Big( A_1/(r)\Big)_t$, and so 
$$ Z(A_1/(r))= A_1/(r)\cap  Z\Big(\Big( A_1/(r)\Big)_t\Big)= A_1/(r)\cap K[t^{\pm 1}]=K[t]$$
as $Z\Big(\Big( A_1/(r)\Big)_t\Big)=  K[t^{\pm 1}]$, by Proposition \ref{A25Jul23}.(1).

2. 
\begin{eqnarray*}
\Big(A_1/(tr)\Big)t&\stackrel{(\ref{qA1Sum20})}{=}& \bigoplus_{i\geq n}\CK x^i\oplus \bigoplus_{i=1}^{n-1}\CK y^{n-i}x^n=\bigoplus_{i\geq n}\CK x^i\oplus \bigoplus_{i=1}^{n-1}\CK \prod_{\nu =0}^{n-i-1}\s^{-\nu} (a)y^i\\
&=&\bigoplus_{i\geq n}\CK x^i\oplus \bigoplus_{i=1}^{n-1}\CK \prod_{\nu =0}^{n-i-1}(h-q^{\nu}(q-1)^{-1})x^i=\bigoplus_{i\geq n}\CK x^i\oplus \bigoplus_{i=1}^{n-1}\bigg( \bigoplus_{\mu =n-i}^{n-1}\CK e_\nu\bigg) x^i,\\
 A_1/(t)&=&  A_1/(rt,t) = \Big(A_1/(rt)\Big)/\Big(A_1/(rt)\Big)t\stackrel{(\ref{qA1Sum20})}{=}
 \bigoplus_{i=1}^{n-1}\bigg(\bigoplus_{j=0}^{n-i-1} \CK e_j\bigg)x^i\overline{1}\oplus \CK [y;\s^{-1}]\overline{1}.
\end{eqnarray*}
Now, for each $i=1,\ldots, n-1$ the map
$\cdot r: \bigg(\bigoplus_{j=0}^{n-i-1} \CK e_j\bigg)x^i\overline{1}\ra \CK y^{n-i}\tilde{1}$ is an {\em injection} since $x^ir=x^iy^n=\prod_{\g =1}^i\s^{\g}(a)\cdot y^{n-i}= \prod_{\g =1}^i\Big(q^{\g}h-(q-1)^{-1}\Big)\cdot y^{n-i}= q^{\frac{i(i+1)}{2}}\prod_{\g =1}^i\Big(h-q^{-\g}(q-1)^{-1}\Big)\cdot y^{n-i}= q^{\frac{i(i+1)}{2}}\prod_{j=n-i}^{n-1}\Big(h-q^j(q-1)^{-1}\Big)\cdot y^{n-i}$. Therefore, $t+(r)\in \CC_{A_1/(r)}$. Hence, $A_1/(t)\subseteq \Big( A_1/(t)\Big)_r$, and so 
$$ Z(A_1/(t))= A_1/(t)\cap  Z\Big(\Big( A_1/(t)\Big)_r\Big)= A_1/(t)\cap K[r^{\pm 1}]=K[r]$$
as $Z\Big(\Big( A_1/(t)\Big)_r\Big)=  K[r^{\pm 1}]$, by Proposition \ref{A25Jul23}.(3).
  \end{proof}
  
For a left denominator set $S$ of a ring $R$ and an $R$-module $M$, the subset of $M$,
$$ \tor_S(M):=\{ m\in M\, | \, sm=0\;\;{\rm for\; some}\;\; s\in S\},$$
is a submodule which is called the {\bf $S$-torsion submodule} of $M$. If $S$ is also a right denominator set of $R$ then $ \tor_S(R)$ is an ideal of $R$. 

\begin{corollary}\label{a29Jul23}%\marginpar{a29Jul23}
$(r),(t)\in \Spec (A_1)$.

%The ideals $(r)$ and $(t)$ are prime ideals of the algebra $A_1$. 
\end{corollary}

\begin{proof}   Let $\tau : A_1/(r)\ra \Big(A_1/(r)\Big)_t$, $\alpha\mapsto \frac{\alpha}{1}.$
By Proposition \ref{aA12Mar15}, $\Spec (A_1/(r))=\Spec (A_1/(r), t)\, \sqcup\, \Spec_t (A_1/(r))$.
  By Proposition \ref{A25Jul23}.(2),
$\Spec \Big(\Big( A_1/(r) \Big)_{t}\Big)=\{ \{ 0\}, (f)\, | \, f\in \Irr_m(K[t])\backslash \{ t\} \}$. By Proposition \ref{aA12Mar15}.(2), $\tau^{-1}(0)\in \Spec_t (A_1/(r))$. Since $t+(r)\in \CC_{A_1/(r)}$ (Proposition \ref{A29Jul23}.(1)),  $$\tau^{-1}(0)=\tor_T(A_1/(r))=\{ 0\}$$ where $T=\{t^i\, | \, i\in \N\}$. Therefore, $\{0\}\in \Spec (A_1/(r))$, and so $(r)\in \Spec (A_1)$.

Let $\rho : A_1/(t)\ra \Big(A_1/(t)\Big)_r$, $\beta\mapsto \frac{\beta}{1}.$
By Proposition \ref{aA12Mar15}, $\Spec (A_1/(t))=\Spec (A_1/(t), r)\, \sqcup\, \Spec_r (A_1/(t))$.
  By Proposition \ref{A25Jul23}.(4),
$\Spec \Big(\Big( A_1/(t) \Big)_{r}\Big)=\{ \{ 0\}, (g)\, | \, g\in \Irr_m(K[r])\backslash \{ r\} \}$. By Proposition \ref{aA12Mar15}.(2), $\rho^{-1}(0)\in \Spec_r (A_1/(t))$. Since $r+(t)\in \CC_{A_1/(t)}$ (Proposition \ref{A29Jul23}.(2)),  $$\rho^{-1}(0)=\tor_R(A_1/(t))=\{ 0\}$$ where $R=\{r^i\, | \, i\in \N\}$. Therefore, $\{0\}\in \Spec (A_1/(t))$, and so $(t)\in \Spec (A_1)$.   \end{proof}  

For an $A_1$-module $M$, let $\soc_{A_1}(M)$ be its  socle. 

\begin{corollary}\label{a31Jul23}%\marginpar{a31Jul23}
\begin{enumerate}
\item $\soc_{A_1}(A_1/(r))=0$.
\item $\soc_{A_1}(A_1/(t))=0$.
\end{enumerate}
\end{corollary}

\begin{proof}   1.  Suppose that $\soc_{A_1}(A_1/(r))\neq 0$. Then there is a simple submodule, say $U$, of $A_1/(r)$. By Proposition \ref{A29Jul23}.(1), $t+(r)\in \CC_{A_1/(r)}$ and $t\in Z(A_1)$, and so $tU=U$. By Proposition \ref{A29Jul23}.(1),
$A_1/(r)=\bigoplus_{i=1}^{n-1}\Big(\bigoplus_{j=i}^{n-1} \CK e_j\Big)y^i\tilde{1}\oplus \CK [x;\s]\tilde{1}$. It follows that the nonzero  $A_1$-module 
$tU$ is  strictly contained in the simple $A_1$-module $U$, a contradiction (since $\bigcap_{i\geq 1}t^iU=0$, use the $\Z$-grading of the $A_1-$module $A_1/(r)$). Therefore, $\soc_{A_1}(A_1/(r))=0$.

2. Suppose that $\soc_{A_1}(A_1/(t))\neq 0$. Then there is a simple submodule, say $V$, of $A_1/(t)$. By Proposition \ref{A29Jul23}.(2), $r+(t)\in \CC_{A_1/(t)}$ and $r\in Z(A_1)$, and so $rV=V$. By Proposition \ref{A29Jul23}.(2),
$A_1/(t)=\bigoplus_{i=1}^{n-1}\Big(\bigoplus_{j=0}^{i} \CK e_j\Big)x^i\overline{1}\oplus \CK [y;\s^{-1}]\overline{1}$. It follows that the nonzero  $A_1$-module 
$rV$ is  strictly contained in the simple $A_1$-module $V$, a contradiction (since $\bigcap_{i\geq 1}r^iV=0$, use the $\Z$-grading of the $A_1-$module $A_1/(t)$). Therefore, $\soc_{A_1}(A_1/(t))=0$.    \end{proof}

{\bf The algebra $A_{1,trs}$.} Let $\CZ :=Z(\CB)_{s-(q-1)^{-n}}=K[t^{\pm 1}, s^{\pm 1}]_{s-(q-1)^{-n}}=K[t^{\pm 1}, r^{\pm 1}]_s$ where the third equality follows from the equality $rt=(-1)^{n-1}(s-(q-1)^{-n})$, see (\ref{qA1Sum1}).
Recall that $s=(-1)^{n-1}rt+\frac{1}{(q-1)^n}\in Z(A_1)=K[r,t]$. By (\ref{qA1Sum}) and the inclusion $trs\in Z(A_1)$, 
%\marginpar{qA1Sum9}
\begin{equation}\label{qA1Sum9}
 A_{1,trs} =\bigoplus_{i,j=0}^{n-1}Z( A_1 )_{trs}x^iy^j\;\; {\rm  and }\;\; Z(A_{1,trs})=Z(A_1)_{trs}=K[r^{\pm 1},t^{\pm 1}]_s=K[s^{\pm 1},t^{\pm 1}]_{s-\frac{1}{(q-1)^n}}=\CZ .
\end{equation}
Let $\CD := K[h^{\pm 1}]_{s-\frac{1}{(q-1)^n}}$.  Furthermore, 
%\marginpar{qA1Sum10}
\begin{equation}\label{qA1Sum10}
A_{1,trs}=\CB_{s-(q-1)^{-n}}=\CD [x^{\pm 1}, \s]=\CD [y^{\pm 1}, \s^{-1}].
\end{equation}
In more detail, 
\begin{eqnarray*}
 A_{1,trs} &=& \Big(A_{1,t}\Big)_{rs}=\CA_{rs}=\Big(\CA_s\Big)_r=\CB_r=\CB_{rt^{-1}}=\CB_{s-(q-1)^{-n}}=K[h^{\pm 1}][x^{\pm 1}; \s]_{s-\frac{1}{(q-1)^n}}\\
 &=&K[h^{\pm 1}]_{s-\frac{1}{(q-1)^n}}[x^{\pm 1}; \s ]= \CD [x^{\pm 1}; \s]=\CD [x^{\pm 1}; \s]=\CD [y^{\pm 1}, \s^{-1}].
 % Z(A_{1,trs}) &=& Z(\CB_{s-(q-1)^{-n}})=Z(\CB)_{s-(q-1)^{-n}}=\CD\;\;  ({\rm by}\;\; (\ref{qCAmASum})\;\; {\rm and}\;\; s-(q-1)^{-n}\in Z(\CB)).\\
\end{eqnarray*}

Proposition \ref{B28Jul23}.(1) and Proposition \ref{B28Jul23}.(2) are  explicit descriptions of  ideals and  prime ideals of the algebra $\CA_{1,trs}$. Proposition \ref{B28Jul23}.(4) shows that the quotient rings of prime factor algebras of  the algebra $\CA_{1,trs}$ are central simple $n^2$-dimensional algebras  of type $(E,\s ,a)$ from Section \ref{ALGCE}.

\begin{proposition}\label{B28Jul23}%\marginpar{B28Jul23}

\begin{enumerate}
\item The map $\CI (Z(A_{1,trs}))\ra \CI (A_{1,trs})$, $\ga \mapsto A_{1,trs}\ga$ is a bijection with  inverse $ I\mapsto Z(A_{1,trs})\cap I$. In particular, $\CI (A_{1,trs})=\{ A_{1,trs}\ga \, | \, \ga \in \CI (Z(A_{1,trs}))\}$ and every nonzero ideal of $A_{1,trs}$ meets the centre of $A_{1,trs}$.
\item   $\Spec (A_{1,trs} )=\{A_{1,trs} \gr' \, | \,   \gr' \in \Spec\, Z(A_{1,trs} )\}$. 
\item   $\Max (A_{1,trs} ) = \{ A_{1,trs} \gm\,|\, \gm \in \Max (Z(A_{1,trs} )) \}$.

\item For each  prime ideal $\gr' \in \Spec\, Z(A_{1,trs} )$, the quotient ring  of the algebra $A_{1,trs} / A_{1,trs} \gr'$, $$Q(A_{1,trs} /A_{1,trs} \gr')=\bigoplus_{i,j=0}^{n-1}k(\gr' )h^ix^j\simeq k(\gr' )\t_{Z(A_{1,trs} )}A_{1,trs},$$ is a central simple $n^2$-dimensional algebra over the field $k(\gr' )$ of fractions of the commutative domain $Z(A_{1,trs} )/\gr'$. The algebra $Q(A_{1,trs} /A_{1,trs} \gr')$ is isomorphic to the algebra $\CE (\gr'):=(E(s,\gr'), \s, a)$ where $E(s,\gr'):=k(\gr' )[h]/(h^n-s)\simeq k(\gr' )[h,h^{-1}]/(h^n-s)$ and  $\s (h)=qh$.

\end{enumerate}
\end{proposition}

\begin{proof}   The proposition follows at once from 
the fact that $A_{1,trs}\simeq \CB_{s-\frac{1}{(q-1)^n}}$ and Theorem \ref{qCAc28Aug20}.(1--4).    \end{proof}

  Corollary \ref{c28Jul23}.(1) is a classification of primitive ideals of the algebra $A_{1,trs}$. Every primitive  ideal of the algebra $A_{1,trs}$ is a maximal ideal (and vice versa) and co-finite (Corollary \ref{c28Jul23}.(2)). There is a one-to-one correspondence between the set of primitive ideals of the algebra $A_{1,trs}$ and the set of isomorphism classes of simple $A_{1,trs}$-modules (Corollary \ref{c28Jul23}.(4)). Every simple $A_{1,trs}$-module is finite dimensional over the field $K$.  
  %*** Corollary \ref{***}.(2) classifies simple $\CB$-modules and  describes their dimensions and endomorphism rings. ***  
 
\begin{corollary}\label{c28Jul23}%\marginpar{c28Jul23}

\begin{enumerate}
\item  $\Prim (A_{1,trs} )=\Max (A_{1,trs} ) = \{ A_{1,trs} \gm\,|\, \gm \in \Max (Z(A_{1,trs} )) \}$.

\item For each maximal ideal $\gm \in \Prim (A_{1,trs} )=\Max (A_{1,trs} )$, the factor algebra 
 $$A_{1,trs} / A_{1,trs} \gm =Q(A_{1,trs} /A_{1,trs} \gm)=\bigoplus_{i,j=0}^{n-1}k(\gm )h^ix^j\simeq k(\gm )\t_{Z(A_{1,trs} )}A_{1,trs},$$ is a central simple $n^2$-dimensional algebra over the field $k(\gm )=Z(A_{1,trs} )/\gm$. The algebra $A_{1,trs} /A_{1,trs} \gm$ is isomorphic to the algebra $\CE (\gm )=(E(s,\gm), \s, a)$ where $E(s,\gm):=k(\gm )[h]/(h^n-s)\simeq k(\gm )[h,h^{-1}]/(h^n-s)$ and  $\s (h)=qh$. Let $U(\gm):=U(\CE(\gm ))$ be a unique simple  $\CE (\gm )$-module (described in (\ref{xu=usxe}) and Theorem \ref{12Jun23}.(1,2)). 

\item $\widehat{A_{1,trs} }=\{ U(\gm)\, |\, \gm \in \Max (Z(A_{1,trs})) \}$. All simple $A_{1,trs}$-modules are finite dimensional.

\item The map $\widehat{A_{1,trs}} \ra {\rm Prim} (A_{1,trs} ) =\Max (A_{1,trs} )$, $U\mapsto \ann_{A_{1,trs}} (U)$ is a bijection with  inverse $A_{1,trs} \gm \mapsto U(\gm)$.

\end{enumerate}
\end{corollary}

\begin{proof}   The corollary follows at once from 
the fact that $A_{1,trs}\simeq \CB_{s-\frac{1}{(q-1)^n}}$ and Corollary \ref{aqAc28Aug20}.(1--4).    \end{proof}

 Corollary  \ref{d28Jul23} is a classification of completely prime ideals of the algebra $A_{1,trs}$.

\begin{corollary}\label{d28Jul23}%\marginpar{d28Jul23}
$\Spec_c (A_{1,trs} )=\Big\{ \{ 0\}, A_{1,trs}\gr' \, | \, 0\neq \gr' \in \Spec \,Z(A_{1,trs})$,  $E(s,\gr')=k(\gr')[h]/(h^n-s)$
is a field and $n=\min\{ d'\geq 1\, | \,  d'\mid n$  there is a matrix $X\in M_{d'}(E(s,\gr'))$ such that $ X^{\s^{n-1}}\cdots X^\s X=a\}\Big\}$ where $\s (h)=qh$.
\end{corollary}

\begin{proof}    The corollary follows at once from 
the fact that $A_{1,trs}\simeq \CB_{s-\frac{1}{(q-1)^n}}$ and Corollary \ref{A20Jul23}.    \end{proof}    

{\bf The prime spectrum of the quantum Weyl algebra $A_1$.} Theorem \ref{28Jul23} classifies the prime ideals of the algebra $A_1$ and describes the containments between primes.

Below is a proof of Theorem \ref{28Jul23}. 

\begin{proof}    We have seen above that 
$\Spec (A_1)=\Spec (A_{1,trs})\, \sqcup \, \Spec \Big( (A_1/(r))_{t}\Big) \, \sqcup \, \Spec \Big( (A_1/(t))_{r}\Big) \, \sqcup \,\Spec (A_1/(t,r)) \, \sqcup \,
\Spec (A_1/(h))$.
Since  $t,r, s\in Z(A_1)$, 
$$\Spec (A_1)=\Spec_{trs} (A_1)\, \sqcup \, \Spec_{t} (A_1/(r)) \, \sqcup \, \Spec_{r} (A_1/(t)) \, \sqcup \,\Spec (A_1/(t,r)) \, \sqcup \,
\Spec (A_1/(h)),$$  by Proposition \ref{aA12Mar15}. 

Recall that $Z(A_1)=K[t,r]$ and $Z(A_{1,trs})=Z(A_1)_{trs}=K[t,r]_{trs}$. Then, 
by Proposition \ref{B28Jul23}.(2), $\Spec(A_{1,trs})=\{A_{1,trs}\gr'\, | \, \gr'\in \Spec(A_1); t,r,s\not\in \gr')\}$. 
It follows from  Proposition \ref{aA12Mar15} and the equality $A_1=\bigoplus_{i,j=0}^{n-1}K[t,r]x^iy^j$ that $$\Spec_{trs} (A_1)=\{A_1\gr'\, | \, \gr'\in \Spec(A_1), t,r,s\not\in \gr')\}$$
since $\tor_{\S}(A_1/A_1\gr')=\tor_{\S}\bigg(\bigoplus_{i,j=0}^{n-1}\Big( K[t,r]/\gr'\Big)x^iy^j\bigg)=\{0\}$ where $\S$ is a submonoid of $A_1$ generated by the central elements $t$, $r$, and $s$. The algebra $A_1$ is a domain. Therefore, $\{ 0\}\in \Spec (A_1)$. Clearly, $\Spec_{trs} (A_1)=\{ 0\} \, \sqcup \,\CN'' \, \sqcup \, \mM''$.

By Proposition \ref{A25Jul23}.(1,2), $\Spec \Big(\Big( A_1/(r) \Big)_{t}\Big)=\{ \{ 0\}, (f)\, | \, f\in \Irr_m(K[t])\backslash \{ t\} \}$ and the factor algebra $A_1(r,f)=\Big( A_1/(r) \Big)_{t}/(f)\simeq M_n(K[t]/(f))$ is  the matrix algebra over the field $K[t]/(f)$ for every element $ f\in \Irr_m(K[t])\backslash \{ t\} $. Hence,
$$\Spec_t  (A_1/(r))=\{ (r), (r,f)\, | \,  f\in \Irr_m(K[t])\backslash \{ t\} \}\;\; {\rm (Corollary\; \ref{a29Jul23})}.$$
By Proposition \ref{A25Jul23}.(3,4), $\Spec \Big(\Big( A_1/(t) \Big)_{r}\Big)=\{ \{ 0\}, (g)\, | \, g\in \Irr_m(K[r])\backslash \{ r\} \}$ and the factor algebra $A_1(t,g)=\Big( A_1/(t) \Big)_{r}/(g)\simeq M_n(K[r]/(g))$ is  the matrix algebra over the field $K[r]/(g)$ for every element $ g\in \Irr_m(K[r])\backslash \{ r\} $. Hence,
$$\Spec_r (A_1/(t))=\{ (t), (t,g)\, | \,  g\in \Irr_m(K[r])\backslash \{ r\} \}\;\; {\rm (Corollary\; \ref{a29Jul23})}.$$
By Proposition \ref{A26Jul23}, $$\Spec (A_1/(t,r)) = \{ (t,r)\}.$$ By Lemma \ref{b28Jul23}.(1), 
$$\Spec (A_1/(h))= \{ (h), (h,f)\, | \, f\in \Irr_m(K[x]\backslash \{ x\}\}=\{ (h)\} \, \sqcup \, \mathbb{H}''.$$
Hence, the diagram (\ref{qXYZA2}) contains all the prime ideals of the algebra $A_1$. By the definitions of the sets in the diagram (\ref{qXYZA2}), we have connections as in the diagram. It remains to show that there are no additional connections. Clearly, there is no connection between $(t)$ and $\mR$ (since the ideal $(t,r)$ and the ideals in the set $\mR$ are maximal). Similarly, there is no connection between $(r)$ and $\mT$ (since the ideal $(t,r)$ and the ideals in the set $\mT$ are maximal).
Since $(t,h)=(r,h)=A_1$ (Lemma \ref{a24Jul23}.(2)), there is no connection between the following pairs:  
$(t)$ and $\mathbb{H}''$, $(r)$ and $\mathbb{H}''$, $(h)$ and $\mT$,  $(h)$ and $(t,r)$, and  $(h)$ and $\mR$.
 By the definition of the set $\mM''$, there is no connection between $\mM''$ and the elements  $(t)$, $(r)$, and $(h)$.   \end{proof}  
 
 \begin{corollary}\label{aa30Jul23}%\marginpar{aa30Jul23}
$\Max(A_1)=\{ (t,r)\}\, \sqcup\,  \mT\, \sqcup\, \mR\, \sqcup\, \mM''\,\sqcup\, \mathbb{H}''$.
\end{corollary}

 \begin{proof} The corollary follows from the diagram (\ref{qXYZA2}).
 \end{proof}

 For an algebra $A$, its {\bf classical Krull dimension} is the largest length of chains of prime ideals.  An algebra $A$ is called a {\bf catenary algebra} if for each pair of prime ideals $\gp$ and $\gq$ of $A$ such that $\gp\subseteq \gq$ the lengths of longest chains  of prime ideals 
between $\gp$ and $\gq$ are the same. A polynomial algebra $F[x_1, \ldots , x_n]$ over a field $F$ is a catenary algebra.

 \begin{corollary}\label{a30Jul23}%\marginpar{a30Jul23}
The algebra $A_1$ is a catenary algebra of classical dimension 2.
\end{corollary}

\begin{proof}   The corollary follows from Theorem \ref{28Jul23}.   \end{proof} 
 
{\bf The prime factor algebras of $A_1$ and their quotient algebras.} Corollary \ref{e28Jul23} describes the prime factor algebras of $A_1$ and their quotient algebras.
 
 \begin{corollary}\label{e28Jul23}%\marginpar{e28Jul23}
We keep the notation of Theorem \ref{28Jul23}. 
\begin{enumerate}
\item For every $A_1\gp \in \Spec_{trs} (A_1)$,
\begin{eqnarray*}
A_1/A_1\gp&=&\bigoplus_{i,j=0}^{n=1}\Big(Z(A_1)/\gp \Big)x^iy^j\\
&=&\bigoplus_{j=0}^{n-1}\Big(Z(A_1)/\gp \Big) h^j\oplus\bigoplus_{i=1}^{n-1}\Bigg(\bigoplus_{j=1}^{n-1}\Big(Z(A_1)/\gp \Big) h^jx^i\oplus \bigoplus_{j=1}^{n-1}\Big(Z(A_1)/\gp \Big) h^jy^{n-i}\Bigg),\\
Q(A_1/A_1\gp)&=& Q(A_{1,trs}/A_{1,trs}\gp )=\bigoplus_{i,j=0}^{n-1}k(\gp )h^ix^j
\end{eqnarray*}
 is a central simple $n^2$-dimensional algebra over the field $k(\gp )$ of fractions of the commutative domain $Z(A_1 )/\gp$. The algebra $Q(A_1/A_1\gp)$ is isomorphic to the algebra $\CE (\gp):=(E(s,\gp), \s, a)$ where $E(s,\gp):=k(\gp )[h]/(h^n-s)\simeq k(\gp )[h,h^{-1}]/(h^n-s)$ and  $\s (h)=qh$. In particular,
 the quotient algebra  of the algebra $A_1$, $$Q(A_1 )=\bigoplus_{i,j=0}^{n-1}K(s,t)h^ix^j$$ is a central simple $n^2$-dimensional division algebra over the field $K(s,t)$ of rational functions in two variables. The algebra $Q(A_1 )$  is isomorphic to the cyclic algebra $\CE (0):=(E(s,0), \s, a)$ where $E(s,0):=K(s,t)[h]/(h^n-s)$ and  $\s (h)=qh$.

\item $A_1/(t)=\bigoplus_{i=1}^{n-1}\Big(\bigoplus_{j=0}^{i} \CK e_j\Big)x^i\overline{1}\oplus \CK [y]\overline{1}$ where $\overline{1}:=1+(t)$ and $Q(A_1/(t))=\bigoplus_{i,j=0}^{n-1}K(r)E_{ij}=M_n(K(r))$
where $E_{ij}:=\varepsilon_iy^{i-j}\varepsilon_j$ are the matrix units.

\item $A_1/(r)=\bigoplus_{i=1}^{n-1}\Big(\bigoplus_{j=i}^{n-1} \CK e_j\Big)y^i\tilde{1}\oplus \CK [x]\tilde{1}$ where $\tilde{1}:=1+(r)$ and $Q(A_1/(r))=\bigoplus_{i,j=0}^{n-1}K(t)E_{ij}=M_n(K(t))$.
 where $E_{ij}:=\varepsilon_ix^{j-i}\varepsilon_j$ are the matrix units.

\item $A_1/(h)\simeq K[x^{\pm 1}]$ and  $Q(A_1/(h))\simeq K(x)$.

\item For every $(t,g)\in \mT$, $A_1/(t,g)=Q(A_1/(t,g))=\bigoplus_{i,j=0}^{n-1}\Big(K[r]/(g)\Big) E_{ij}= M_n(K[r]/(g))$ is the matrix algebra over the field  $K[r]/(g)$   where $E_{ij}=\varepsilon_iy^{i-j}\varepsilon_j$. 

\item For every $(r,f)\in \mR$, $A_1/(r,f)=Q(A_1/(r,f))=\bigoplus_{i,j=0}^{n-1}\Big(K[t]/(f)\Big) E_{ij}=  M_n(K[t]/(f))$ is the matrix algebra over the field  $K[t]/(f)$  where $E_{ij}=\varepsilon_ix^{j-i}\varepsilon_j$.

\item For every $(h,f)\in \mathbb{H}''$, $A_1/(h,f)=Q(A_1/(h,f))\simeq K[x]/(f)$ is a finite field extension of $K$. 

\item $A_1/(t,r)\simeq M_n(K)$ and $L=A_1/A_1(t,y)$ is a unique simple $A_1/((t,r)$-module.

\end{enumerate}

\end{corollary}

\begin{proof}    1. By Theorem \ref{28Jul23}.(1), $\Spec_{trs} (A_1) = \{   A_1\gp \, | \, \gp \in \Spec(Z(A_1) ); t,r,s \not\in \gp\}$. Hence, the first two equalities of statement 1 follow from (\ref{qA1Sum}) and (\ref{qA1Sum3}). Hence, $t,r,s \in \CC_{A_1/A_1\gp}$ since $t,r,s \not\in \gp$. Therefore,
$$Q(A_1/A_1\gp)= Q(A_{1,trs}/A_{1,trs})$$
and statement 1 follows from Proposition \ref{B28Jul23}.(4).

2. By Proposition \ref{A25Jul23}.(3), 
$\Big( A_1/(t) \Big)_{r}=\bigoplus_{i,j=0}^{n-1}K[r^{\pm 1}]E_{ij}\simeq M_n(K[r^{\pm 1}])$
where $E_{ij}:=\varepsilon_iy^{i-j}\varepsilon_j$ are the matrix units. By Proposition \ref{A29Jul23}.(2), $r+(t)\in \CC_{A_1/(t)}$.
Therefore, $$Q( A_1/(t))=Q \Big(\Big( A_1/(t) \Big)_{r}\Big)=Q(M_n(K[r^{\pm 1}]))=Q(M_n(K(r))).$$

3. By Proposition \ref{A25Jul23}.(1), $\Big( A_1/(r) \Big)_{t}=\bigoplus_{i,j=0}^{n-1}K[t^{\pm 1}]E_{ij}\simeq M_n(K[t^{\pm 1}])$ where $E_{ij}:=\varepsilon_ix^{j-i}\varepsilon_j$ are the matrix units. By Proposition \ref{A29Jul23}.(1), $t+(r)\in \CC_{A_1/(r)}$. Therefore, $$Q( A_1/(r))=Q \Big(\Big( A_1/(r) \Big)_{t}\Big)=Q(M_n(K[t^{\pm 1}]))=Q(M_n(K(t))).$$

4. By (\ref{qA1Sum11}), $A_1/(h)\simeq K[x^{\pm 1}]$, and statement 4 follows.

5. By Proposition \ref{A25Jul23}.(3,4), $$A_1/(t,g)=\Big( A_1/(t)\Big)_r/(g)=\bigoplus_{i,j=0}^{n-1}\Big(K[r]/(g)\Big) E_{ij}=M_n(K[r]/(g))$$ where $E_{ij}=\varepsilon_iy^{i-j}\varepsilon_j$.
 Hence, $Q(A_1/(t,g))=Q(M_n(K[r]/(g)))=M_n(K[r]/(g))$.

6. By Proposition \ref{A25Jul23}.(1,2), $$A_1/(r,f)=\Big( A_1/(r)\Big)_t/(f)=\bigoplus_{i,j=0}^{n-1}\Big(K[t]/(f)\Big) E_{ij}=M_n(K[t]/(f))$$ where $E_{ij}=\varepsilon_ix^{j-i}\varepsilon_j$.
 Hence, $Q(A_1/(r,f))=Q(M_n(K[t]/(f)))=M_n(K[t]/(f))$.

7. Statement 7 follows from statement 4.   

8. Statement 8 is Proposition \ref{A26Jul23}.  \end{proof}

{\bf Classifications of simple $A_1$-modules and  primitive ideals of $A_1$.} 

\begin{proposition}\label{30Jul23}%\marginpar{30Jul23}
The quantum Weyl algebra $A_1$ satisfies the assumptions of Corollary \ref{b2Aug23}, ${\rm ZL}(A_1)=\Spec (Z(A_1))$, and Theorem \ref{2Aug23} holds:
 \begin{enumerate}
\item The ring $A_1$ is PLM-ring with ${\rm PL}(A_1)=\Max (Z(A_1))$.

\item  For every $\gm \in \Max (Z(A_1))$, $\Prim (A_1,\gm)=\{ A_1\gm\}$.

\item $\Prim(A_1)=\Max (A_1)=\{ A_1\gm\, | \, \gm\in \Max (Z(A_1))\}=\{ (t,r)\}\, \sqcup\,  \mT\, \sqcup\, \mR\, \sqcup\, \mM''\,\sqcup\, \mathbb{H}''$. 

\item For every $\gp \in \mM''$,
 $A_1/A_1\gp =
Q(A_1/A_1\gp)=\bigoplus_{i,j=0}^{n-1}k(\gp )h^ix^j
$
 is a central simple $n^2$-dimensional algebra over the field $k(\gp )=Z(A_1 )/\gp$. The algebra $A_1/A_1\gp$ is isomorphic to the algebra $\CE (\gp):=(E(s,\gp), \s, a)$ where $E(s,\gp):=k(\gp )[h]/(h^n-s)\simeq k(\gp )[h,h^{-1}]/(h^n-s)$ and  $\s (h)=qh$.  Let $U(\gm)$ be a unique simple  $\CE (\gm )$-module (described in (\ref{xu=usxe}) and Theorem \ref{12Jun23}.(1,2))). 

\item $\widehat{A_1} =\{ U(\gm )\, | \, \gm \in \Max(Z(A_1))\}$  where $U(\gm)$ is a unique simple module of the simple artinian ring $k(\gm)\t_{Z(A_1)}A_1/A_1\gm$
$$
U(\gm):=\begin{cases}
A_1/A_1(t,y)& \text{if }\gm=(t,r),\\
A_1/(t,g)\varepsilon_0& \text{if }\gm=(t,g)\in \mT,\\
A_1/(r,f)\varepsilon_0& \text{if }\gm=(r,f)\in \mR,\\
 & \text{ is described in (\ref{xu=usxe}) and Theorem \ref{12Jun23}.(1,2) if $\gm \in \mM''$)}\\
 A_1/(h,f)\varepsilon_0& \text{if }\gm=(h,f)\in \mathbb{H}''.\\
\end{cases}
$$
All simple $\mA$-modules are finite dimensional.

\item The map $\widehat{A_1} \ra \Prim (A_1)$, $U(\gm )\mapsto \ann_{A_1}(U(\gm))=A_1/A_1\gm$ is a bijection with inverse $A_1\gm\mapsto U(\gm)$.

\item For $\gm \in \Max (Z(A_1))$,
$$
\End_{A_1}(U(\gm)):=\begin{cases}
K& \text{if }\gm=(t,r),\\
K[r]/(g)& \text{if }\gm=(t,g)\in \mT,\\
K[t]/(f)& \text{if }\gm=(r,f)\in \mR,\\
D(\gm)  & \text{ is described in Theorem \ref{12Jun23}.(1,2) if $\gm \in \mM''$)},\\
 K[x]/(f)& \text{if }\gm=(h,f)\in \mathbb{H}''.\\
\end{cases}
$$

\end{enumerate}

\end{proposition}

\begin{proof} By Corollary \ref{e28Jul23}, the  algebra $A_1$ satisfies the assumptions of Corollary \ref{b2Aug23}, and so  ${\rm ZL}(A_1)=\Spec (Z(A_1))$, and Theorem \ref{2Aug23} holds apart from statement 4, 7, and explicit description of simple modules in statement 5. Statement 4 follows from Corollary \ref{e28Jul23}.(1). The explicit description of simple modules in statement 5 follows from Corollary \ref{e28Jul23}.(5--8). Statement 7 follows from statement 5 and Corollary \ref{e28Jul23}.(1, 5--8).
\end{proof}

{\bf The set of completely prime ideals of the algebra $A_1$.}  Corollary  \ref{b30Jul23} is a classification of completely prime ideals of the algebra $A_1$.

\begin{corollary}\label{b30Jul23}%\marginpar{b30Jul23}
$\Spec_c (A_1)=\Big\{ \{ 0\}, (h), (h,f), A_1\gp \, | \, f\in \Irr_m(K[x])\backslash \{ x\},  0\neq \gp \in \Spec \,Z(A_1)$; $t,r,s\not\in \gp$,  $E(s,\gp)=k(\gp)[h]/(h^n-s)$
is a field and $n=\min\{ d'\geq 1\, | \,  d'\mid n$  there is a matrix $X\in M_{d'}(E(s,\gp))$ such that $ X^{\s^{n-1}}\cdots X^\s X=a\}\Big\}$ where $k(\gp)$ is the field of fractions of the commutative domain $Z(A_1)/\gp$ and $\s (h)=qh$.

\end{corollary}

\begin{proof}   A prime ideal $P$ of the algebra $A_1$ is a completely prime ideal iff the quotient algebra $Q(A_1/P)$ is a division algebra. 
 A prime ideal $P\in \Spec_{trs}(A_1)$ is a completely prime ideal  of  $A_1$ iff $P\in \Spec_c(A_{1,trs})$.
By Proposition \ref{aA12Mar15}.(2) and Theorem \ref{28Jul23}, the map  $$\Spec_{trs}(A_1)\ra \Spec(A_{1,trs}), \;\;A_1\gp\mapsto A_{1,trs}\gp$$ is a bijection. Now, the corollary follows from   Corollary \ref{d28Jul23}  and Corollary \ref{e28Jul23}.   \end{proof}

%%%%%%%%%%%%%%%%%%%%%% Section 6  %%%%%%%%%%%%%%%%%%

\section{Automorphism groups of the  algebras $\mA$, $A_1$, $\CA$, and $\CB$}\label{AUTGROUPS}%\marginpar{AUTGROUPS}

The aim of the section is to give explicit descriptions of the automorphism groups of the algebras  $\mA$, $A_1$, $\CA$, and $\CB$. Recall that $q$ is an $n$'th primitive root of unity. \\

{\bf The automorphism group $\Aut_K(\mA )$.}
  There are obvious  subgroups of $\Aut_K(\mA)$:
\begin{eqnarray*}
 \mI&:=& \{1, \iota \}, \;\; \iota:  h\mapsto x, \;\; x\mapsto  h, \;\; {\rm if}\;\; n=2,\\
\mT^2 &:=& \{ t_{\l,\mu}\, | \, \l, \mu\in K^\times\}, \;\; t_{\l , \mu}:  h\mapsto \l h, \;\; x\mapsto \mu x.
\end{eqnarray*}
The group $\mI$ has order 2 since $\iota^2=1$. The group $\mT$ is a  {\em 2-dimensional algebraic torus} since $\mT^2\simeq \Big(K^\times\Big)^2$. Let $G$ be a group and $N$, $H$ be its subgroups such that $N$ is a normal subgroup. If each element $g\in G$ is a unique product $d=hn$ for some elements $h\in H$ and $a\in N$ then the group is denote $H\ltimes N$ and is  called a {\em semidirect product} of groups.

\begin{theorem}\label{Aut(mA)}%\marginpar{Aut(mA)}
$\Aut_K(\mA )=\begin{cases}
\mT^2\ltimes \mI& \text{if }n=2,\\
\mT^2& \text{if }n\neq 2,\\
\end{cases}
=\begin{cases}
\{ t_{\l,\mu}, t_{\l,\mu}'\, |\, \l,\mu \in K^\times\}& \text{if }n=2,\\
\{ t_{\l,\mu}\, |\, \l\in K^\times\}& \text{if }n\neq 2,\\
\end{cases}$
 where $t_{\l, \mu}': h\mapsto \l x$, $x\mapsto \mu h$.
\end{theorem}

\begin{proof} The second equality of the theorem is obvious. Clearly, the group $\Aut_K(\mA )$ contains the semidirect product $\mT^2\ltimes \mI$ if $n=2$,  and $\mT^2$ if $n\neq 2$. Let us show that the reverse inclusions hold. 

Let $\tau \in \Aut_K(\mA)$. By (\ref{qXYZA}), the  ideals $(h)$ and $(x)$ are the only prime ideals of the algebra $\mA$ of height 1 that are generated by a  normal non central element of $\mA$. Therefore, there two cases either $\tau (h)=\l h$, $\tau (x)=\mu x$ or otherwise $\tau (h)=\l x$, $\tau (x)=\mu h$ for some units $\l , \mu \in \mA^\times = K^\times$ of the algebra $\mA$. 
So, in the first case, $\tau \in \mT^2$. In the second case, up to  multiplication by an element of $\mT^2$, we may assume that  $\l = \mu =1$, i.e. $\tau =\iota\in \mI$, and the theorem follows. 
\end{proof}

{\bf The automorphism group $\Aut_K(A_1)$.}   There are obvious  subgroups of $\Aut_K(A_1)$:
\begin{eqnarray*}
 \mJ&:=& \{1, \zeta \}, \;\; \zeta:  x\mapsto y, \;\; y\mapsto  x, \;\; {\rm if}\;\; n=2,\\
\mT^1 &:=& \{ t_{\l}\, | \, \l\in K^\times\}, \;\; t_{\l}:  x\mapsto \l x, \;\; y\mapsto \l^{-1} y.
\end{eqnarray*}
Notice that the conditions that $n=2$ and $q=-1$ is a primitive 2nd root of unity imply that $\char (K)\neq 2$. The group $\mJ$ has order 2 since $\zeta^2=1$. The group $\mT^1$ is a  {\em 1-dimensional algebraic torus} since $\mT^1\simeq K^\times $.

\begin{theorem}\label{Aut(A1(q))}%\marginpar{Aut(A1(q))}
$\Aut_K(A_1 )=\begin{cases}
\mT^1\ltimes \mJ& \text{if }n=2,\\
\mT^1& \text{if }n\neq 2,\\
\end{cases}
=\begin{cases}
\{ t_{\l}, t_{\l}'\, |\, \l \in K^\times\}& \text{if }n=2,\\
\{ t_{\l}\, |\, \l\in K^\times\}& \text{if }n\neq 2,\\
\end{cases}$
 where $t_{\l}': x\mapsto \l y$, $y\mapsto \l^{-1} x$. For all $\tau \in \mT^1$, $\tau (h)=h$ but $\zeta (h)=-h$ if $n=2$. 
\end{theorem}

\begin{proof} The second equality of the theorem is obvious. Clearly, the group $\Aut_K(A_1 )$ contains the semidirect product $\mT^1\ltimes \mJ$ if $n=2$,  and $\mT^1$ if $n\neq 2$. Let us show that the reverse inclusions hold. 

Let $\tau \in \Aut_K(A_1)$. By (\ref{qXYZA2}),
 the  ideals $(x)$ and $(y)$ are the only prime ideals of the algebra $\mA$ of height 1 that are generated by a non normal non central element of $\mA$. Therefore, there two cases either $\tau (x)=\l x$, $\tau (y)=\mu y$ or otherwise $\tau (x)=\l y$, $\tau (y)=\mu x$ for some units $\l , \mu \in A_1^\times = K^\times$ of the algebra $A_1$. By applying the automorphism $\tau $ that belongs to the first case to the defining relation $xy-qyx=1$ of the algebra $A_1$, we deduce that $\mu=\l^{-1}$.  So, in the first case, $\tau \in \mT^1$. 
 
 In the second case, up to  multiplication by an element of $\mT^1$, we may assume that  $\l=1$. Then applying $\tau$ to the defining relation of the algebra $A_1$ we deduce that $\mu =1$ and $n=2$,  i.e. $\tau =\zeta\in \mJ$, and the reverse inclusions hold. 
 
 It follows from the definition of the element  $h=yx+\frac{1}{q-1}$ (see (\ref{qA1Sum1})) that  $\tau (h)=h$  for all $\tau \in \mT^1$ but $\zeta (h)=-h$ if $n=2$.
\end{proof}

{\bf The automorphism group $\Aut_K(\CA)$.}   There are obvious  subgroups of $\Aut_K(\CA)$:
\begin{eqnarray*}
 \mK &:=& \{1, \kappa \}, \;\; \kappa : h\mapsto  h,\;\;  x\mapsto x^{-1},  \;\; {\rm if}\;\; n=2,\\ 
  \mU &:=& \{ \s_{\l x^i}\, | \, \l \in K^\times, i\in \Z\}, \;\; \s_{\l x^i}: h\mapsto \l x^ih, \;\; x\mapsto x,\\
  \mT^1_h &:=& \{ t_{1, \mu}\, | \,  \mu\in K^\times\}, \;\; t_{1,\mu}:  h\mapsto  h, \;\; x\mapsto \mu x,\\
   \mT^1_x &:=& \{ t_{\l, 1}\, | \,  \l\in K^\times\}, \;\; t_{\l,1}:  h\mapsto  \l h, \;\; x\mapsto  x,\\
 \Xi &:=&\{ \xi_i\, | \, i\in \Z\}, \;\; \xi_i: h\mapsto x^ih, \;\; x\mapsto x, \\
\mT^2 &:=& \{ t_{\l,\mu}\, | \, \l, \mu\in K^\times\}, \;\; t_{\l , \mu}:  h\mapsto \l h, \;\; x\mapsto \mu x.
\end{eqnarray*}
The group $\mK$ has order 2 since $\kappa^2=1$.  The group $\mU$ is isomorphic to the group of units $\CA^\times=\{ \l x^i\, | \, \l \in K^\times, i\in \Z\}$ of the algebra $\CA$ via $\s_{\l x^i}\mapsto \l x^i$. In particular, the group $\mU$ is an abelian group.  The group $\Xi$ is isomorphic to $\Z$ via $\xi_i\mapsto i$ and  $\mT^2=\mT^1_h\times \mT^1_x$ is a direct product of groups.  The group $\Xi$ is isomorphic to the group of inner automorphisms $\Inn (\CA)\simeq \mA^\times /K^\times$ of the algebra $\CA$ via $\xi_i\mapsto \o_{K^\times x^i}$. In particular, the group $\Xi$ is a normal subgroup of $\Aut_K(\CA)$. Notice that $\mU =\Xi\times \mT^1_x$.

\begin{theorem}\label{Aut(CA)}%\marginpar{Aut(CA)}
$\Aut_K(\CA )=
\Big( \mU\ltimes \mT^1_h\Big)  \ltimes \mK =   
\{ \s_{\l x^i,\mu}, \s_{\l x^i, \mu}'\, |\, \l, \mu \in K^\times, i\in \Z\}$
 where $\s_{\l x^i,\mu}: h\mapsto \l x^i h$, $x\mapsto \mu x$ and  $\s_{\l x^i,\mu}': h\mapsto \l x^i h$, $x\mapsto \mu x^{-1}$.
\end{theorem}

\begin{proof} The second equality of the theorem is obvious. Clearly, the group $\Aut_K(\CA )$ contains the iterated semidirect product $\Big( \mU\ltimes \mT^1_h\Big)  \ltimes \mK$. Let us show that the reverse inclusion hold. 

Let $\tau \in \Aut_K(\CA)$. By (\ref{CAqXYZA}),
 the  ideal $(h)$ is the only prime ideals of the algebra $\CA$ of height 1 that is generated by a  normal non central element of $\CA$. Therefore, $\tau (h)=uh$ for some unit $u \in \CA^\times =\{ \l x^i\, | \, \l \in K^\times , i\in \Z\}$, i.e. $u=\l x^i$ for some $\l \in K^\times$ and $i\in \Z$. The group of inner automorphisms 
$\Inn (\CA)$ of the algebra $\CA$ is isomorphic to the factor group $\mA^\times /K^\times\simeq \Z$ that has only two group generators: $K^\times x$ or $K^\times x^{-1}$. Therefore, there two cases either $\tau (x)=\mu x$ or otherwise $\tau (x)=\mu x^{-1}$ for some $\mu \in K^\times$. 
Hence, either $\tau =\s_{\l x^i, \mu} $ or $\tau =\s_{\l x^i, \mu}' $, and the reverse inclusion follows.
\end{proof}

\begin{corollary}\label{aAut(CA)}%\marginpar{aAut(CA)}
\begin{enumerate}
\item 
$\Aut_K(\CA)=\Big(\Xi\ltimes \mT^2\Big)\ltimes \mK \simeq \Big(\Inn (\CA)\ltimes \mT^2\Big)\ltimes \mK$.
\item The group of outer isomorphisms $\Out (\CA):=\Aut_K(\CA)/\Inn (\CA)$ is isomorphic to the group $\mT^2\ltimes \mK $.
\end{enumerate}
\end{corollary}

\begin{proof} 1. Recall that $\mU =\Xi\times \mT^1_x$,   $\mT^2=\mT^1_h\times \mT^1_x$, and  $\Inn (\CA)\simeq \mA^\times /K^\times\simeq \Xi$. 
 By Theorem \ref{Aut(CA)}),
 $$\Aut_K(\CA )=
\Big( \mU\ltimes \mT^1_h\Big)  \ltimes \mK = \Big( \Big( \Xi \times \mT^1_x\Big) \ltimes \mT^1_h\Big)  \ltimes \mK =\Big(\Xi\ltimes \mT^2\Big)\ltimes \mK \simeq \Big(\Inn (\CA)\ltimes \mT^2\Big)\ltimes \mK.$$
2. Statement 2 follows from statement 1.
\end{proof}

{\bf The automorphism group $\Aut_K(\CB)$.}  
The group of units $\CB^\times$of the algebra $\CB$ is equal to the set $\{ \l h^ax^b\, | \, \l \in K^\times, a,b\in \Z\}$. The multiplication in the group $\CB^\times$ is given by the rule: For all elements $\l , \mu \in K^\times$ and $a,b,c,d\in \Z$,
$$ \l h^ax^b\cdot \mu h^cx^d= \l \mu q^{bc} h^{a+c}x^{b+d}.$$
The group $\CB^\times$ is not an abelian group. 
It contains the group $K^\times$ that belongs to the centre of the group $\CB^\times$. So, $K^\times$ is a normal subgroup of $\CB^\times$ such that the factor group $\CB^\times / K^\times$ is isomorphic  to the abelian group $\Z^2$. Let $\Aut_{{\rm gr}, K^\times}(\CB^\times)$ be the  subgroup of the automorphism group of the group $\CB^\times$ such that every element of which fixes elements of $K^\times$ ($\g (\l) =\l$ for all $\l \in K^\times$). Clearly,
$$
\Aut_{{\rm gr}, K^\times}(\CB^\times)=
\{ \tau_{A,\l,\mu}\, |\, A\in \SL_2(\Z),  \l, \mu \in K^\times\}\;\;  {\rm where}\;\;\tau_{A,\l,\mu}: h\mapsto \l h^ax^b,\;\; x\mapsto \mu h^cx^
d$$
%$$
%\Aut_{\rm gr}(\CB^\times)=
%\begin{cases}
%\{ \tau_{A,\l,\mu}\, |\,A\in \GL_2(\Z),  \l, \mu %\in K^\times\}& \text{if }n=2,\\
%\{ \tau_{A,\l,\mu}\, |\,A\in \SL_2(\Z),  \l, \mu %\in K^\times\}& \text{if }n\neq 2,\\
%\end{cases}
%$$
 and  $A=\begin{pmatrix}
a & b \\  c & d
 \end{pmatrix}
$. In more detail, let $\tau \in \Aut_{{\rm gr}, K^\times}(\CB^\times)$. The automorphism $\tau$
induces a group automorphism of the factor group
$\CB^\times /K^\times\simeq \Z^2$. Then necessarily $\tau = \tau_{A,\l,\mu}$ for some $A\in \GL_2(\Z)$
 and $\l,\mu \in K^\times$ such that $\tau$ respects the defining relation $xh=qhx$, i.e. $\tau (x)\tau (h)=q\tau (h)\tau (x)$ or equivalently $q^{da}h^{a+c}x^{d+b}=q\cdot q^{cb}h^{a+c}x^{d+b}$, i.e. $$q^{da-bc-1}=1.$$ Hence, $A\in \SL_2(\Z)$.

\begin{theorem}\label{Aut(CB)}%\marginpar{Aut(CB)}
$\Aut_K(\CB )=\Aut_{{\rm gr}, K^\times}(\CB^\times) =
\{ \tau_{A,\l,\mu}\, |\,A\in \SL_2(\Z),  \l, \mu \in K^\times\}
$
 where $\tau_{A,\l,\mu}: h\mapsto \l h^ax^b $, $x\mapsto \mu h^cx^d$. 
\end{theorem}

\begin{proof}  The second equality of the theorem has been proven above. Clearly, $\Aut_K(\CB )\supseteq \Aut_{{\rm gr}, K^\times}(\CB^\times) $, as we checked already that every  $\tau_{A,\l,\mu}$ respects the defining relation $xh=qhx$ of the algebra $\CB$. Since every automorphism of the algebra $\CB$ induces a group automorphism of $\CB^\times$, we must have $\Aut_K(\CB )= \Aut_{{\rm gr}, K^\times}(\CB^\times)$. 
\end{proof}
The group $\Aut_K(\CB)$  contains the subgroup 
$$\mT^2 = \{ t_{\l,\mu}=\tau_{E,\l, \mu}\, | \, \l, \mu\in K^\times\}, \;\; t_{\l , \mu}:  h\mapsto \l h, \;\; x\mapsto \mu x.$$
Notice that the subset $\{ \tau_{A,1,1}\, |\,A\in \SL_2(\Z)\}$ is  not subgroups of $\Aut_K(\CB)$.

Recall that the group of inner automorphisms 
$\Inn (\CB)\simeq \CB^\times /K^\times$ of the algebra $\CB$ is isomorphic to $\Z^2$ via $K^\times h\mapsto (1,0), \;\; K^\times x\mapsto (0,1)$.  So, $\Inn (\CB)=\langle \o_h\rangle \times 
\langle \o_x\rangle =\{t_{q^i, q^j}\, | \, i,j\in \Z\} \subseteq \mT^2$. The set  $\O :=\{(q^i,q^j)\, | \, (i,j)\in \Z^2\}$ is a subgroup of $\Big(K^\times\Big)^2$.

\begin{corollary}\label{aAut(CB)}%\marginpar{aAut(CB)}
 $\Out (\CB)=\Aut_{\rm gr}(\CB ^\times)/\Inn (\CB) =
\{ \tau_{A,\l,\mu}\, |\,A\in \SL_2(\Z), (\l, \mu)\in \Big(K^\times\Big)^2/\O\}.
$
 \end{corollary}

\begin{proof} The corollary follows from Theorem \ref{Aut(CB)}.
\end{proof}

%$${\bf Acnowledgements} $$

%The author would like to thank the Reviewer for the comments  and the Royal Society for support.

{\bf Licence.} For the purpose of open access, the author has applied a Creative Commons Attribution (CC BY) licence to any Author Accepted Manuscript version arising from this submission.

{\bf Disclosure statement.} No potential conflict of interest was reported by the author.

{\bf Data availability statement.} Data sharing not applicable – no new data generated.

\small{

School of Mathematics and Statistics

University of Sheffield

Hicks Building

Sheffield S3 7RH

UK

email: v.bavula@sheffield.ac.uk}

\end{document}